\documentclass[a4paper,10pt]{article}
\usepackage[utf8]{inputenc}

\usepackage{wrapfig}
\usepackage{amsmath} 
\usepackage{amsthm} 
\usepackage{amssymb} 
\usepackage{enumerate} 
\usepackage{esint} 
\usepackage{pgf,tikz} 
\usetikzlibrary{arrows} 
 \usepackage{yfonts} 
 \usepackage{mathrsfs} 
 \usepackage{mathabx} 
 \usepackage{graphicx}
\usepackage{caption}
\usepackage{subcaption}
\usepackage{mathtools} 
\usepackage[titletoc,toc]{appendix} 

\definecolor{ffffff}{rgb}{1.0,1.0,1.0}
\definecolor{qqqqff}{rgb}{0.0,0.0,1.0}
\definecolor{ffqqqq}{rgb}{1.0,0.0,0.0}
\definecolor{zzzzqq}{rgb}{0.6,0.6,0.0}
\definecolor{marronet}{rgb}{0.6,0.2,0}
\definecolor{negre}{rgb}{0,0,0}
\definecolor{vermell}{rgb}{0.8,0.05,0.05}
\definecolor{blau}{rgb}{0.3,0.2,1.}
\definecolor{blauclar}{rgb}{0.,0.,1.}
\definecolor{grisfosc}{rgb}{0.5,0.5,0.5}
\definecolor{verd}{rgb}{0.05,0.7,0.05}
\definecolor{taronja}{rgb}{0.9,0.5,0.05}
\definecolor{vermellclar}{rgb}{1.,0.,0.}
\definecolor{verdet}{rgb}{0,0.8,0.1}
\definecolor{blauverd}{rgb}{0,0.4,0.2}
\definecolor{grisclar}{rgb}{0.6274509803921569,0.6274509803921569,0.6274509803921569}
\definecolor{cqcqcq}{rgb}{0.7529411764705882,0.7529411764705882,0.7529411764705882}
\definecolor{aqaqaq}{rgb}{0.6274509803921569,0.6274509803921569,0.6274509803921569}
\definecolor{xdxdff}{rgb}{0.49019607843137253,0.49019607843137253,1.}
\definecolor{uuuuuu}{rgb}{0.26666666666666666,0.26666666666666666,0.26666666666666666}

\newcommand*\circled[1]{\tikz[baseline=(char.base)]{
            \node[shape=circle,draw,inner sep=2pt] (char) {#1};}}

\newcommand*\squared[1]{\tikz[baseline=(char.base)]{
            \node[shape=rectangle,draw,inner sep=2.4pt] (char) {#1}; \node[shape=rectangle,draw,inner sep=1pt] (char) {#1};}}
            
\newcommand*\squaredGreek[1]{\tikz[baseline=(char.base)]{
            \node[shape=rectangle,draw,inner sep=2.4pt] (char) {$#1$}; \node[shape=rectangle,draw,inner sep=1pt] (char) {$#1$};}}

\newcommand{\R}{{\mathbb R}}       
\newcommand{\N}{{\mathbb N}}       
\newcommand{\Z}{{\mathbb Z}}       

\newcommand{\diam}{{\rm diam}}
\newcommand{\dist}{{\rm dist}}
\newcommand{\Dist}{{\rm D}}
\newcommand{\Sh}{{\mathbf {Sh}}} 
\newcommand{\SH}{{\mathbf {SH}}} 

\newcommand{\rf}[1]{{(\ref{#1})}}
\newcommand{\norm}[1]{{\left\| {#1} \right\|}}

\newtheorem{theorem}{Theorem}
\newtheorem*{theorem*}{Theorem}
\newtheorem*{conjecture*}{Conjecture}
\newtheorem{lemma}[theorem]{Lemma}

\newtheorem*{corollary*}{Corollary}
\newtheorem{proposition}[theorem]{Proposition}

\newtheorem{definition}[theorem]{Definition}

\newtheorem{remark}[theorem]{Remark}

\numberwithin{subsection}{section}
\numberwithin{theorem}{section}
\numberwithin{equation}{section}
\numberwithin{figure}{section}

\usepackage[affil-it]{authblk}

\usepackage{xcolor}

\title{Triebel-Lizorkin regularity and bi-Lipschitz maps: composition operator and inverse function regularity}

\author{ Mart\'i Prats
\thanks{MP (De\-par\-ta\-ment de Ma\-te\-m\`a\-ti\-ques, U\-ni\-ver\-si\-tat Au\-t\`o\-no\-ma de Bar\-ce\-lo\-na, Catalonia): \texttt{marti.prats@uab.cat}} }

\begin{document}
\maketitle
\bibliographystyle{alpha}

\begin{abstract} 
We study the stability of Triebel-Lizorkin regularity of bounded functions and Lipschitz functions under bi-Lipschitz changes of variables and the regularity of the inverse function of a Triebel-Lizorkin bi-Lipschitz map in Lipschitz domains. To obtain the results we provide an equivalent norm for the Triebel-Lizorkin spaces with fractional smoothness  in uniform domains in terms of the first-order difference of the last weak derivative available averaged on balls. 
\end{abstract}

\section{Introduction}

Let $\Omega_1,\Omega_2$ be Lipschitz domains in $\R^d$ and let $f:\Omega_1\to \Omega_2$ be a bi-Lipschitz homeomorphism belonging to the non-homogeneous Triebel-Lizorkin space $F^s_{p,q}(\Omega_j)$, where $\norm{f}_{F^s_{p,q}(\Omega_j)}:=\inf_{g|_{\Omega_j}\equiv f}\norm{g}_{F^s_{p,q}(\R^d)}$. In this paper we give sufficient conditions for $f^{-1}$ to be in $F^s_{p,q}(\Omega_2)$ and conditions to ensure that the composition operator $T_f: g\mapsto g\circ f$ maps the function space $F^s_{p,q}(\Omega_{2})$ into $F^s_{p,q}(\Omega_{1})$. 

As it turns out, if $(s-1)p>d$ and $f\in F^s_{p,q}(\Omega_1)$, then the inverse function has the same regularity, and the composition operator map leaves the Triebel-Lizorkin regularity invariant as well. If $(s-1)p\leq d$, with $s>1$, a positive answer is also provided but we have to substitute  $F^s_{p,q}$ by $\mathbf{F}^s_{p,q}={F}^s_{p,q}\cap C^{0,1}$. The reason for this to happen is that the chain rule involves products of the derivatives of two mappings, so we require an algebra structure for the function spaces to grant that the indices remain invariant.

The precise result is the following:
\begin{theorem}\label{theoTriebelAdmissibleBanachBis}
Let $0<s<\infty$, $s\notin \N$, let $ 1\leq p < \infty$,  $ 1\leq q \leq \infty$ and $d\in\N$. Given bounded Lipschitz domains $\Omega_j\subset \R^d$ and a bi-Lipschitz function $f$ with $f(\Omega_1)= \Omega_2$, then 
\begin{align}\label{eqInvarianceUnderbiLipschitzTRIEBELBis}
f \in \mathbf{F}^{s}_{p,q}(\Omega_1) \mbox{ and } g \in \mathbf{F}^{s}_{p,q}(\Omega_2) & \implies g \circ f \in F^{s}_{p,q}(\Omega_1),
\end{align}
(see Figure \ref{figComposition}) and
\begin{equation}\label{eqInvarianceUnderInversionTRIEBELBis}
f \in \mathbf{F}^{s}_{p,q}(\Omega_1) \implies f^{-1} \in F^{s}_{p,q}(\Omega_2).
\end{equation}
 \end{theorem}
Note that if $(s-1)p{>} d$  then $\mathbf{F}^s_{p,q}=F^s_{p,q}$. 

The reason behind the rather unnatural assumption $s\notin\N$ in Theorem \ref{theoTriebelAdmissibleBanachBis} is the use of first-order differences to characterize the function space, since otherwise one needs to use second-order differences and the techniques used throughout this paper are not enough. However, the composition result holds for $s\in\N$ in the one-dimensional case at least (see \cite{BourdaudMoussaiSickelTL}) and the inverse function result holds also in the one-dimensional case for $p=q$ and $sp>1$ (see \cite[Lemma 2.11]{AstalaPratsSaksman}), and for $s\in \N$ and $q=2$ in arbitrary dimensions, that is, in the Sobolev scale  (see Lemma \ref{lemSobolevAdmissibleSpace}). {For the Sobolev scale ($q=2$, $s\in\N$) in the higher dimensional case we refer the reader to \cite{CampbellHenclKonopecky}, and Lemma \ref{lemSobolevAdmissibleSpace} below.} The result also holds for non-integer H\"older spaces, see Lemma \ref{lemCsAdmissibleBanachBis}.

We can conjecture that Theorem \ref{theoTriebelAdmissibleBanachBis} remains true in uniform domains, see Remark \ref{remDomainType} for a discussion. {Recall that bounded Lipschitz domains are uniform.}

To obtain the preceding result, we use elementary techniques such as H\"older inequalities, but we need to build on first-order differences to be able to use the change of variables. We will use the following characterization, {which is proven for uniform domains, see Section \ref{secUniform} for the definitions.}
\begin{theorem}\label{theoDifference}
Let $\Omega$ be a uniform domain, let $1\leq p<\infty$, $1\leq q\leq \infty$, $s=k+\sigma$ with $0<\sigma<1$, $k\in \N_0:=\N\cup \{0\}$ and consider the auxiliary index $1\leq u\leq \infty$ so that $\sigma>\frac{d}{p\wedge q}-\frac du$. Then 
\begin{equation}\label{eqEquivalentNormFSPQ}
\norm{f}_{F^s_{p,q}(\Omega)}\approx \norm{f}_p + \left(\int_\Omega \left(\int_0^1 \frac{\left(  \int_{B(x,t)\cap\Omega} |\nabla^k f(x)-\nabla^k f(y)|^u dy\right)^\frac qu}{t^{\left(\sigma+\frac du\right)q}}  \frac{dt}t \right)^\frac pq dx\right)^\frac1p.
\end{equation}
\end{theorem}

Theorem \ref{theoDifference} is proven in Section \ref{secExtension}. The idea is to check that this norm is equivalent to the restriction of the usual Fourier definition for the Triebel-Lizorkin scale for $\Omega=\R^d$ (see \cite{TriebelTheoryIII}). Thus, it is enough to find a suitable extension operator such that the Triebel-Lizorkin norm of the extended function is bounded by the right-hand side of \rf{eqEquivalentNormFSPQ}. As a matter of fact, in Theorem 4.7 below we will see that the Peter-Jones extension operator for the Sobolev scale is also an extension operator for the relevant function spaces, following a similar reasoning to \cite{PratsSaksman}.

{The introduction of uniform domains at this step is quite natural. At the end of the day, most extension operators are defined by a discretization of the complementary domain by associating to a given patch the value defined at a certain symmetrization of this patch (see \cite{Rychkov} for a different approach). Usually, this patches are Whitney cubes. The connectivity properties of this symmetrizations for neighboring patches are used to ensure that no new jumps are created when gluing together the discretized reflections. In Lipschitz domains one can use the trivial symmetrization (at least locally), but as showed by Peter Jones in \cite{Jones}, this procedure can be extended quite naturally to uniform domains and the proof does not get essentialy more involved.}

At this point we want to remark that the Peter Jones extension operator defined \cite{Jones} for Sobolev spaces with smoothness one in uniform domains is also an extension operator for Triebel-Lizorkin spaces on domains with smoothness below one in interior corkscrew domains. This fact was unnoticed in \cite{PratsSaksman}, although the proof there can be easily modified to cover this quite general setting. See Theorem \ref{theoExtensionOperator0} and Remark \ref{remProofExtension} below for a discussion on this matter.

The issue of stability of the composition operators has already been discussed thoroughly in the literature.
See \cite{ClopFaracoRuiz, HenclKoskelaComposition, OlivaPrats} for results concerning the linear composition operator for quasiconformal mappings, and \cite{HenclKoskela} for mappings of finite distortion, in both cases the authors study the case of smoothness smaller or equal than one. It is interesting to note that for critical Bessel potential spaces i.e. $F^{d/p}_{p,2}$ with $\frac dp\leq 1$, every quasiconformal mapping preserves the function space. Quasiconformal mappings are known to have  H\"older regularity below one, determined by their distortion. This is much weaker than bi-Lipschitz, and it is natural to wonder whether Theorem \ref{theoDifference} can be also weakened in such a way for function spaces with regularity greater than one.

We also refer to \cite{Dahlberg, Vodopyanov, BourdaudSickel, Bourdaud, BourdaudMoussaiSickel, BourdaudMoussaiSickel2} for the study of the non-linear  composition operator $\widetilde{T}_fg=f\circ g$ and regularity in the Besov and Triebel-Lizorkin scales. It is worthy to note the result in \cite{Dahlberg}, where it is seen that for $d\geq 3$ and $f\in C^\infty(\R)$ then $\widetilde{T}_f:W^{s,p}(\R^d)\to W^{s,p}(\R^d)$ implies that $f=c {\rm Id}$ whenever $s\in\N$, $s\geq 2$, $1\leq p<\infty$ with $s<\frac dp$. 

The author of the present paper is unaware of any study concerning the Triebel-Lizorkin regularity of the inverse of bi-Lipschitz mappings, {see \cite{CampbellHenclKonopecky} for the inverse mapping theorem in the Sobolev case.}

\subsection{Notation}
Throughout this paper we will write $C$ for constants which may change from one occurrence to the next. If we want to make clear in which parameters $C$ depends, we will add them as a subindex. In the same spirit, when comparing two quantities $x_1$ and $x_2$, we may write $x_1\lesssim x_2$ instead of $x_1\leq C x_2$, and $x_1\lesssim_{p_{1},\dots, p_{j}} x_2$ for $x_1\leq C_{p_{1},\dots, p_{j}} x_2$, meaning that the constant depends on all these parameters.

Given $1\leq p\leq \infty$ we write $p'$ for its H\"older conjugate, that is $\frac{1}{p}+\frac{1}{p'}=1$.

Given $x\in \R^d$ and $r>0$, we write $B(x,r)$ or $B_r(x)$ for the open ball centered at $x$ with radius $r$ and $Q(x,r)$ for the open cube centered at $x$ with sides parallel to the axis and side-length $2r$. Given any cube $Q$, we write $\ell(Q)$ for its side-length, and $rQ$ will stand for the cube with the same center but enlarged by a factor $r$. We will use the same notation for one dimensional  balls and cubes, that is, intervals. {Given two bounded sets $A$ and $B$, we write $\Dist(A,B):=\diam(A)+\diam(B)+\dist(A,B)$, and we call it the \emph{long distance} of $A$ and $B$.}

\begin{definition}\label{defLipschitzDomain}
Let $\delta,R>0$, $d\geq 2$. We say that a domain $\Omega\subset \R^d$ is a $(\delta,R)$-Lipschitz domain (or just a Lipschitz domain when the constants are not important) if for every point $z\in\partial\Omega$, there exists a cube $\mathcal{Q}=Q(0,R)$ and a Lipschitz function $A_z:\R^{d-1}\to\R$ supported in $[-4R,4R]^{d-1}$ such that $\norm{A_z'}_{L^\infty}\leq \delta$ and,  possibly after a translation that sends $z$ to the origin and a rotation, we have that
$$\mathcal{Q}\cap \Omega = \{(x,y)\in \mathcal{Q}: y>A_z(x)\}.$$

If $d=1$ we say that $\Omega\subset \R$ is a Lipschitz domain if $\Omega$ is an open interval.
\end{definition}

The natural numbers are denoted by $\N$ if $0$ is not included, and $\N_0=\N \cup \{0\}$. The multiindex notation for exponents and derivatives will be used: for $\alpha\in\Z^d$  its {\em modulus} is $|\alpha|=\sum|\alpha_i|$ and its {\em factorial} is $\alpha!=\prod(\alpha_i!)$. Given two multiindices $\alpha, \gamma\in \Z^d$ we write $\alpha\leq \gamma$ if $\alpha_i\leq \gamma_i$ for every $i$. We say $\alpha<\gamma$ if, in addition, $\alpha\neq\gamma$. 
For $x\in\R^d$ and $\alpha\in \Z^d$ we write $x^\alpha:=  \prod x_i^{\alpha_i} $. A similar notation is used for directional weak derivatives: $D^\alpha:=  \prod \partial_{x_i}^{\alpha_i} $.

\section{Composition and inverse function theorems}

In this section we will show that the function spaces considered are stable under composition with bi-Lipschitz mappings of the same space and they satisfy an inverse function theorem.

First we need a lemma on a generalized chain rule. For this purpose we recover the multivariate version of  Fa\`a di Bruno's formula (see \cite[Lemma 1.3.1]{KrantzParks} for the one-dimensional case), whose proof is a mere exercise on induction. Given a multiindex $\vec{i}\in \N_0^D$, where $\N_0{=\N \cup \{0\}}$,  we define $m(\vec{i})\in \{1,\cdots,D\}^{|\vec{i}|}$ as the vector whose components are non-decreasing (i.e, $m(\vec{i})_\ell \leq m(\vec{i})_{\ell+1}$), and such that 
$$\#\{j: m(\vec{i})_\ell= j\} = \vec{i}_j.$$
For instance, $m(3,2)=(1,1,1,2,2)$, and $m(4,0,1)=(1,1,1,1,3)$.
\begin{lemma}[Chain rule]\label{lemChainRule}
Given $f=(f^1,\cdots,f^{D}):\R^d\to\R^D$, $g:\R^D\to\R$ with $f^i \in W^{M,\infty}(\R^d)$ and {given a bi-Lipschitz function}  $g\in W^{M,\infty}(\R^D)$ and $\vec{k}\in \N_0^d$ with $|\vec{k}|=M$, there exist appropriate constants such that
\begin{equation}\label{eqGeneralizedChainRule}
D^{\vec{k}}(g\circ f)= \sum_{\substack{1\leq |\vec{i}|\leq M \\ \{\alpha_j\}_{j=1}^{|\vec{i}|} \subset \N_0^d\setminus\{\vec{0}\}: \sum |\alpha_j|=M }}C_{\vec{k}, \vec{i},\{\alpha_j\}} D^{\vec{i}}g(f)  \prod_{\ell=1}^{|\vec{i}|} D^{\alpha_\ell}f^{m(\vec{i})_\ell}
\end{equation}
almost everywhere.
\end{lemma}

\begin{remark}\label{remChainRule}
The chain rule \rf{eqGeneralizedChainRule} can be applied also to functions with weaker a-priori conditions. Note that given a bi-Lipschitz function $f$ and $g\in W^{M,1}_{\rm loc}$, for $|\vec{i}|\leq M-1$ we have that  $D(D^{\vec{i}}g(f))=D(D^{\vec{i}}g)(f)\cdot Df$ by \cite[Theorem 2.2.2]{Ziemer}. Thus, to prove \rf{eqGeneralizedChainRule} by induction for functions in $W^{M,1}_{\rm loc}$, one only needs to check that the product rule for the derivatives applies at each step. For this to hold it is enough that for $|\vec{k}|\leq M$ the right-hand side of \rf{eqGeneralizedChainRule} is locally in $L^1$, see \cite[(7.18)]{GilbargTrudinger}.
\end{remark}

\subsection{Toy case: H\"older continuity}
\begin{definition}\label{defHolderZygmund}
Given an open set $U\subset\R^d$,  {a measurable function $f:U\to \R$,} and $0<s<1$, we say that 
$f\in\dot{\mathcal{C}}^s(U)$ if 
$$\norm{f}_{\dot{\mathcal{C}}^s(U)}:=\sup_{x, y\in U} \frac{|f(x)-f(y)|}{|x-y|^s}<\infty.$$
For $k\in \N$ and $k<s{<} k+1$, we say that $f\in\dot{\mathcal{C}}^s(U)$ if $\nabla^k f:=(\partial_1^k f, \partial_1^{k-1}\partial_2 f,\cdots, \partial_{d}^{k} f)$ (that is, a vector with all the partial derivatives of order $k$) is in $\dot{\mathcal{C}}^{s-k}(U)$, with
 $$\norm{f}_{\dot{\mathcal{C}}^s(U)}:=\norm{\nabla^k f}_{\dot{\mathcal{C}}^{s-k}(U)}.$$
\end{definition}

One can define Banach spaces of functions modulo polynomials using the previous seminorms. However, the standard  non-homogeneous H\"older-Zygmund spaces are more suitable for our purposes:
\begin{definition}\label{defHolderNonhomogeneousNorm}
For $0<s<\infty$ with $s\notin\N$, we say that  $f\in \mathcal{C}^{s}(U)$ if $f\in L^\infty \cap \dot{\mathcal{C}}^s(U)$. We define the norm
$$\norm{f}_{\mathcal{C}^{s}(U)}:=\norm{f}_{L^\infty(U)}+\norm{f}_{\dot{\mathcal{C}}^s(U)}.$$
\end{definition}

Most likely the following results appear in the literature,  but we were not able to locate them, so we include these results for the sake of completeness. Moreover, the main steps of the proof of Theorem \ref{theoTriebelAdmissibleBanachBis} appear already in the H\"older scale: 
\begin{lemma}\label{lemCsAdmissibleBanachBis}
Let $1<s$, $s\notin\N$ and $d\in\N$.   Assume that  $\Omega_j\subset \R^d$, $j=1,2$, are open sets. Let  $f:\Omega_1 \twoheadrightarrow \Omega_2$ be bi-Lipschitz with $f\in \mathcal{C}^{s}(\Omega_1.)$ Then for any $ g\in \mathcal{C}^{s}(\Omega_2)$ we have
\begin{align}\label{eqInvarianceUnderbiLipschitzHOLD}
g \circ f \in \mathcal{C}^{s}(\Omega_1),\\
\nonumber \nabla g \circ f \in \mathcal{C}^{s-1}(\Omega_1),
\end{align}
and
\begin{equation}\label{eqInvarianceUnderInversionHOLD}
f^{-1} \in \mathcal{C}^{s}(\Omega_2).
\end{equation}
 \end{lemma}
\begin{proof}
Let us check \rf{eqInvarianceUnderbiLipschitzHOLD}. According to \rf{eqGeneralizedChainRule}, for $s={k}+\sigma$ with ${k} \in \N_0$, $0<\sigma < 1$, we get
\begin{align}\label{eqGeneralizedChainRuleDifferences}
\nonumber |\nabla^k &(g\circ f)(x)-\nabla^k(g\circ f)(y)|\\
	& \lesssim  \sum_{1\leq i\leq k} |\nabla^{i} g(f(x))-\nabla^i g(f(y))| \sum_{\alpha\in\N^i: |\alpha|=k} \prod_{j=1}^{i} | \nabla^{\alpha_j} f (x)| \\
\nonumber	&  + \sum_{1\leq i\leq k} |\nabla^i g(f(y))| \sum_{\alpha\in\N^i: |\alpha|=k} \sum_{\ell=1}^{i} | \nabla^{\alpha_\ell} f (x)-\nabla^{\alpha_\ell} f (y)|\prod_{j=1}^{\ell-1} | \nabla^{\alpha_j} f (y)| \prod_{j=\ell+1}^{i} | \nabla^{\alpha_j} f (x)|,
\end{align}
where we assume always $\alpha_j\geq1$. This implies that 
\begin{align*}
\norm{ g\circ f}_{\dot{\mathcal{C}}^s}
	& \lesssim  \sum_{1\leq i\leq k} C_{f} \norm{\nabla^{i} g}_{\dot{\mathcal{C}}^{\sigma}} \sum_{\alpha\in\N^i: |\alpha|=k} \prod_{j=1}^{i} \norm{\nabla^{\alpha_j} f}_{L^\infty} \\
	& \quad + \sum_{1\leq i\leq k} \norm{\nabla^i g}_{L^\infty} \sum_{\alpha\in\N^i: |\alpha|=k} \sum_{\ell=1}^{i}\norm{ \nabla^{\alpha_\ell} f}_{\dot{\mathcal{C}}^{\sigma}} \prod_{j\neq \ell}  \norm{\nabla^{\alpha_j} f}_{L^\infty},
\end{align*}
so
\begin{equation}\label{eqQuantifyCompositionHolder}
\norm{g\circ f}_{\dot{\mathcal{C}}^s}
	 \leq C_{f} \norm{g}_{\mathcal{C}^s(\Omega_2)},
\end{equation}
with $C_{f}$ depending polynomially on the $\mathcal{C}^{\sigma}$ norm of the derivatives of $f$ and its bi-Lipschitz constant. The second inequality follows from the first one. In fact,  $h \circ f \in \mathcal{C}^{s-1}(\Omega_1)$ whenever $h\in \mathcal{C}^{s-1}(\Omega_2)$.

Finally, let us prove \rf{eqInvarianceUnderInversionHOLD}. Applying the inverse function theorem, 
\begin{equation*}
D(f^{-1})(x) = (Df)^{-1}(f^{-1}(x)).
\end{equation*}
That is, the first-order derivatives of the inverse can be expressed {using Cramer's rule} as 
\begin{equation}\label{eqInverseDerivatives}
(D(f^{-1}))_{ij}=g_{ij}\circ{(f^{-1})}\mbox{, \quad\quad where \quad\quad}g_{ij}=\frac{P_{ij}(Df)}{\det(Df)}
\end{equation} 
for certain homogeneous polynomials $P_{ij}:\R^{d\times d}\to\R$ of degree $d-1$. By induction we can 
assume $f^{-1}\in \mathcal{C}^{s-1}$ (note that the starting point of the induction is obtained from the bi-Lipschitz assumption), and by \rf{eqInvarianceUnderbiLipschitzHOLD} it is enough to check that $g_{ij}\in C^{s-1}(\Omega_1)$.
But every derivative of degree $k-1$ of $g_{ij}$ is a polynomial of degree $kd-1$ on the derivatives of $f$ with $k-1$ new derivations taken at each {summand}, possibly taking more than one of these new derivations to some of the derivatives of $f$, divided by the $k$-th power of the Jacobian determinant, i.e., for every $\alpha\in\N_0^d$  with $|\alpha|=k-1$ we have
\begin{equation}\label{eqDerivativesInverse}
D^\alpha g_{ij}= \sum_{\substack{\beta\in\N_0^{d\times d}: |\beta|=(d-1)k \\ \gamma\in(\N_0^d)^{k-1}: |\gamma_\ell|\geq 1\, \&\,\sum|\gamma_\ell|=2k-2\\\mu\in\{1,\dots,d\}^{k-1}}} \frac{C_{i,j,\alpha,\beta,\gamma,\mu} (D f)^{\beta} \prod_{\ell=1}^{k-1} D^{\gamma_\ell} f_{\mu_\ell}}{\det(Df)^k}.
\end{equation}
 Applying the argument in \rf{eqGeneralizedChainRuleDifferences} to each of these derivatives we obtain \rf{eqInvarianceUnderInversionHOLD}.
\end{proof}

\subsection{Justification of the chain rule: the Sobolev scale}
Next we adapt the approach above to show a counterpart to Theorem \ref{theoTriebelAdmissibleBanachBis} for Sobolev spaces.  To adapt the argument above to the Sobolev setting we need to add a restriction that allows us to take appropriate H\"older inequalities. This is based on the following interpolation inequalities:
\begin{proposition}[{see \cite[Theorem 2.2.5]{RunstSickel}}]\label{propoRunstSickel}
Let $0<t<\infty$, $0<p<\infty$, $0<r,\ell \leq\infty$ and $0<\Theta<1$. Then every {locally integrable function} $g$ satisfies that
\begin{equation}\label{eqInterpolationInfty}
\norm{g}_{F^{\Theta t}_{\frac p\Theta, r}}\leq C_{t,p,r,\ell,\Theta} \norm{g}_{F^t_{p,\ell}}^\Theta \norm{g}_{L^\infty}^{1-\Theta}.
\end{equation}
\end{proposition}

We also need the following property of the Rychkov extension operator:
\begin{theorem}[{see \cite[Appendix]{AstalaPratsSaksman}}]\label{theoRychkov}
Given a bounded Lipschitz domain $\Omega$ and $s\in \N$, there exists an operator $\mathcal{E}:=\mathcal{E}_s$ defined in $\mathcal{D}'(\Omega)$ that is an extension operator from $L^\infty(\Omega)$ to $L^\infty$ and from $F^\sigma_{p,q}(\Omega)$ to $F^\sigma_{p,q}$ for every $\sigma\leq s$, every $1/s<p<\infty$ and every $1/s\leq q\leq \infty$.
\end{theorem}
\begin{remark}\label{remDomainType}
It would be highly appreciated to have the same result for uniform domains. In the Sobolev scale this is proven in \cite{RogersExtension}. It seems natural to think that the same operator may work in the Triebel-Lizorkin scale and may include also $L^\infty$. If that was true, all the results in this paper could be extended to uniform domains.
\end{remark}

\begin{lemma}\label{lemReadyForHolderInequality}
Let {${k} \in \N$}, $0<\sigma \leq 1$ and  $s:={k}+\sigma$, let $ 1\leq p < \infty$, $ 1\leq q \leq\infty$ and $d\in\N$ and let $f\in F^{s}_{p,q}(\Omega)\cap C^{0,1}(\Omega)$ where $\Omega\subset\R^d$ is a bounded Lipschitz domain. Then, for every positive index $j\leq k$
$$\norm{\nabla^j f}_{ L^{p\frac{s-1}{j-1}}(\Omega)}\lesssim_{s,p,q,j,\Omega}  \norm{f}_{ F^{s}_{p,q}(\Omega)}^{\frac{j-1}{s-1}}\norm{\nabla f}_{ L^\infty(\Omega)}^{\frac{s-j}{s-1}}$$
{(here, and throughout the paper, we make the convention $\frac{t}0=\infty$ for any $t>0$).}
Moreover, for every $1\leq r\leq \infty$ and $M\geq1$ with $j+\sigma/M<s$, we have
$$\norm{\nabla^j f}_{ F^{\sigma/M}_{\frac{p(s-1)}{j+\sigma/M-1},r}(\Omega)}\lesssim_{s,p,q,r,j+\sigma/M,\Omega} \norm{f}_{F^{s}_{p,q}(\Omega)}^{\frac{j+\sigma/M-1}{s-1}} \norm{\nabla f}_{ L^\infty(\Omega)}^{\frac{s-j-\sigma/M}{s-1}} .$$
\end{lemma}
Note that $j+\sigma/M<s$ excludes only the case when both $M=1$ and $j=k$.

\begin{proof}
For the first embedding, {note that the case $j=1$ holds trivially:
$$\norm{\nabla f}_{ L^{\infty}(\Omega)}= \norm{f}_{ F^{s}_{p,q}(\Omega)}^0 \norm{\nabla f}_{ L^\infty(\Omega)}^1.$$
 Otherwise,} use Proposition \ref{propoRunstSickel} choosing $g=\mathcal{E}_s(\nabla f)$ where $\mathcal{E}_s$ is the Rychkov extension operator from Theorem \ref{theoRychkov},  $t=s-1$, $r=2$, $\ell=q$ and  $\Theta=\frac{j-1}{s-1}$. Then 
$$\norm{\mathcal{E}_s(\nabla f)}_{F^{j-1}_{\frac {p(s-1)}{j-1}, r}}\leq C_{s,p,q,j} \norm{\mathcal{E}_s(\nabla f)}_{F^{s-1}_{p,q}}^{\frac{j-1}{s-1}} \norm{\mathcal{E}_s(\nabla f)}_{L^\infty}^{\frac{s-j}{s-1}},$$
and the first statement follows.

For the second inequality, we do the same trick but we take instead $r$ given and set $\Theta = \frac{j+\sigma/M-1}{s-1}$. In this way we obtain 
$$\norm{\mathcal{E}_s(\nabla f)}_{F^{j+\sigma/M-1}_{\frac {p(s-1)}{j+\sigma/M-1}, r}}\leq C_{s,p,q,r,j+\sigma/M} \norm{\mathcal{E}_s(\nabla f)}_{F^{s-1}_{p,q}}^{\frac{j+\sigma/M-1}{s-1}} \norm{\mathcal{E}_s(\nabla f)}_{L^\infty}^{\frac{s-j-\sigma/M}{s-1}},$$
and the second statement follows as well.
\end{proof}

\begin{figure}[t]
\caption{On the first graphic, $f\in W^{5,p}\cap W^{1,\infty}$, so  $\nabla^2 f\in L^{4p}$, $\nabla^3 f\in L^{2p}$ and  $\nabla^4 f\in L^{\frac{4p}{3}}$. On the second we depict the case $f\in F^s_{p,q}\cap W^{1,\infty}$, with $4<s=4+\sigma<5$; in that case $\nabla f\in F^{\sigma/M}_{\frac{(s-1)p}{\sigma/M},r}$, $\nabla^2 f\in F^{\sigma/M}_{\frac{(s-1)p}{1+\sigma/M},r}$, $\nabla^3 f\in F^{\sigma/M}_{\frac{(s-1)p}{2+\sigma/M},r}$, and $\nabla^4 f\in F^{\sigma/M}_{\frac{(s-1)p}{3+\sigma/M},q}$ ($q$ can be replaced by $r$ if $M>1$). The circular dots describe the case $M=1$, the squares describe the case $M=2$. See Lemma \ref{lemReadyForHolderInequality}.}\label{figInterpolation}

\center
\begin{tikzpicture}[line cap=round,line join=round,>=triangle 45,x=3.8cm,y=0.7cm]
\clip(-0.41,-0.9277074696345513) rectangle (1.45,5.4);
\draw [line width=.5pt,color=grisfosc] (0.,0.) -- (0.,5.203518980893238);
\draw [line width=.5pt,color=grisfosc] (1.,0.) -- (1.,5.203518980893238);
\draw [line width=.5pt,color=grisfosc] (0.,5.)-- (1.,5.);
\draw [line width=.5pt,color=grisfosc] (0.,3.)-- (1.,3.);
\draw [line width=.5pt,color=grisfosc] (0.,4.)-- (1.,4.);
\draw [line width=.5pt,color=grisfosc] (0.,2.)-- (1.,2.);
\draw [line width=.5pt,color=grisfosc] (0.,1.)-- (1.,1.);
\draw [line width=.5pt,color=grisfosc] (0.,0.)-- (1.,0.);
\begin{scriptsize}
\draw (-0.13,5.3) node[anchor=north west] {$s$};
\draw (0.94,0.057219536897733694) node[anchor=north west] {$\frac11$};
\draw (-0.1,0.05) node[anchor=north west] {$\frac1\infty$};
\end{scriptsize}
\draw [line width=.5pt,color=grisfosc] (0.6880985479597884,5.)-- (0.,1.);
\begin{scriptsize}
\draw (0.64,5.1) node[anchor=north west] {$f\in W^{5,p}$};
\draw (0.47,4.1) node[anchor=north west] {$W^{4,\frac{4p}3}$};
\draw (0.3,3.1) node[anchor=north west] {$W^{3,2p}$};
\draw (0.12,2.1) node[anchor=north west] {$W^{2,4p}$};
\draw (-0.35,1.6) node[anchor=north west] {$f\in W^{1,\infty}$};
\end{scriptsize}
\draw [line width=0.5pt,dotted,color=cqcqcq] (0.1720246369899471,0.) -- (0.1720246369899471,5.203518980893238);
\draw [line width=0.5pt,dotted,color=cqcqcq] (0.34404927397989427,0.) -- (0.34404927397989427,5.203518980893238);
\draw [line width=0.5pt,dotted,color=cqcqcq] (0.5160739109698413,0.) -- (0.5160739109698413,5.203518980893238);
\draw [line width=0.5pt,dotted,color=cqcqcq] (0.6880985479597884,0.) -- (0.6880985479597884,5.203518980893238);
\begin{scriptsize}
\draw (0.1,0) node[anchor=north west] {$\frac1{4p}$};
\draw (0.28,0) node[anchor=north west] {$\frac1{2p}$};
\draw (0.45,0) node[anchor=north west] {$\frac3{4p}$};
\draw (0.64,0) node[anchor=north west] {$\frac1{p}$};
\draw [fill=uuuuuu] (0.,1.) circle (1.0pt);
\draw [fill=uuuuuu] (0.6880985479597884,5.) circle (1.0pt);
\draw [fill=uuuuuu] (0.34404927397989427,3.) circle (1.0pt);
\draw [fill=uuuuuu] (0.1720246369899471,2.) circle (1.0pt);
\draw [color=uuuuuu] (0.,0.)-- ++(-1.5pt,-1.5pt) -- ++(3.0pt,3.0pt) ++(-3.0pt,0) -- ++(3.0pt,-3.0pt);
\draw [color=uuuuuu] (0.1720246369899471,0.)-- ++(-1.5pt,-1.5pt) -- ++(3.0pt,3.0pt) ++(-3.0pt,0) -- ++(3.0pt,-3.0pt);
\draw [color=uuuuuu] (0.34404927397989427,0.)-- ++(-1.5pt,-1.5pt) -- ++(3.0pt,3.0pt) ++(-3.0pt,0) -- ++(3.0pt,-3.0pt);
\draw [color=uuuuuu] (0.6880985479597884,0.)-- ++(-1.5pt,-1.5pt) -- ++(3.0pt,3.0pt) ++(-3.0pt,0) -- ++(3.0pt,-3.0pt);
\draw [fill=uuuuuu] (0.5160739109698413,4.) circle (1.0pt);
\draw [color=uuuuuu] (0.5160739109698413,0.)-- ++(-1.5pt,-1.5pt) -- ++(3.0pt,3.0pt) ++(-3.0pt,0) -- ++(3.0pt,-3.0pt);
\draw [color=uuuuuu] (1.,0.)-- ++(-1.5pt,-1.5pt) -- ++(3.0pt,3.0pt) ++(-3.0pt,0) -- ++(3.0pt,-3.0pt);
\end{scriptsize}
\end{tikzpicture}
\, 
\begin{tikzpicture}[line cap=round,line join=round,>=triangle 45,x=3.8cm,y=0.7cm]
\clip(-0.41,-0.9277074696345513) rectangle (1.45,5.4);
\draw [line width=.5pt,color=grisfosc] (0.,0.) -- (0.,5.203518980893238);
\draw [line width=.5pt,color=grisfosc] (1.,0.) -- (1.,5.203518980893238);
\draw [line width=.5pt,color=grisfosc] (0.,5.)-- (1.,5.);
\draw [line width=.5pt,color=grisfosc] (0.,3.)-- (1.,3.);
\draw [line width=.5pt,color=grisfosc] (0.,4.)-- (1.,4.);
\draw [line width=.5pt,color=grisfosc] (0.,2.)-- (1.,2.);
\draw [line width=.5pt,color=grisfosc] (0.,1.)-- (1.,1.);
\draw [line width=.5pt,color=grisfosc] (0.,0.)-- (1.,0.);
\begin{scriptsize}
\draw (-0.13,5.3) node[anchor=north west] {$s$};
\draw (0.94,0.057219536897733694) node[anchor=north west] {$\frac11$};
\draw (-0.1,0.05) node[anchor=north west] {$\frac1\infty$};
\end{scriptsize}
\draw [line width=.5pt,color=grisfosc] (0.7341623110203971,4.392796751026526)-- (0.,1.);
\draw [line width=.5pt,color=grisfosc] (0.08499671264059261,0.392796751026526)-- (0.,0.);
\draw [line width=.5pt,color=grisfosc] (0.3013852454338608,0.392796751026526)-- (0.21638853279326817,0.);
\draw [line width=.5pt,color=grisfosc] (0.517773778227129,0.392796751026526)-- (0.4327770655865364,0.);
\draw [line width=.5pt,color=grisfosc] (0.7341623110203971,0.392796751026526)-- (0.6491655983798045,0.);
\draw [line width=0.5pt,dotted,color=cqcqcq] (0.21638853279326817,0.) -- (0.21638853279326817,5.203518980893238);
\draw [line width=0.5pt,dotted,color=cqcqcq] (0.4327770655865364,0.) -- (0.4327770655865364,5.203518980893238);
\draw [line width=0.5pt,dotted,color=cqcqcq] (0.6491655983798045,0.) -- (0.6491655983798045,5.203518980893238);
\draw [line width=0.5pt,dotted,color=cqcqcq] (0.7341623110203971,0.) -- (0.7341623110203971,5.203518980893238);
\draw [line width=0.5pt,dotted,color=cqcqcq] (0.517773778227129,0.) -- (0.517773778227129,5.203518980893238);
\draw [line width=0.5pt,dotted,color=cqcqcq] (0.30138524543386075,0.) -- (0.30138524543386075,5.203518980893238);
\draw [line width=0.5pt,dotted,color=cqcqcq] (0.08499671264059261,0.) -- (0.08499671264059261,5.203518980893238);
\draw [line width=0.5pt,dotted,color=cqcqcq] (0.042498356320296304,0.) -- (0.042498356320296304,5.203518980893238);
\draw [line width=0.5pt,dotted,color=cqcqcq] (0.25888688911356444,0.) -- (0.25888688911356444,5.203518980893238);
\draw [line width=0.5pt,dotted,color=cqcqcq] (0.4752754219068327,0.) -- (0.4752754219068327,5.203518980893238);
\draw [line width=0.5pt,dotted,color=cqcqcq] (0.6916639547001008,0.) -- (0.6916639547001008,5.203518980893238);
\begin{scriptsize}
\draw (0.7,5) node[anchor=north west] {$f\in F^s_{p,q}$};
\draw (-0.35,1.6) node[anchor=north west] {$f\in W^{1,\infty}$};
\draw (-0.04,-0.08) node[anchor=north west] {$\frac{\sigma}{(s-1)p}$};
\draw (0.18,0) node[anchor=north west] {$\frac{1+\sigma}{(s-1)p}$};
\draw (0.4,0) node[anchor=north west] {$\frac{2+\sigma}{(s-1)p}$};
\draw (0.68,0) node[anchor=north west] {$\frac1{p}$};
\draw (0.7,0.95) node[anchor=north west] {$\nabla^4 f$};
\draw (0.48,0.95) node[anchor=north west] {$\nabla^3 f$};
\draw (0.26,0.95) node[anchor=north west] {$\nabla^2 f$};
\draw (0.05,0.95) node[anchor=north west] {$\nabla f$};
\draw [fill=uuuuuu] (0.,1.) circle (1.0pt);
\draw [fill=uuuuuu] (0.7341623110203971,4.392796751026526) circle (1.0pt);
\draw [fill=uuuuuu] (0.7341623110203971,0.392796751026526) circle (1.0pt);
\draw [fill=uuuuuu] (0.517773778227129,0.392796751026526) circle (1.0pt);
\draw [fill=uuuuuu] (0.3013852454338608,0.392796751026526) circle (1.0pt);
\draw [fill=uuuuuu] (0.08499671264059261,0.392796751026526) circle (1.0pt);
\draw [fill=uuuuuu] (0.517773778227129,3.3927967510265256) circle (1.0pt);
\draw [fill=uuuuuu] (0.3013852454338608,2.392796751026526) circle (1.0pt);
\draw [fill=uuuuuu] (0.08499671264059261,1.392796751026526) circle (1.0pt);
\draw [fill=uuuuuu] (0.2588868891135645,0.196398375513263)  ++(-1.5pt,0 pt) -- ++(1.5pt,1.5pt)--++(1.5pt,-1.5pt)--++(-1.5pt,-1.5pt)--++(-1.5pt,1.5pt);
\draw [fill=uuuuuu] (0.4752754219068327,0.196398375513263)  ++(-1.5pt,0 pt) -- ++(1.5pt,1.5pt)--++(1.5pt,-1.5pt)--++(-1.5pt,-1.5pt)--++(-1.5pt,1.5pt);
\draw [fill=uuuuuu] (0.6916639547001008,0.196398375513263)  ++(-1.5pt,0 pt) -- ++(1.5pt,1.5pt)--++(1.5pt,-1.5pt)--++(-1.5pt,-1.5pt)--++(-1.5pt,1.5pt);
\draw [fill=uuuuuu] (0.4752754219068327,3.1963983755132626)  ++(-1.5pt,0 pt) -- ++(1.5pt,1.5pt)--++(1.5pt,-1.5pt)--++(-1.5pt,-1.5pt)--++(-1.5pt,1.5pt);
\draw [fill=uuuuuu] (0.2588868891135645,2.196398375513263)  ++(-1.5pt,0 pt) -- ++(1.5pt,1.5pt)--++(1.5pt,-1.5pt)--++(-1.5pt,-1.5pt)--++(-1.5pt,1.5pt);
\draw [fill=uuuuuu] (0.042498356320296304,1.196398375513263)  ++(-1.5pt,0 pt) -- ++(1.5pt,1.5pt)--++(1.5pt,-1.5pt)--++(-1.5pt,-1.5pt)--++(-1.5pt,1.5pt);
\draw [fill=uuuuuu] (0.6916639547001008,4.1963983755132634)  ++(-1.5pt,0 pt) -- ++(1.5pt,1.5pt)--++(1.5pt,-1.5pt)--++(-1.5pt,-1.5pt)--++(-1.5pt,1.5pt);
\draw [fill=uuuuuu] (0.042498356320296304,0.196398375513263) ++(-1.5pt,0 pt) -- ++(1.5pt,1.5pt)--++(1.5pt,-1.5pt)--++(-1.5pt,-1.5pt)--++(-1.5pt,1.5pt);
\draw [color=uuuuuu] (0.,0.)-- ++(-1.5pt,-1.5pt) -- ++(3.0pt,3.0pt) ++(-3.0pt,0) -- ++(3.0pt,-3.0pt);
\draw [color=uuuuuu] (0.7341623110203971,0.)-- ++(-1.5pt,-1.5pt) -- ++(3.0pt,3.0pt) ++(-3.0pt,0) -- ++(3.0pt,-3.0pt);
\draw [color=uuuuuu] (0.517773778227129,0.)-- ++(-1.5pt,-1.5pt) -- ++(3.0pt,3.0pt) ++(-3.0pt,0) -- ++(3.0pt,-3.0pt);
\draw [color=uuuuuu] (0.30138524543386075,0.)-- ++(-1.5pt,-1.5pt) -- ++(3.0pt,3.0pt) ++(-3.0pt,0) -- ++(3.0pt,-3.0pt);
\draw [color=uuuuuu] (0.08499671264059261,0.)-- ++(-1.5pt,-1.5pt) -- ++(3.0pt,3.0pt) ++(-3.0pt,0) -- ++(3.0pt,-3.0pt);
\draw [color=uuuuuu] (1,0.)-- ++(-1.5pt,-1.5pt) -- ++(3.0pt,3.0pt) ++(-3.0pt,0) -- ++(3.0pt,-3.0pt);
\end{scriptsize}
\end{tikzpicture}
\end{figure}

According to the previous result,  we will prove some properties  for subspaces of $W^{s,p}$ whose functions have bounded first derivatives. Namely, we define the space $\mathbf{W}^{s,p}(\Omega):=W^{s,p}(\Omega)\cap C^{0,1}(\Omega)$. By the Sobolev embedding {theorem,} when $sp>d$ we have that $\mathbf{W}^{s+1,p}(\Omega)=W^{s+1,p}(\Omega)$.

\begin{lemma}\label{lemSobolevAdmissibleSpace}
Let $s,d\in\N$, and $ 1< p< \infty$. Given bounded Lipschitz domains $\Omega_j\subset \R^d$ and functions $f,g$ with $f(\Omega_1)= \Omega_2$ and $f$ {bi-Lipschitz}, then 
\begin{align}\label{eqInvarianceUnderbiLipschitzSOBOBis}
f \in \mathbf{W}^{s,p}(\Omega_1) \mbox{ and } g \in \mathbf{W}^{s,p}(\Omega_2) & \implies g \circ f \in {W}^{s,p}(\Omega_1),\\
\nonumber f \in \mathbf{W}^{s,\frac{ps}{s-1}}(\Omega_1)\mbox{ and } g \in {W}^{s,p}\cap L^\infty({\Omega_2}) & \implies   g \circ f \in {W}^{s,p}(\Omega_1)
\end{align}
(see Figure \ref{figComposition}) and the chain rule \rf{eqGeneralizedChainRule} applies (for $M=s$). Moreover,
\begin{equation}\label{eqInvarianceUnderInversionSOBOBis}
f \in \mathbf{W}^{s,p}(\Omega_1)  \implies f^{-1} \in {W}^{s,p}(\Omega_2),
\end{equation}
and \rf{eqDerivativesInverse} holds.
\end{lemma}

\begin{proof}
To check \rf{eqInvarianceUnderbiLipschitzSOBOBis}, the case $s=1$ is \cite[Theorem 2.2.2]{Ziemer}, so let us assume that $s\geq 2$.  Since both $f_1$ and $f_2$ are in $ W^{s,p}(\Omega_j)$, all their derivatives satisfy that $\nabla^i f_j\in L^{p\frac{s-1}{i-1}}$ in view of Lemma \ref{lemReadyForHolderInequality}. 

By Remark \ref{remChainRule}, we only need to check that the right-hand side of \rf{eqGeneralizedChainRule} is in $L^p$, and then by induction it follows that the chain rule applies. By H\"older's inequality, 
\begin{equation}\label{eqCompositionDecomposition}
\circled{1}:=\norm{\sum_{1\leq i\leq s} \sum_{\alpha\in\N^i: |\alpha|=s}   \nabla^i f_2(f_1) \prod_{\ell=1}^{i} \nabla^{\alpha_\ell}f_1}_{L^p} \lesssim_{d,s}  \sum_{1\leq i\leq s} \sum_{\alpha\in\N^i: |\alpha|=s}   \norm{\nabla^i f_2(f_1)}_{p_0} \prod_{\ell=1}^{i} \norm{\nabla^{\alpha_\ell}f_1}_{p_\ell},
\end{equation}
where $\sum_0^i \frac{1}{p_\ell}=\frac1p$. This can be achieved by letting ${p_0}=\frac{p(s-1)}{i-1}$ and ${p_\ell}=\frac{p(s-1)}{\alpha_\ell-1}$. {Note that both $i$ and $\alpha_\ell$ may equal $1$ in some occasions, in which case we consider the $L^\infty$ norm of a derivative of order one, which is finite by hypothesis.} Thus,  Lemma \ref{lemReadyForHolderInequality} applies and using Young's inequality for products  we get 
\begin{align}\label{eqGeneralizedChainRuleSobolev}
\circled{1}
	& \lesssim \sum_i \norm{\nabla f_1^{-1}}_{\infty}^{\frac{d(i-1)}{p(s-1)}}  \norm{f_2}_{ W^{s,p}(\Omega_2)}^{\frac{i-1}{s-1}}\norm{\nabla f_2}_{L^\infty(\Omega_2)}^{\frac{s-i}{s-1}}  \norm{f_1}_{ W^{s,p}(\Omega_1)}^{\frac{s-i}{s-1}}\norm{\nabla f_1}_{ L^\infty(\Omega_1)}^{\frac{is-s}{s-1}}\\
	\nonumber & \lesssim C_{f_1}   (\norm{f_2}_{{W}^{s,p}(\Omega_2)}\norm{\nabla f_1}_{L^\infty(\Omega_1)}^s+\norm{\nabla f_2}_{L^\infty(\Omega_2)}\norm{f_1}_{{W}^{s,p}(\Omega_1)}),
\end{align}
with the constant $C_{f_1}$ depending on the bi-Lipschitz character of $f_1$.

%
%

The second inclusion in \rf{eqInvarianceUnderbiLipschitzSOBOBis} can be shown analogously, setting ${p_0}=\frac{ps}{i}$ and ${p_\ell}=\frac{ps}{\alpha_\ell-1}$ in \rf{eqCompositionDecomposition}.  We leave the details to the reader.

The inverse function bound \rf{eqInvarianceUnderInversionSOBOBis} can be proven by the same methods using \rf{eqDerivativesInverse}. 
Indeed, \rf{eqInverseDerivatives} holds for every bi-Lipschitz function by the chain rule, and arguing as in Remark \ref{remChainRule}, we only need to check that the right-hand side in \rf{eqDerivativesInverse}  belongs to $L^p$ for every $|\alpha|=s$ in order to prove that \rf{eqDerivativesInverse} holds and that $g_{ij}\in W^{s-1,p}(\Omega_1)$. Indeed, 
\begin{align*}
& \norm{\sum_{ \substack{ \gamma\in(\N_0^d)^{s-1}\\ |\gamma_\ell|\geq 1\, \&\,\sum|\gamma_\ell|=2s-2}} |D f|^{(d-1)s} \prod_{\ell=1}^{s-1} |D^{\gamma_\ell} f| |D(f^{-1})|^{ds}}_{L^p} \\
	& \quad \lesssim_{d,s}   \sum_{ \substack{ \gamma\in(\N_0^d)^{s-1}\\ |\gamma_\ell|\geq 1\, \&\,\sum|\gamma_\ell|=2s-2}} \norm{D f}_\infty^{(d-1)s} \prod_{\ell=1}^{s-1} \norm{D^{\gamma_\ell} f}_{p_\ell}\norm{D(f^{-1})}_{L^\infty(\Omega)}^{ds}\\
	&\quad  \lesssim \norm{f}_{W^{s,p}(\Omega)}\norm{Df}_{L^\infty(\Omega)}^{ds-2}\norm{D(f^{-1})}_{L^\infty(\Omega)}^{ds}
\end{align*}
by choosing ${p_\ell}=\frac{p(s-1)}{|\gamma_\ell|-1}$ and applying Lemma \ref{lemReadyForHolderInequality} with $q=2$. We obtain that $g_{ij}\in W^{s-1,p}(\Omega_1)$ and \rf{eqDerivativesInverse} holds. 

On the other hand, if $s=2$ then $f\in C^{0,1}$ and $f^{-1}\in C^{0,1}$. By \rf{eqInverseDerivatives} and \rf{eqInvarianceUnderbiLipschitzSOBOBis} we get $(D(f^{-1}))_{ij}\in W^{s-1,p}(\Omega_2)$.
If, instead, $s>2$ then $f\in \mathbf{W}^{s-1,\frac{p(s-1)}{s-2}}(\Omega_1)$  by Lemma \ref{lemReadyForHolderInequality} and therefore, by induction, we can assume that \rf{eqInvarianceUnderInversionSOBOBis} holds in this case, so  $f^{-1} \in W^{s-1,\frac{p(s-1)}{s-2}}(\Omega_2)$ and applying the second estimate in \rf{eqInvarianceUnderbiLipschitzSOBOBis} to the composition in \rf{eqInverseDerivatives} we get that $(D(f^{-1}))_{ij}\in W^{s-1,p}(\Omega_2)$. 
\end{proof}


\subsection{Proof of Theorem \ref{theoTriebelAdmissibleBanachBis}: the Triebel-Lizorkin scale}
Finally we will verify that Triebel-Lizorkin spaces have the same properties.
Again, we define $\mathbf{F}^s_{p,q}(\Omega):=F^{s}_{p,q}(\Omega) \cap C^{0,1}(\Omega)$ for $k<s<k+1$. Note that when $sp>d$ we have that $\mathbf{F}^{s+1}_{p,q}(\Omega)=F^{s+1}_{p,q}(\Omega)$.

%

\begin{proof}[Proof of \rf{eqInvarianceUnderbiLipschitzTRIEBELBis}]
Let us write $\Omega=\Omega_1$. We begin by showing \rf{eqInvarianceUnderbiLipschitzTRIEBELBis}. 
{Since bounded Lipschitz domains are uniform, by Theorem \ref{theoDifference}, it is enough to check that 
\begin{equation}\label{eqQuantifyComposition}
\circled{1}:=\int_\Omega \left(\int_0^1 \frac{\left( \int_{B(x,t)\cap\Omega} |\nabla^k (g\circ f)(x)-\nabla^k (g\circ f)(y)| dy\right)^q}{t^{(\sigma+d)q}}\frac{dt}{t} \right)^\frac pq dx  \leq C_{f}^p\norm{g}_{\mathbf{F}^{s}_{p,q}(\Omega_2)}^p,
\end{equation}
{where $\sigma=s-k$.}

The case ${k}=0$ follows using $f$ as a bi-Lipschitz change of variables in the integrals in $x$ and $y$:
\begin{align*}
\circled{1} 
	&  \leq C_f \int_{f(\Omega)} \left(\int_0^1 \frac{\left( \int_{f(B(x,t)\cap\Omega)} |g(x)- g(y)| dy\right)^q}{t^{(\sigma+d)q}}\frac{dt}{t} \right)^\frac pq dx \\
	&  \leq C_f \int_{\Omega_2} \left(\int_0^1 \frac{\left( \int_{B(x,C_f t)\cap\Omega_2} |g(x)- g(y)| dy\right)^q}{t^{(\sigma+d)q}}\frac{dt}{t} \right)^\frac pq dx.
\end{align*}
The right hand-side expression corresponds to the norm of a rescaling of $g$ in a rescaled domain, and therefore it is bounded by $C_{d,p,q,\sigma}C_{f}\norm{g}_{\mathbf{F}^{s}_{p,q}(\Omega_2)}^p$. 

Let us assume $ k \geq 1$. Note that $\mathbf{F}^s_{p,q}(\Omega_j)\subset \mathbf{W}^{s,p}(\Omega_j)$, and by} Lemma \ref{lemSobolevAdmissibleSpace} we can use the chain rule 
almost everywhere and in particular \rf{eqGeneralizedChainRuleDifferences} applies.  However, after considering \rf{eqGeneralizedChainRuleDifferences}, the reader will note that there are functions on $x$ and functions on $y$ in the integrand, and this is an obstruction for using H\"older inequalities as it was done in the previous proof. Instead, we need to write all the functions depending on $x$ and then address all the terms appearing in a telescopic summation. To keep the notation compact, we write $\Delta_h g(x):=g(x+h)-g(x)$ for $h=y-x$ in \rf{eqGeneralizedChainRuleDifferences}. It follows that $\Delta_h(g_1g_2)=\Delta_h g_1\Delta_hg_2+\Delta_h g_1 g_2+g_1\Delta_hg_2$ and, by induction,

\begin{equation}\label{eqDifferencesProduct}
\prod_{i=1}^\ell g_i (x+h) =  \sum_{\nu \in \{0,1\}^{\ell}} \prod_{r\leq \ell: \nu_r=1} \Delta_h g_r (x)  \prod_{e\leq \ell: \nu_e=0} g_e(x).
\end{equation} 
Combining with \rf{eqGeneralizedChainRuleDifferences}  we get
\begin{align}\label{eqGeneralizedChainRuleDifferencesLonger}
 |\Delta_h &\nabla^k (g\circ f) (x)|\\
\nonumber	& \lesssim  \sum_{1\leq i\leq k} |\Delta_h [( \nabla^{i} g)\circ f](x)| \sum_{\alpha\in\N^i: |\alpha|=k} \prod_{j=1}^{i} | \nabla^{\alpha_j} f (x)| \\
\nonumber	&  + \sum_{1\leq i\leq k} | \nabla^i g(f(x))| \sum_{\substack{\alpha\in\N^i\\ |\alpha|=k} } \sum_{\ell=1}^{i} \sum_{\substack{\nu \in \{0,1\}^{i} \\ \nu_\ell=1 \\ \nu_r=0   \,\forall r>\ell}} \prod_{r\leq i: \nu_r=1}|\Delta_h (\nabla^{\alpha_r} f) (x) | \prod_{e\leq i: \nu_e=0}
|\nabla^{\alpha_e} f (x)|\\
\nonumber	&  + \sum_{1\leq i\leq k} |\Delta_h [( \nabla^{i} g)\circ f](x)|  \sum_{\substack{\alpha\in\N^i\\ |\alpha|=k} } \sum_{\ell=1}^{i} \sum_{\substack{\nu \in \{0,1\}^{i} \\ \nu_\ell=1 \\ \nu_r=0 \,\forall r>\ell}} \prod_{r\leq i: \nu_r=1}|\Delta_h (\nabla^{\alpha_r} f) (x) | \prod_{e\leq i: \nu_e=0}|\nabla^{\alpha_e} f (x)|.
\end{align}

Plugging this decomposition in the numerator of the integrand in \rf{eqQuantifyComposition} {and applying the triangle inequality, we get an expression with several terms which we summarize as follows:}
 \begin{align}\label{eqBreakTheBesovNormOfComposition}
\circled{1}
	& \lesssim \sum_{1\leq i\leq k}\sum_{\substack{\alpha\in\N^i\\ |\alpha|=k} }  \squared{$2i\alpha$}+ \sum_{1\leq i\leq k} \sum_{\substack{\alpha\in\N^i\\ |\alpha|=k} } \sum_{\ell=1}^{i} \sum_{\substack{\nu \in \{0,1\}^{i} \\ \nu_\ell=1 \\ \nu_r=0   \,\forall r>\ell}} \left(\squared{$3i\alpha\ell\nu$}+\squared{$4i\alpha\ell\nu$}\right).
\end{align}
{The precise definition of each summand is what the reader may expect, and will be provided in due time. To clarify concepts, one can understand for instance $\squared{$4i\alpha\ell\nu$}$ as the term of \emph{fourth kind} which depends on particular values of $i$, $\alpha$, $\ell$ and $\nu$, which are indices with the restrictions specified in the summation. Note that} the coefficients $\alpha_j$ are all strictly positive natural numbers.
Regarding the first term, we have
\begin{align*}
\squared{$2i\alpha$} 
	& = \int_\Omega \left(\int_0^1 \frac{ \left( \int_{(B(x,t)\cap\Omega)-x} |\Delta_h [( \nabla^{i} g)\circ f](x)| dh\right)^q }{t^{(\sigma+d)q}}\frac{dt}{t} \right)^\frac pq \prod_{j=1}^{i} | \nabla^{\alpha_j} f (x)|^p dx \\
	& \lesssim \norm{\nabla f^{-1}}_\infty^{\frac {dp}{p_0}+ {dp}} \norm{\nabla f}_\infty^{\left(\sigma+d\right)p} \norm{\nabla^{i} g}_{F^{\sigma}_{p_{0},q}(\Omega_2)}^p \prod_{j=1}^{i} \norm{\nabla^{\alpha_j}  f}_{L^{p_{j}}(\Omega_1)}^p,
\end{align*}
where $\sum_0^i \frac{1}{p_{j}}=\frac1p$. Note that we have used $f$ as a bi-Lipschitz change of variables to obtain the $F^\sigma_{p_0,q}$ norm in the right-hand side of the last inequality above. Letting ${p_0}=\frac{p(s-1)}{i+\sigma-1}$ and ${p_j}=\frac{p(s-1)}{\alpha_j-1}$ so that we can apply Lemma \ref{lemReadyForHolderInequality}, and noting that $\sum_1^i  \alpha_j-1=k-i$ and $\sum_1^i s-\alpha_j=i s-k$, we get
\begin{align*}
\squared{$2i\alpha$}^\frac1p
	& \lesssim \norm{\nabla f^{-1}}_\infty^{\frac d{p_0}+d} \norm{\nabla f}_\infty^{d}  \norm{g}_{F^{s}_{p,q}(\Omega_2)}^\frac{i+\sigma-1}{s-1}\norm{\nabla g}_\infty^\frac{s-i-\sigma}{s-1} \norm{f}_{F^{s}_{p,q}(\Omega_1)}^{\sum_1^i  \frac{\alpha_j-1}{s-1}} \norm{\nabla f}_\infty^{\sigma + \sum_1^i \frac{s-\alpha_j}{s-1}}\\
	&= C_{f}  \left(\norm{g}_{F^{s}_{p,q}(\Omega_2)}\norm{\nabla f}_\infty^{s}\right)^\frac{i+\sigma-1}{s-1}\left(\norm{\nabla g}_\infty\norm{f}_{F^{s}_{p,q}(\Omega_1)}\right)^\frac{k-i}{s-1} \\
	&\leq C_{f} \left( \norm{g}_{F^{s}_{p,q}(\Omega_2)}\norm{\nabla f}_\infty^{s} + \norm{\nabla g}_\infty\norm{f}_{F^{s}_{p,q}(\Omega_1)}\right),
\end{align*}
where $C_{f}=\norm{\nabla f^{-1}}_\infty^{\frac d{p_0}+d} \norm{\nabla f}_\infty^{d} $ depends only on the bi-Lipschitz character of $f$.

For the second term in the right-hand side of \rf{eqBreakTheBesovNormOfComposition}, we need to apply H\"older inequality three times, once for each variable. Namely, writing $U_{x}^t:=B(x,t)\cap\Omega-x$, 
\begin{align*}
\squared{$3i\alpha\ell\nu$} 
	& = \int_\Omega \left(\int_0^1  \frac{\left( \int_{U_{x}^t}\prod_{r\leq i: \nu_r=1}|\Delta_h (\nabla^{\alpha_r} f) (x) | dh\right)^q}{t^{(\sigma+d)q}} \frac{dt}{t}\right)^\frac pq | \nabla^i g(f(x))|^p \prod_{e\leq i: \nu_e=0} |\nabla^{\alpha_e} f (x)|^p dx\\
	& \leq \int_\Omega \prod_{r\leq i: \nu_r=1} \left(\int_0^1  \frac{\left( \int_{U_{x}^t}|\Delta_h (\nabla^{\alpha_r} f) (x) |^{u_r} dh\right)^\frac{q_r}{u_r} }{t^{(\sigma+d)q}} \frac{dt}{t}\right)^\frac {p}{q_r} | \nabla^i g(f(x))|^p \prod_{e\leq i: \nu_e=0} |\nabla^{\alpha_e} f (x)|^p dx
\end{align*}
where we assume that $\sum_r \frac1{u_r}=\sum_r \frac q{q_r}=1$. In particular, let us fix $q_r:=u_r q$ so that  $(\sigma+d)q=(\frac{\sigma q}{q_r}+\frac{d}{u_r})q_r$. Take also $\sum_0^i \frac{1}{p_{j}}=\frac1p$ and apply H\"older's inequality again to get
\begin{align*}
\squared{$3i\alpha\ell\nu$} 
	& \lesssim \norm{\nabla f^{-1}}_\infty^{\frac {dp}{p_0}} \norm{\nabla^i g}_{L^{p_{0}}(\Omega_2)}^p \prod_{r\leq i: \nu_r=1} \norm{\nabla^{\alpha_r} f}_{F^{\sigma q/q_r}_{p_r,q_r}(\Omega_1)}^p \prod_{e\leq i: \nu_e=0} \norm{\nabla^{\alpha_e} f }_{L^{p_{e}}(\Omega_1)}^p, 
\end{align*}
as long as $1\leq u_r\leq \min\{p_r,q_r\} \frac{d+\sigma}{d}$. 

The fact that $u_r\leq q_r$ is clear from the definition of $q_r$. 
Let us write $M=\sum_{r:\nu_r=1}(\alpha_r-1)$ and define $u_r:= \frac{M}{\alpha_r - 1}$, $p_0:= \frac{p(s-1)}{i-1}$, $p_r=\frac{p(s-1)}{\alpha_r+\sigma/u_r -1}$ and $p_e=\frac{p(s-1)}{\alpha_e-1}$. Note that $\sum \frac1{u_r}=1$ and  $1\leq u_r$ trivially, while the condition $u_r \leq p_r$ is equivalent to $u_r(\alpha_r-1)\leq p(s-1) -\sigma$, that is, equivalent to $M\leq p(s-1) -\sigma$. But $M\leq |\alpha|-1= k-1=s-1-\sigma\leq p(s-1) -\sigma$, and thus it follows that $u_r\leq p_r$.

Thus, we can apply Lemma \ref{lemReadyForHolderInequality} again to get
\begin{align*}
\squared{$3i\alpha\ell\nu$}^\frac1p
	& \lesssim C_{f} \norm{g}_{F^{s}_{p,q}(\Omega_2)}^\frac{i-1}{s-1}\norm{\nabla g}_\infty^\frac{s-i}{s-1} \norm{f}_{F^{s}_{p,q}(\Omega_1)}^{\sum_r  \frac{\alpha_r+\sigma/u_r-1}{s-1}+\sum_e  \frac{\alpha_e-1}{s-1}} \norm{\nabla f}_\infty^{ \sum_r \frac{s-\alpha_r-\sigma/u_r}{s-1}+\sum_e \frac{s-\alpha_e}{s-1}}\\
	& \leq C_{f} \left( \norm{g}_{F^{s}_{p,q}(\Omega_2)}\norm{\nabla f}_\infty^{s} + \norm{\nabla g}_\infty\norm{f}_{F^{s}_{p,q}(\Omega_1)}\right),
\end{align*}
where $C_{f}= \norm{\nabla f^{-1}}_\infty^{\frac d{p_0}} $ depends only on the bi-Lipschitz character of $f$.

For the last term in the right-hand side of \rf{eqGeneralizedChainRuleDifferencesLonger}, we argue analogously to get
\begin{align*}
\squared{$4i\alpha\ell\nu$} 
	& =  \int_\Omega \left(\int_0^1  \frac{\left( \int_{U_{x}^t} |\Delta_h [( \nabla^{i} g)\circ f](x)| \prod_{\substack{r\leq i\\ \nu_r=1}}|\Delta_h (\nabla^{\alpha_r} f) (x) | dh\right)^q}{t^{(\sigma+d)q}} \frac{dt}{t}\right)^\frac pq  \prod_{\substack{e\leq i\\ \nu_e=0}} |\nabla^{\alpha_e} f (x)|^p dx\\
	& \leq  \int_\Omega \left(\int_0^1  \frac{\left( \int_{U_{x}^t} |\Delta_h [( \nabla^{i} g)\circ f](x)|^{u_0} dh\right)^\frac{q_0}{u_0}}{t^{(\sigma+d)q}} \frac{dt}{t}\right)^\frac p{q_0} \\
	&\quad  \cdot\left(\int_0^1  \frac{\left( \int_{U_{x}^t}\prod_{\substack{r\leq i\\ \nu_r=1}}|\Delta_h (\nabla^{\alpha_r} f) (x) |^{u_r} dh\right)^\frac{q_r}{u_r}}{t^{(\sigma+d)q}} \frac{dt}{t}\right)^\frac p{q_r}  \prod_{\substack{e\leq i\\ \nu_e=0}} |\nabla^{\alpha_e} f (x)|^p dx,
\end{align*}
where we assume
$\frac1{u_0}+\sum_r \frac1{u_r}=\frac{q}{q_0}+\sum_r \frac q{q_r}=1$. In particular, let us fix $q_0:=u_0 q$ and $q_r:=u_r q$ so that $(\sigma+d)q=(\frac{\sigma q}{q_r}+\frac{d}{u_r})q_r$. Take also $\sum_0^i \frac{1}{p_{j}}=\frac1p$ and apply H\"older's inequality again to get
\begin{align*}
\squared{$4i\alpha\ell\nu$} 
	& \lesssim \norm{\nabla f^{-1}}_\infty^{\frac {dp}{p_0}+\frac {dp}{u_0}} \norm{\nabla f}_\infty^{\frac{(\sigma +d) p}{u_0}} \norm{\nabla^i g}_{F^{\sigma q/q_0}_{p_0,q_0}(\Omega_2)}^p \prod_{\substack{r\leq i\\ \nu_r=1}} \norm{\nabla^{\alpha_r} f}_{F^{\sigma q/q_r}_{p_r,q_r}(\Omega_1)}^p \prod_{\substack{e\leq i\\ \nu_e=0}} \norm{\nabla^{\alpha_e} f }_{L^{p_{e}}(\Omega_1)}^p, 
\end{align*}
as long as $1\leq u_r\leq \min\{p_r,q_r\}  \frac{d+\sigma}{d}$. 

The fact that $u_r\leq q_r$ is clear from the definition of $q_r$. 
Let us write $M=\sum_r(\alpha_r-1)$ and define $u_0:=\frac{M+i-1}{i-1}$, $u_r:= \frac{M+i-1}{\alpha_r - 1}$, $p_0:= \frac{p(s-1)}{i+\sigma/u_0-1}$, $p_r=\frac{p(s-1)}{\alpha_r+\sigma/u_r -1}$ and $p_e=\frac{p(s-1)}{\alpha_e-1}$. Note that $\sum \frac1{u_r}=1$ and  $1\leq u_r$ trivially, while the condition $u_r \leq p_r$ is equivalent to $M+i-1\leq p(s-1) -\sigma$. But 
$$M+i-1\leq \sum_{j=1}^i(\alpha_j-1)+i-1= |\alpha|-1= k-1=s-1-\sigma\leq p(s-1) -\sigma,$$ and thus it follows that $u_r\leq p_r$.

Thus, we can apply Lemma \ref{lemReadyForHolderInequality} again to get
\begin{align*}
\squared{$4i\alpha\ell\nu$}^\frac1p
	& \lesssim C_{f} \norm{g}_{F^{s}_{p,q}(\Omega_2)}^\frac{i+\sigma/u_0-1}{s-1}\norm{\nabla g}_\infty^\frac{s-i-\sigma/u_0}{s-1} \norm{f}_{F^{s}_{p,q}(\Omega_1)}^{\sum_r  \frac{\alpha_r+\sigma/u_r-1}{s-1}+\sum_e  \frac{\alpha_e-1}{s-1}} \norm{\nabla f}_\infty^{ \frac\sigma {u_0}+\sum_r \frac{s-\alpha_r-\sigma/u_r}{s-1}+\sum_e \frac{s-\alpha_e}{s-1}}\\
	& \leq C_{f} \left( \norm{g}_{F^{s}_{p,q}(\Omega_2)}\norm{\nabla f}_\infty^{s} + \norm{\nabla g}_\infty\norm{f}_{F^{s}_{p,q}(\Omega_1)}\right),
\end{align*}
where $C_{f}=\norm{\nabla f^{-1}}_\infty^{\frac d{p_0}+\frac d{u_0}}  \norm{\nabla f}_\infty^{\frac {d}{u_0}}$ depends only on the bi-Lipschitz character of $f$.

Combining these estimates with \rf{eqBreakTheBesovNormOfComposition}, we obtain \rf{eqInvarianceUnderbiLipschitzTRIEBELBis}. In particular, we obtain \rf{eqQuantifyComposition}, where the constant $C_{f}$ is affine with respect to the  $F^{s}_{p,q}(\Omega)$ norm of $f$ and depends polynomially on its bi-Lipschitz constants, as well as on $d$, $s$, $p$, $q$ and the extension constants of the domains for all the different indices $p_j,q_j,u_j$ appearing in the proof.  
\end{proof}

\begin{figure}[t]
\caption{Composition rule for Sobolev and Triebel-Lizorkin scales of a bounded continuous function $g$ and a bi-Lipschitz function $f$ in {Lemmata} \ref{lemSobolevAdmissibleSpace} and \ref{theoTriebelAdmissibleBanachBisBisBis}.}\label{figComposition}

\center
\begin{tikzpicture}[line cap=round,line join=round,>=triangle 45,x=3.8cm,y=0.7cm]
\clip(-0.41,-0.9277074696345513) rectangle (1.45,5.4);
\draw [line width=.5pt,color=grisfosc] (0.,0.) -- (0.,5.203518980893238);
\draw [line width=.5pt,color=grisfosc] (1.,0.) -- (1.,5.203518980893238);
\draw [line width=.5pt,color=grisfosc] (0.,5.)-- (1.,5.);
\draw [line width=.5pt,color=grisfosc] (0.,3.)-- (1.,3.);
\draw [line width=.5pt,color=grisfosc] (0.,4.)-- (1.,4.);
\draw [line width=.5pt,color=grisfosc] (0.,2.)-- (1.,2.);
\draw [line width=.5pt,color=grisfosc] (0.,1.)-- (1.,1.);
\draw [line width=.5pt,color=grisfosc] (0.,0.)-- (1.,0.);
\begin{scriptsize}
\draw (-0.13,5.3) node[anchor=north west] {$s$};
\draw (0.94,-0.1) node[anchor=north west] {$\frac11$};
\draw (-0.06,-0.1) node[anchor=north west] {$\frac1\infty$};
\draw (-0.12,0.45) node[anchor=north west] {$C^0$};
\draw (-0.12,1.45) node[anchor=north west] {$C^1$};
\draw [fill=uuuuuu] (0.,1.) circle (1.0pt);
\draw [fill=uuuuuu] (0.,0.) circle (1.0pt);
\end{scriptsize}
\draw [line width=0.5pt,dotted,color=cqcqcq] (0.3441557837739109,0.) -- (0.3441557837739109,5.203518980893241);
\draw [line width=0.5pt,dotted,color=cqcqcq] (0.45437183207571896,0.) -- (0.45437183207571896,5.203518980893241);
\draw [line width=0.5pt,dotted,color=cqcqcq] (0.6113256095254408,0.) -- (0.6113256095254408,5.203518980893241);
\draw [line width=0.5pt,dotted,color=cqcqcq] (0.9169884142881612,0.) -- (0.9169884142881612,5.203518980893241);
\begin{scriptsize}
\draw (0.22438999981632807,0) node[anchor=north west] {$\frac{s_2-1}{s_2p_2}$};
\draw (0.38,-.08) node[anchor=north west] {$\frac1{p_2}$};
\draw (0.525339918478971,0) node[anchor=north west] {$\frac{s_1-1}{s_1p_1}$};
\draw (0.85,-.08) node[anchor=north west] {$\frac1{p_1}$};
\draw [color=uuuuuu] (0.3441557837739109,0.)-- ++(-1.5pt,-1.5pt) -- ++(3.0pt,3.0pt) ++(-3.0pt,0) -- ++(3.0pt,-3.0pt);
\draw [color=uuuuuu] (0.6113256095254408,0.)-- ++(-1.5pt,-1.5pt) -- ++(3.0pt,3.0pt) ++(-3.0pt,0) -- ++(3.0pt,-3.0pt);
\draw [color=uuuuuu] (0.45437183207571896,0.)-- ++(-1.5pt,-1.5pt) -- ++(3.0pt,3.0pt) ++(-3.0pt,0) -- ++(3.0pt,-3.0pt);
\draw [color=uuuuuu] (0.9169884142881612,0.)-- ++(-1.5pt,-1.5pt) -- ++(3.0pt,3.0pt) ++(-3.0pt,0) -- ++(3.0pt,-3.0pt);
\end{scriptsize}

\draw [line width=.5pt,color=grisfosc]  (0.3441557837739109,4.122556007737622)-- (0.,1.);
\draw [line width=.5pt,color=grisfosc]  (0.45437183207571896,4.122556007737622)-- (0.,0.);
\draw [line width=.5pt,color=grisfosc]  (0.3441557837739109,4.122556007737622)-- (0.45437183207571896,4.122556007737622);
\begin{scriptsize}
\draw [fill=uuuuuu] (0.3441557837739109,4.122556007737622) circle (1.0pt);
\draw [fill=uuuuuu] (0.45437183207571896,4.122556007737622) circle (1.0pt);
\draw (0.4669924852688668,4.687604834614424) node[anchor=north west] {$F^{s_2}_{p_2,q}$};
\draw (0.16, 5) node[anchor=north west] {$F^{s_2}_{\frac{s_2p_2}{s_2-1},q}$};
\end{scriptsize}

\draw [line width=.5pt,color=grisfosc]  (0.,1.)-- (0.6113256095254408,3.);
\draw [line width=.5pt,color=grisfosc]  (0.,0.)-- (0.9169884142881612,3.);
\draw [line width=.5pt,color=grisfosc]  (0.6113256095254408,3.)-- (0.9169884142881612,3.);
\begin{scriptsize}
\draw [fill=uuuuuu] (0.6113256095254408,3.) circle (1.0pt);
\draw [fill=uuuuuu] (0.9169884142881612,3.) circle (1.0pt);
\draw (0.9,3.5206561704123396) node[anchor=north west] {$W^{s_1,p_1}$};
\draw (0.5,3.8) node[anchor=north west] {$W^{s_1,\frac{s_1p_1}{s_1-1}}$};
\end{scriptsize}
\end{tikzpicture}
\end{figure}

\begin{lemma}\label{theoTriebelAdmissibleBanachBisBisBis}
Let $0<s<\infty$, $s\notin \N$, let $ 1\leq p < \infty$,  $ 1\leq q \leq \infty$ and $d\in\N$. Given bounded Lipschitz domains $\Omega_j\subset \R^d$ and a bi-Lipschitz function $f$ with $f(\Omega_1)= \Omega_2$, then 
\begin{align}\label{eqInvarianceUnderbiLipschitzTRIEBELBisBis}
 f \in \mathbf{F}^{s}_{\frac{ps}{s-1},q}(\Omega_1)\mbox{ and } g \in F^{s}_{p,q}\cap L^\infty(\Omega_2) &\implies   g \circ f \in F^{s}_{p,q}(\Omega_1)
\end{align}
(see Figure \ref{figComposition})
\end{lemma}

\begin{proof}
The proof is just a modification of the proof of \rf{eqInvarianceUnderbiLipschitzTRIEBELBis}. One has to set ${p_0}=\frac{ps}{i+\sigma}$ and ${p_j}=\frac{ps}{\alpha_j-1}$ in \squared{$2i\alpha$}, ${p_0}=\frac{ps}{i}$, ${p_r}=\frac{ps}{\alpha_r+\sigma/u_r-1}$, ${p_e}=\frac{ps}{\alpha_e-1}$ and $u_r=\frac{M}{\alpha_r-1}$ in 
\squared{$3i\alpha\ell\nu$} and ${p_0}=\frac{ps}{i+\sigma/u_0}$, ${p_r}=\frac{ps}{\alpha_r+\sigma/u_r-1}$, ${p_e}=\frac{ps}{\alpha_e-1}$,  $u_0=\frac{M+i}{i}$ and $u_r=\frac{M+i}{\alpha_r-1}$ in \squared{$4i\alpha\ell\nu$}. We leave the details to the reader.
\end{proof}

\begin{remark}The precise dependence obtained in the preceeding proofs is 
$$\norm{g \circ f }_{F^{s}_{p,q}(\Omega_1)} {\lesssim_{d,p,q,s}} C_{f} \left( \norm{g}_{F^{s}_{p,q}(\Omega_2)}\norm{\nabla f}_\infty^{s} + \norm{\nabla g}_\infty\norm{f}_{F^{s}_{p,q}(\Omega_1)}\right),$$
where
$$C_{f}=\left(1+\norm{\nabla f^{-1}}_\infty^{\frac dp} \right)\left(1+\norm{\nabla f^{-1}}_\infty^d \norm{\nabla f}_\infty^{d}\right),$$
and
$$\norm{g \circ f }_{F^{s}_{p,q}(\Omega_1)} {\lesssim_{d,p,q,s}} C_{f} \left( \norm{g}_{F^{s}_{p,q}(\Omega_2)}\left(\norm{\nabla f}_\infty^\frac s{s-1}\right)^{s} + \norm{g}_\infty \norm{f}_{F^{s}_{\frac{ps}{s-1},q}(\Omega_1)}^\frac{s}{s-1}\right),$$
where
$$C_{f}=\norm{\nabla f}_\infty^\frac{-s}{s-1}\left(\norm{\nabla f^{-1}}_\infty^{\frac d{ps}}+\norm{\nabla f^{-1}}_\infty^{\frac dp} \right)\left(1+\norm{\nabla f^{-1}}_\infty^d \norm{\nabla f}_\infty^{d}\right).$$

\end{remark}

\begin{proof}[Proof of \rf{eqInvarianceUnderInversionTRIEBELBis}]
Inequality \rf{eqInvarianceUnderInversionTRIEBELBis} is proven by analogous techniques using \rf{eqInverseDerivatives} and \rf{eqDerivativesInverse} which apply by Lemma \ref{lemSobolevAdmissibleSpace}. We claim that it is enough to check that $g_{ij}\in F^{s-1}_{p,q}(\Omega_1)$. Indeed, in case $1<s<2$, then we have that $f^{-1}$ is a bi-Lipschitz change of variables and, therefore, $g_{ij}\circ (f^{-1})\in F^{s-1}_{p,q}(\Omega_2)$ if and only if $g_{ij}\in F^{s-1}_{p,q}(\Omega_1)$.
Otherwise, by Lemma \ref{lemReadyForHolderInequality} we have that  $f\in F^{s-1}_{\frac{p(s-1)}{s-2},q}(\Omega_1)$. Inductively we can assume that $f^{-1}\in F^{s-1}_{\frac{p(s-1)}{s-2},q}(\Omega_1)$, and by \rf{eqInvarianceUnderbiLipschitzTRIEBELBisBis}, if $g_{ij}\in F^{s-1}_{p,q}(\Omega_1)$ then we get $g_{ij}\circ (f^{-1})\in F^{s-1}_{p,q}(\Omega_2)$ and  the claim follows.

Now, we want to prove
\begin{equation}\label{eqQuantifyInversion}
\circled{1}:=\left(\int_{\Omega_2} \left(\int_0^1 \frac{\left( \int_{U_{x,t}} |\nabla^{k-1} g_{ij}(x)-\nabla^{k-1} g_{ij}(y)| dy\right)^q}{t^{(\sigma+d)q}}\frac{dt}{t} \right)^\frac pq dx \right)^\frac1p \leq C_{f}\norm{f}_{\mathbf{F}^{s}_{p,q}(\Omega_1)},
\end{equation}
where $U_{x,t}:=B(x,t)\cap\Omega_2$. Again we use first-order differences, and write $h:=y-x$.  For $|\alpha|=k+1$, we have
\begin{align*}
|\Delta_hD^\alpha g_{ij}(x)|
	& \lesssim \sum_{\beta, \gamma,\mu} \left| \frac{\prod_{\ell=1}^{k-1} D^{\gamma_\ell} f_{\mu_\ell}(x)}{\det(Df)^k(x)} \Delta_h(Df)^{\beta}(x) \right| \\
	& \quad\quad\quad +\left| (Df)^{\beta}(x+h) \prod_{\ell=1}^{k-1} D^{\gamma_\ell} f_{\mu_\ell}(x) \Delta_h \left(\frac{1}{\det(Df)^k}\right)(x)\right|\\
	& \quad\quad\quad+ \left|\frac{(Df)^{\beta}(x+h)}{\det(Df)^k(x+h)} \Delta_h\left(\prod_{\ell=1}^{k-1}  D^{\gamma_\ell} f_{\mu_\ell} (x)\right)\right|
\end{align*}
Now we use some trivial properties of first order differences, together with \rf{eqDifferencesProduct} to get
\begin{align*}
|\Delta_hD^\alpha g_{ij}|
	&\lesssim \sum_{\gamma,\mu}\norm{\nabla f}_\infty^{(d-1)k-1}\norm{\nabla f^{-1}}_\infty^{dk} |\Delta_h(D f)| \prod_{\ell=1}^{k-1} |D^{\gamma_\ell} f_{\mu_\ell}|   \\
	& \quad + \norm{\nabla f}_\infty^{(d-1)k}\norm{\nabla f^{-1}}_\infty^{d(k+1)}| \Delta_h \det(D f) | \prod_{\ell=1}^{k-1} |D^{\gamma_\ell} f_{\mu_\ell}|\\
	& \quad+ \norm{\nabla f}_\infty^{(d-1)k}\norm{\nabla f^{-1}}_\infty^{dk} \sum_{\nu\in\{0,1\}^{k-1}:|\nu|\geq 1} \prod_{r\leq k-1:\nu_r=1} |\Delta_hD^{\gamma_r} f_{\mu_r} |\prod_{e\leq k-1:\nu_e=0} |D^{\gamma_e} f_{\mu_e} |.
\end{align*}
To end, note that $|\Delta_h \det(D f)|\leq c \norm{\nabla f}_\infty^{d-1}|\Delta_h(Df)|$, so
\begin{align*}
|\Delta_hD^\alpha g_{ij}|
	&\lesssim \sum_{\gamma,\mu} \left(\norm{\nabla f}_\infty^{(d-1)k-1}\norm{\nabla f^{-1}}_\infty^{dk} + \norm{\nabla f}_\infty^{(d-1)k+d-1}\norm{\nabla f^{-1}}_\infty^{d(k+1)}\right)|\Delta_h(D f)| \prod_{\ell=1}^{k-1} |\nabla^{|\gamma_\ell|} f|   \\
	& \quad+ \norm{\nabla f}_\infty^{(d-1)k}\norm{\nabla f^{-1}}_\infty^{dk} \sum_{\nu\in\{0,1\}^{k-1}:|\nu|\geq 1} \prod_{r\leq k-1:\nu_r=1} |\Delta_h \nabla^{|\gamma_r|} f |\prod_{e\leq k-1:\nu_e=0} |\nabla^{|\gamma_e|} f |.
\end{align*}

Therefore, we write
\begin{align*}
\circled{1}
	& \lesssim C_f \sum_{\substack{ \gamma\in(\N_0^d)^{k-1}\\ |\gamma_\ell|\geq 1\, \&\,\sum|\gamma_\ell|=2k-2}} \left( \left(1+\norm{\nabla f}_\infty^{d}\norm{\nabla f^{-1}}_\infty^{d} \right) \squared{$2\gamma$}+\norm{\nabla f}_\infty\sum_{\nu\in\{0,1\}^{k-1}:|\nu|\geq 1}\squared{$3\gamma\nu$}\right),
\end{align*}
with $C_f=\norm{\nabla f}_\infty^{(d-1)k-1}\norm{\nabla f^{-1}}_\infty^{dk} $, with  \squared{$2\gamma$} and \squared{$3\gamma\nu$} as defined below.

Regarding the first term,  by H\"older's inequality we have
\begin{align*}
\squared{$2\gamma$} 
	& := \left( \int_{\Omega_2} \left(\int_0^1 \frac{ \left( \int_{U_x^t} |\Delta_h(D f)(x)|  dh\right)^q }{t^{(\sigma+d)q}}\frac{dt}{t} \right)^\frac pq \prod_{\ell=1}^{k-1} |\nabla^{|\gamma_\ell|} f(x)|^p dx \right)^\frac1p\\
	& \lesssim \norm{D f}_{F^\sigma_{p_0,q}(\Omega_1)} \prod_{\ell=1}^{k-1}\norm{\nabla^{|\gamma_\ell|} f}_{L^{p_\ell}(\Omega_1)},
\end{align*}
where $\sum_0^{k-1}\frac1{p_\ell}=\frac1p$ and $U_x^t:=B(x,t)\cap \Omega-x$. In particular choose $p_0=\frac{p(s-1)}{\sigma}$ and $p_\ell=\frac{p(s-1)}{|\gamma_\ell|-1}$. By Lemma \ref{lemReadyForHolderInequality} we get
\begin{align*}
\squared{$2\gamma$} 
	& \lesssim \norm{f}_{F^s_{p,q}(\Omega_1)}^{\frac{\sigma}{s-1}+\sum_\ell\frac{|\gamma_\ell|-1}{s-1}}\norm{\nabla f}_\infty^{\frac{s-1-\sigma}{s-1}+\sum_\ell\frac{s-|\gamma_\ell|}{s-1}} \\
	& =\norm{f}_{F^s_{p,q}(\Omega_1)}^{\frac{\sigma+k-1}{s-1}}\norm{\nabla f}_\infty^{\frac{s-1-\sigma +s(k-1)-(2k-2)}{s-1}} =\norm{f}_{F^s_{p,q}(\Omega_1)}\norm{\nabla f}_\infty^{k-1} 
\end{align*}

On the other hand, by H\"older's inequality again
\begin{align*}
\squared{$3\gamma\nu$} 
	& := \left( \int_{\Omega_2} \left(\int_0^1 \frac{ \left( \int_{U_{x,t}} \prod_{r\leq k-1:\nu_r=1} |\Delta_h \nabla^{|\gamma_r|} f  (x)|dh\right)^q }{t^{(\sigma+d)q}}\frac{dt}{t} \right)^\frac pq \prod_{e\leq k-1:\nu_e=0} |\nabla^{|\gamma_e|} f (x)|^p dx \right)^\frac1p\\
	& \lesssim \prod_r \norm{\nabla^{|\gamma_r|} f}_{F^{\sigma/u_r}_{p_r,q_r}(\Omega_1)} \prod_{e}\norm{\nabla^{|\gamma_\ell|} f}_{L^{p_e}(\Omega_1)},
\end{align*}
where we assume that $\sum_r \frac1{u_r}=\sum_1^{k-1} \frac{p}{p_{j}}=1$, $q_r:=u_r q$, as long as $1\leq u_r\leq \min\{p_r,q_r\} \frac{d+\sigma}{d}$. 

The fact that $u_r\leq q_r$ is clear from the definition of $q_r$. 
Let us write $M=\sum_r(|\gamma_r|-1)$ and define $u_r:= \frac{M}{|\gamma_r| - 1}$, $p_r=\frac{p(s-1)}{|\gamma_r|+\sigma/u_r -1}$ and $p_e=\frac{p(s-1)}{|\gamma_e|-1}$. Note that $\sum \frac1{u_r}=1$ and  $1\leq u_r$ trivially, while the condition $u_r \leq p_r$ is equivalent to $u_r(|\gamma_r|-1)\leq p(s-1) -\sigma$, that is, equivalent to $M\leq p(s-1) -\sigma$. But $M\leq |\gamma|-(k-1)= k-1=s-1-\sigma\leq p(s-1) -\sigma$, and thus it follows that $u_r\leq p_r$.

By Lemma \ref{lemReadyForHolderInequality} we get
\begin{align*}
\squared{$3\gamma\nu$} 
	& \lesssim \norm{f}_{F^s_{p,q}(\Omega_1)}^{\sum_r\frac{|\gamma_r|+\sigma/u_r-1}{s-1}+\sum_e\frac{|\gamma_e|-1}{s-1}}\norm{\nabla f}_\infty^{\sum_r\frac{s-|\gamma_r|-\sigma/u_r}{s-1}+\sum_e\frac{s-|\gamma_e|}{s-1}} \\
	& =\norm{f}_{F^s_{p,q}(\Omega_1)}^{\frac{\sigma+k-1}{s-1}}\norm{\nabla f}_\infty^{\frac{s(k-1)-(2k-2)-\sigma}{s-1}} =\norm{f}_{F^s_{p,q}(\Omega_1)}\norm{\nabla f}_\infty^{k-2} .
\end{align*}
All in all, 
\begin{align*}
\circled{1}
	& \lesssim \norm{\nabla f}_\infty^{dk-2}\norm{\nabla f^{-1}}_\infty^{dk}   \left( \left(1+\norm{\nabla f}_\infty^{d}\norm{\nabla f^{-1}}_\infty^{d} \right) \right)\norm{f}_{F^s_{p,q}(\Omega_1)}.
\end{align*}

\end{proof}


\section{Corkscrew and uniform domains}\label{secUniform}
{Next we provide the necessary background to work with uniform and corkscrew domains. We begin by defining a Whitney covering and then we provide definitions of interior corkscrew domains, chains of cubes and uniform domains. Here, uniform domains are defined in terms of the existence of this chains, which can be understood as discretized \emph{cigar paths} or Harnack chains with a uniform bound on the number of balls of each scale. Finally we define the shadow of a cube, which plays the role of Carleson boxes, and we provide some well-known results relating the Hardy-Littlewood maximal operator to Whitney coverings.

The definitions will be given in terms of dyadic cubes, which is a usual tool in harmonic analysis. We define $\mathcal{D}_0$ as the collection of open cubes with side length $1$ whose vertices have integer coordinates. Then its rescalings $\mathcal{D}_j:=2^{-j}\mathcal{D}_0$ are the cubes of generation $j$, whose side-length is $2^{-j}$, and $\mathcal{D}:=\bigcup_{j\in\Z}\mathcal{D}_j$ is the whole collection of \emph{dyadic cubes}. Note that each generation is formed by disjoint cubes, and the union of all the cubes in a generation is dense in $\R^d$.}

\begin{definition}\label{defWhitney}
Given a domain $\Omega$, we say that a collection of open dyadic cubes $\mathcal{W}$ is a {\rm Whitney covering} of $\Omega$ if they are disjoint, the union of the cubes and their boundaries is $\Omega$, there exists a constant $C_{\mathcal{W}}$ such that 
$$C_\mathcal{W} \ell(Q)\leq \dist(Q, \partial\Omega)\leq 4C_\mathcal{W}\ell(Q),$$
and the family $\{50 Q\}_{Q\in\mathcal{W}}$ has finite superposition {i.e., $\sum_Q \chi_{50 Q}\leq C<\infty$}. Moreover, 
\begin{equation}\label{eqWhitney5}
S\subset 5Q \implies \ell(S)\geq \frac12 \ell(Q) \quad\quad {\text{for every } S,Q\in\mathcal{W}.}
\end{equation}
{We call $Q\in \mathcal{W}$ a \emph{Whitney cube}. We say that two cubes $Q,S\in\mathcal{W}$ are neighbors if $\overline{Q}\cap\overline{S}\neq \emptyset$.}
\end{definition}
The existence of such a covering is granted for any open set different from $\R^d$ and in particular for any domain as long as $C_\mathcal{W}$ is big enough (see \cite[Chapter 1]{SteinPetit} for instance).

\begin{definition}
We say that a domain $\Omega\subset \R^d$ is an \emph{interior (resp. exterior) $(\varepsilon,\delta)$-corkscrew domain} if there is a Whitney covering of $\Omega$ (resp. $\overline{\Omega}^c$) such that given any ball $B(x,r)$ centered at $\partial\Omega$ with $0<r\leq \delta$ there exists a Whitney cube $Q\subset B(x,r)$ such that $\ell(Q)\geq \varepsilon r$, {see Figure \ref{figWhitneyCorkscrew}.}
\end{definition}

\begin{figure}[ht]
 \centering
\begin{subfigure}{0.5 \textwidth}
 \centering{\includegraphics[width=\textwidth]{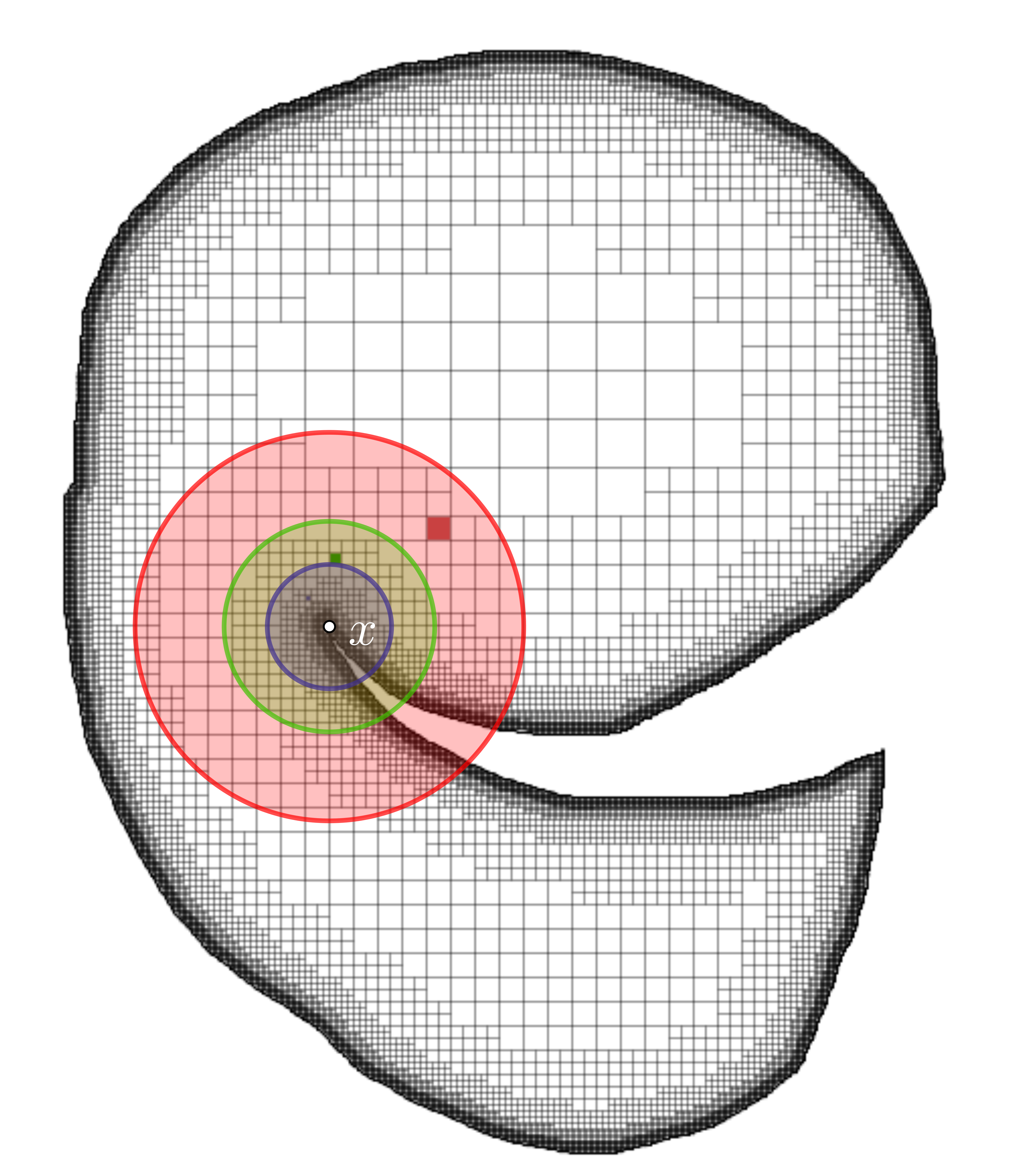}}
\end{subfigure}
\caption{{Interior corkscrew domains may have cusps, in opposition to uniform domains.}}\label{figWhitneyCorkscrew}
\end{figure}

\begin{definition}\label{defEpsilonAdmissible}
Let $\Omega$ be a domain, $\mathcal{W}$ a Whitney decomposition of $\Omega$ and $Q,S\in\mathcal{W}$.  Given $M$ cubes $Q_1,\dots,Q_M\in\mathcal{W}$ with $Q_1=Q$ and $Q_M=S$, the $M$-tuple ${\mathcal{V}}=(Q_1,\dots,Q_M)\in\mathcal{W}^M$  is a {\em chain} connecting $Q$ and $S$ if the cubes $Q_j$ and $Q_{j+1}$ are neighbors for $j<M$. 

Let $\varepsilon\in\R$. We say that the chain ${\mathcal{V}}$ is {\em $\varepsilon$-admissible} if 
\begin{itemize}
\item the \emph{length}  of the chain is bounded by
\begin{equation}\label{eqLengthDistance}
\ell({\mathcal{V}}):=\sum_{j=1}^M\ell(Q_j)\leq \frac1\varepsilon\Dist(Q,S),
\end{equation}
(recall that $\Dist$ stands for the long distance, and $\Dist(Q,S)\approx \ell(Q)+\ell(S)+\dist(Q,S)$)
\item and there exists $j_{\mathcal{V}}<M$ such that the cubes in the chain satisfy
\begin{equation}\label{eqAdmissible1}
\ell(Q_j)\geq\varepsilon \Dist(Q_1,Q_j) \mbox{ for all } j\leq j_{\mathcal{V}} \mbox{\quad\quad  and \quad\quad }
\ell(Q_j)\geq\varepsilon \Dist(Q_j,Q_M) \mbox{ for all } j\geq j_{\mathcal{V}} .
\end{equation}
\end{itemize}
The $j_{\mathcal{V}}$-th cube, which we call \emph{central}, satisfies that $\ell(Q_{j_{\mathcal{V}}})\gtrsim_d \varepsilon \Dist(Q,S)$ by \rf{eqAdmissible1} and the triangle inequality. 

We write (abusing notation) ${\mathcal{V}}$ also for the set $\{Q_j\}_{j=1}^M$. Thus, we will write $P\in {\mathcal{V}}$ if $P$ appears in a coordinate of the $M$-tuple ${\mathcal{V}}$.
\end{definition}

\begin{definition}\label{defUniform}
We say that a domain $\Omega\subset\R^d$ is a {\em uniform domain} if there exists a Whitney covering $\mathcal{W}$ of $\Omega$ and $\varepsilon, \delta \in\R$ such that for any pair of cubes $Q,S \in\mathcal{W}$ with $\Dist(Q,S)\leq \delta$, there exists an $\varepsilon$-admissible chain {connecting $Q$ and $S$ which we denote as} $[Q,S]$, see Figure \ref{figWhitneyChains}. Sometimes we will write {\em $(\varepsilon,\delta)$-uniform domain} to fix the constants. 

{We will write  $Q_S=Q_{j_{[Q,S]}}$ for its central cube.}
\end{definition}
\begin{figure}[ht]
 \centering
\begin{subfigure}{0.45 \textwidth}
 \centering{\includegraphics[width=\textwidth]{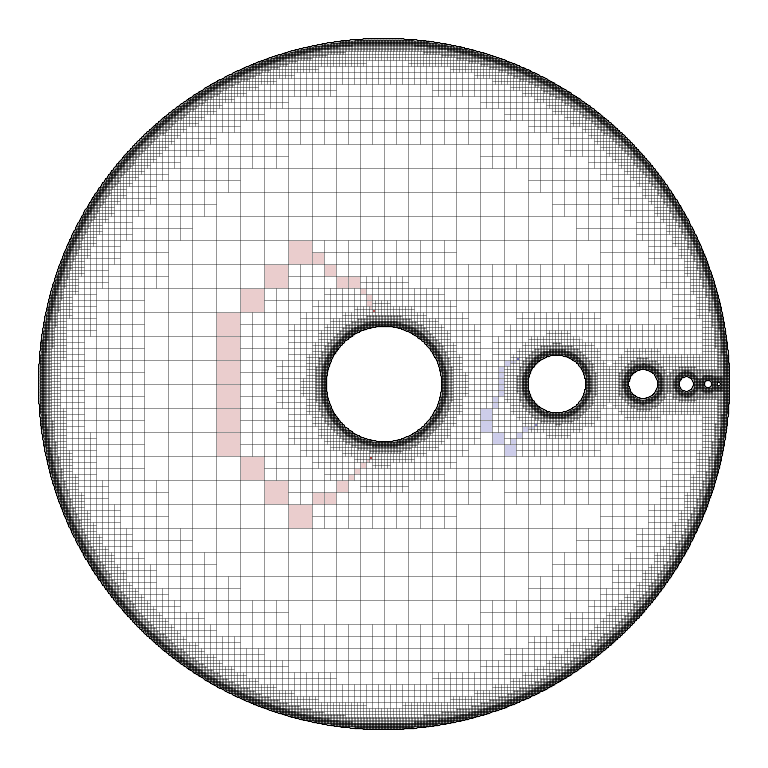}}
\end{subfigure}
\begin{subfigure}{0.45 \textwidth}
 \centering{\includegraphics[width=\textwidth]{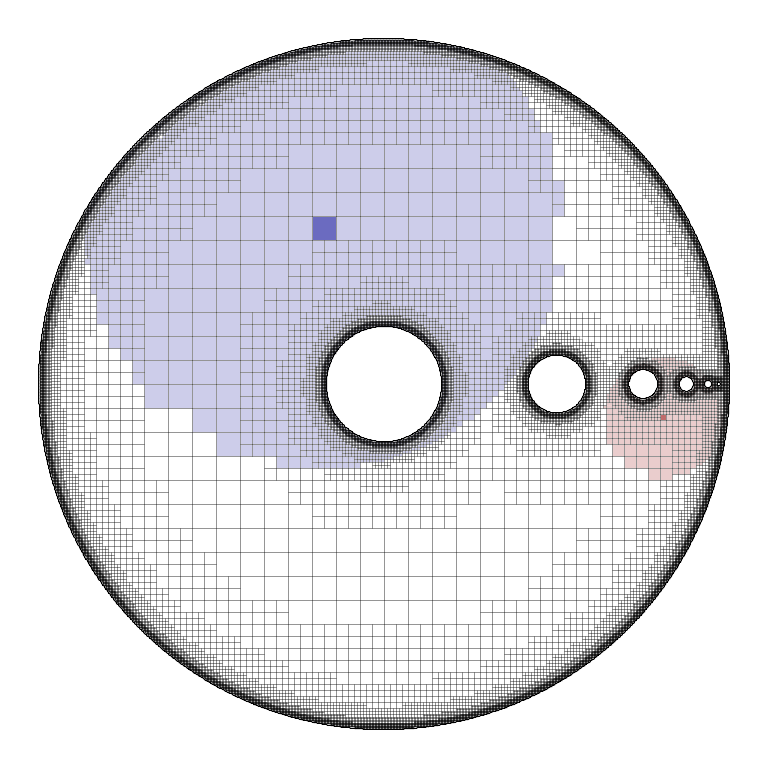}}
\end{subfigure}
\caption{{In uniform domains we can join cubes by admissible chains, which play the role of cigar paths. Shadows are the cubes \emph{under the influence} of a given cube, and they play the role of Carleson boxes.}}\label{figWhitneyChains}
\end{figure}

 {
Consider a uniform domain $\Omega$ with covering $\mathcal{W}$ and two cubes $Q,S\in\mathcal{W}$  with $\Dist(Q,S)\leq \delta$. From Definition \ref{defEpsilonAdmissible} it follows that
\begin{equation}\label{eqAdmissible2}
\Dist(Q,S)\approx_{\varepsilon,d} \ell([Q,S])\approx_{\varepsilon,d} \ell(Q_S).
\end{equation}}

{Note that any bounded Lipschitz domain is an $(\varepsilon, \infty)$-uniform domain, and every uniform domain is also an interior corkscrew domain, perhaps with smaller  parameters. Note also that there may be several admissible chains joining $Q$ and $S$, but we are assuming that a choice is made for every pair of cubes when the domain is uniform, as soon as $Q,S$ are close enough.}
 
Now we can define the shadows:
\begin{definition}\label{defShadow}
Let $\Omega$ be an $(\varepsilon,\delta)$-uniform domain with Whitney covering $\mathcal{W}$. 
Given a cube $P\in\mathcal{W}$ centered at $x_P$ and a real number $\rho$,  the {\em $\rho$-shadow} of $P$ is the collection of cubes
$$\SH_\rho(P)=\{Q\in\mathcal{W}:Q\subset B(x_P,\rho\,\ell(P))\}, $$
and its  {\em ``realization''} is the set
$$\Sh_{\rho}(P)=\bigcup_{Q\in\SH_\rho(P)} Q,$$
{see Figure \ref{figWhitneyChains}.}

By the previous remark and the properties of the Whitney covering, we can define $\rho_\varepsilon>1$ such that the following properties hold:
\begin{itemize}
\item For every $\varepsilon$-admissible chain $[Q,S]$, and every $P\in[Q,Q_S]$ we have that $Q\in\SH_{\rho_\varepsilon}(P)$.
\item Moreover, every cube $P$ belonging to an $\varepsilon$-admissible chain $[Q,S]$ belongs to the shadow $\SH_{\rho_\varepsilon}(Q_S)$.
\end{itemize}
\end{definition}

\begin{remark}[see {\cite[Remark 2.6]{PratsSaksman}}]
\label{remInTheShadow}
Given an $(\varepsilon,\delta)$-uniform domain $\Omega$ we will write $\Sh$ for $\Sh_{\rho_\varepsilon}$. We will write also $\SH$ for $\SH_{\rho_{\varepsilon}}$.

For $Q\in\mathcal{W}$ and $s>0$,  we have that  
\begin{equation}\label{eqAscendingToGlory}
 \sum_{L: Q\in \SH(L)}\ell(L)^{-s} \lesssim \ell(Q)^{-s} \quad\quad\mbox{ and }\quad\quad  \sum_{\substack{L: Q\in \SH(L)\\\ell(L)\leq \rho}}\ell(L)^{s} \lesssim \rho^{s}
 \end{equation}
and, moreover, if $Q\in\SH(P)$ and $\Dist(Q,P)\leq \delta$, then
\begin{equation}\label{eqAscendingPath}
 \sum_{L\in[Q,P]}\ell(L)^{s} \lesssim \ell(P)^s \mbox{\quad\quad and \quad\quad}  \sum_{L\in[Q,P]}\ell(L)^{-s}\lesssim \ell(Q)^{-s} .
 \end{equation}
\end{remark}
Note that the property \rf{eqAscendingToGlory} is not a consequence of uniformity, but of the definition of shadow.

We recall the definition of the non-centered Hardy-Littlewood maximal operator. Given $f\in L^1_{loc}(\R^d)$ and $x\in\R^d$, we define $Mf(x)$ as the supremum of the mean of $f$ in cubes containing $x$, that is,
$$Mf(x)=\sup_{Q:  x\in Q} \frac{1}{|Q|} \int_Q f(y) \, dy.$$
It is a well known fact that this operator is bounded in $L^p$ for $1<p<\infty$.
The following lemma is proven in \cite{PratsTolsa} and will be used repeatedly along the proofs contained in the present text.

\begin{lemma}\label{lemMaximal}
Let $\Omega$ be a  domain with Whitney covering $\mathcal{W}$. Assume that $g\in L^1(\Omega)$ and $r>0$. For every $\eta>0$, $Q\in\mathcal{W}$ and $x\in \R^d$, we have
\begin{enumerate}[1)]
\item The non-local {inequalities} for the maximal operator
	\begin{equation}\label{eqMaximalFar}
	 \int_{|y-x|>r} \frac{g(y) \, dy}{|y-x|^{d+\eta}}\lesssim_d \frac{Mg(x)}{r ^\eta}
\mbox{\quad\quad and \quad\quad}
	 \sum_{S:\Dist(Q,S)>r}  \frac{\int_S g(y) \, dy}{\Dist(Q,S)^{d+\eta}}\lesssim_d \frac{\inf_{y\in Q} Mg(y)}{r ^\eta}.
	 \end{equation}
\item The local {inequalities} for the maximal operator
	\begin{equation}\label{eqMaximalClose}
	 \int_{|y-x|<r} \frac{g(y) \, dy}{|y-x|^{d-\eta}}\lesssim_d r ^\eta Mg(x)
\mbox{\quad\quad and \quad\quad}
	\sum_{S:\Dist(Q,S)<r}  \frac{\int_S g(y) \, dy}{\Dist(Q,S)^{d-\eta}}\lesssim_d \inf_{y\in Q} Mg(y) \,r^\eta.
	 \end{equation}
\item In particular, if $\Omega$ is a uniform domain, we have
	\begin{equation}\label{eqMaximalAllOver}
		\sum_{S\in\mathcal{W}} \frac{\ell(S)^d}{\Dist(Q,S)^{d+\eta}} \lesssim_d \frac{1}{\ell(Q)^\eta}
\mbox{\quad\quad and \quad\quad}
		\sum_{S\in \SH_{{\rho}}(Q)} \ell(S)^{d} \lesssim_{d,\rho} \ell(Q)^d
	\end{equation}
and, by Definition \ref{defShadow},
	\begin{equation}\label{eqMaximalGuay}
	\sum_{S\in\SH_{{\rho}} (Q)} \int_S g(x) \, dx\lesssim_{d,\rho} \inf_{y\in Q} Mg(y) \, \ell(Q)^d.
	 \end{equation}
\end{enumerate}
\end{lemma}

\section{Extension operators}\label{secExtension}
{In this section we prove Theorem \ref{theoDifference}. First we reduce  Theorem \ref{theoDifference} to finding an extension operator. Then, in Section \ref{secCorkscrew} we define this extension operator for smoothness parameter below one in corkscrew domains and then we show that this operator is indeed an extension operator, establishing Theorem \ref{theoDifference} for $s<1$. As a matter of fact we obtain a much better result than in the statement of the theorem, since we can substitute the hypothesis of the domain being uniform by just being interior corkscrew. Finally, in Section \ref{secUniformSmoothness} we establish the remaining cases.

The proof of Theorem \ref{theoExtensionOperator0}, that is, the extension operator when $s<1$, follows a similar structure to \cite[Theorem 1.4]{PratsSaksman}. The novelty here is that we use a norm with means on balls, that allows us to drop the restriction $\sigma>\frac dp-\frac dq$, so we get results with full generality. The use of means in balls limiting the radius to a constant depending on the corkscrew character allows us to avoid the error terms $\circled{a2}$, $\circled{b2}$ and $\circled{c4}$ from \cite{PratsSaksman}. In this way we relate the extended function to the homogeneous norm in all terms except for $\circled{c3}$.

The proof of Theorem \ref{theoExtension}, that is, the boundedness of the  extension operator when $s>1$, follows a similar structure to \cite[Theorem 1.5]{PratsTL}. Again the main difference is the usage of means in balls allowing us to drop the restriction $\sigma>\frac dp-\frac dq$.  We use a similar notation to \cite{PratsTL} so that it is easy for the reader to spot the differences. Here $\squared{0}$ corresponds to $\squared{$\beta$}$ in the aforementioned paper. Using means on balls allows us to drop the error terms $\squared{$\beta.a.2$}$, $\squared{$\beta.b.2$}$ and $\squared{$\beta.c.3$}$ in \cite{PratsTL}. In the present paper, we need to distinguish the local part $\squared{c.1}$ from the non-local part $\squared{c.2}$. Finally, the non-homogeneous contribution is obtained from $\squared{c.3}$, which corresponds to $\squared{$\beta.c.2$}$ in \cite{PratsTL}. Following also the scheme in that paper, we defer the main estimates to a couple of technical lemmata. Lemma \ref{lemControlTotal} is the natural counterpart to \cite[Lemma 3.4]{PratsTL}, and Lemma \ref{lemNormA} is the counterpart to \cite[Lemma 3.5]{PratsTL}. The proof of the first lemma is much easier this time, thanks to a clever reduction to Lemma \ref{lemNormA}. The latter, however, deals with the main technical difficulty in this setting. }

\begin{definition}\label{defAspqU}
Consider $1\leq p<\infty$, $1\leq q \leq \infty$, $1\leq u \leq \infty$ and $0<\sigma<1$ so that $\sigma> \frac{d}{\min\{p,q\}}-\frac{d}{u}$. Let $U$ be an open set in $\R^d$. We say that a locally integrable function $f\in F^{\sigma,\rho}_{p,q,u}(U)$ if
\begin{itemize}
\item The function $f\in L^p(U)$, and
\item the seminorm
\begin{equation}\label{eqSeminormAspq}
\norm{f}_{\dot{F}^{\sigma,\rho}_{p,q,u}(U)}:=\left(\int_U\left(\int_0^\rho \frac{\left(\int_{U_{x,t}} |f(x)-f(y)|^u\right)^\frac qu}{t^{\sigma q+\frac{dq}{u}}} \,\frac{dt}{t} \right)^{\frac{p}{q}}dx\right)^{\frac{1}{p}}
\end{equation}
is finite, where we denote $U_{x,t}:=B(x,t)\cap U$. 
\end{itemize}
We define the norm 
\begin{equation*}
\norm{f}_{{F}^{\sigma,\rho}_{p,q,u}(U)}:=\norm{f}_{L^p(U)}+\norm{f}_{\dot{F}^{\sigma,\rho}_{p,q,u}(U)}.
\end{equation*}
For $s=k+\sigma$ with $k\in\N$, we write
\begin{equation*}
\norm{f}_{{F}^{s,\rho}_{p,q,u}(U)}:=\norm{f}_{W^{k,p}(U)}+\norm{\nabla^k f}_{\dot{F}^{\sigma,\rho}_{p,q,u}(U)}.
\end{equation*}
\end{definition}

In order to prove that $F^s_{p,q}(\Omega)=F^{s,1}_{p,q,u}(\Omega)$ for a given domain $\Omega$, it suffices to find an extension operator $\mathcal{E}:F^{s,\rho}_{p,q,1}(\Omega)\to F^{s,1}_{p,q,1}(\R^d)$ with $\rho<1$. Once this is established, using the equivalence of norms in the ambient space (see \cite[Theorem 1.116]{TriebelTheoryIII}) we obtain the equivalence of norms in the domain by classical arguments:

First note that 
\begin{equation}\label{eqCompareNormsRadius}
\norm{g}_{F^{s}_{p,q}(\R^d)}\approx\norm{g}_{F^{s,1}_{p,q,u}(\R^d)}\approx\norm{g}_{F^{s,\rho}_{p,q,u}(\R^d)}.
\end{equation}
{The} first comparison comes from \cite[Theorem 1.116]{TriebelTheoryIII}. The second can be obtained easily by using the change of variables $\widetilde{x}=\rho^{-1}x$, $\widetilde{t}=\rho^{-1}t$, $\widetilde{y}=\rho^{-1}y$ in the last norm and then compare the norms of $g$ and its rescaling $g(\rho\cdot)$  in $F^s_{p,q}$. 

Thus, 
$$\inf_{g|_\Omega \equiv f} \norm{g}_{F^s_{p,q}(\R^d)}\leq \norm{\mathcal{E}f}_{F^s_{p,q}(\R^d)} \approx \norm{\mathcal{E}f}_{F^{s,1}_{p,q,1}(\R^d)}\lesssim  \norm{f}_{F^{s,\rho}_{p,q,1}(\Omega)}\lesssim  \norm{f}_{F^{s,\rho}_{p,q,u}(\Omega)}\leq \inf_{g|_\Omega \equiv f}\norm{g}_{F^{s,\rho}_{p,q,u}(\R^d)}.$$
Since the first and the last are {comparable} it follows that all the quantities are comparable and, in particular, 
$$F^s_{p,q}(\Omega)=F^{s,\rho}_{p,q,u}(\Omega).$$

To end, since $\rho<1$ then 
$$\norm{f}_{F^{s,\rho}_{p,q,u}(\Omega)}\leq \norm{f}_{F^{s,1}_{p,q,u}(\Omega)}\leq  \norm{\mathcal{E} f}_{F^{s,1}_{p,q,u}(\R^d)}\lesssim \norm{f}_{F^{s,\rho}_{p,q,u}(\Omega)}.$$

\subsection{Corkscrew domains and smoothness below one}\label{secCorkscrew}
Let $\Omega$ be an interior corkscrew domain {and consider given Whitney coverings $\mathcal{W}_0$ of $\Omega$ and $\mathcal{W}_2$ of $\overline{\Omega}^c$. To define the extension operator we will take averages on certain `reflected' cubes, but we need to avoid large cubes, say with side-length bounded by a constant $\ell_0$ which depends on the corkscrew constants. For this reason we  define some related collections of cubes as follows:
\begin{itemize}
\item In case $\Omega$ is unbounded and $\delta=\infty$, choose $\ell_0=1$. Otherwise, fix $\ell_0$ depending on the corkscrew parameters of the domain, so that $c_0\ell_0 << \delta \wedge 1$. Note that $c_0$ is a constant which will be fixed during the proof of Theorem \ref{theoExtension} below, right before \rf{eqBetaA1Bounded}.
\item We define $\mathcal{W}_3$ to be the collection of cubes in $\mathcal{W}_2$ with side-lengths smaller than $10 \ell_0$.
\item We define $\mathcal{W}_1$ to be the collection of cubes in $\mathcal{W}_0$ with side-length smaller than $c_0\ell_0$.
\item Choose $\ell_0$ so that for any $Q\in \mathcal{W}_3$ there is a $S\in \mathcal{W}_1$ with 
\begin{equation}\label{eqProperties}
\Dist(Q,S)\leq C \ell(Q)\quad\text{ and }\quad \ell(Q)=\ell(S)\quad \text{(see \cite[Lemma 2.4]{Jones})}. 
\end{equation}
\end{itemize}
For $Q\in \mathcal{W}_3$ we define the symmetrized cube $Q^*\in\mathcal{W}_1$ as one of the cubes satisfying \rf{eqProperties}. Note that the number of possible choices for $Q^*$ is uniformly bounded.
}

\begin{lemma}\label{lemSymmetrized}[see \cite[Section 2]{Jones}]
Let $\Omega$ be an interior corkscrew domain. For cubes $Q_1,Q_2\in\mathcal{W}_3$ and $S\in\mathcal{W}_1$ we have that
\begin{itemize}
\item The symmetrized cubes have finite overlapping: there exists a constant $C$ depending on the parameter $\varepsilon$ and the dimension $d$ such that $\#\{Q\in\mathcal{W}_3: Q^*=S\}\leq C$.
\item The long distance is invariant in the following sense:
\begin{equation}\label{eqLongDistanceInvariant}
\Dist(Q_1^*,Q_2^*)\approx \Dist(Q_1,Q_2) \mbox{\quad\quad and \quad\quad}\Dist(Q_1^*,S)\approx \Dist(Q_1,S) 
\end{equation}
\end{itemize}
\end{lemma}

We define the family of bump functions $\{\psi_Q\}_{Q\in \mathcal{W}_2}$ to be a partition of the unity associated to $\left\{\frac{11}{10}Q\right\}_{Q\in\mathcal{W}_2}$, that is, their sum $\sum\psi_Q\equiv 1$, they satisfy the pointwise inequalities $0\leq \psi_Q\leq \chi_{\frac{11}{10}Q}$ and $\norm{\nabla\psi_Q}_\infty \lesssim \frac{1}{\ell(Q)}$.  We can define the operator
$$\Lambda_0 f(x)= f(x)\chi_\Omega(x) + \sum_{Q\in\mathcal{W}_3} \psi_Q(x) f_{Q^*} \mbox{ for any }f\in L^1_{loc}(\Omega)$$
 (here $f_{U}$ stands for the mean of a function $f$ in a set $U$). This function is defined almost everywhere because the boundary of the domain $\Omega$ has zero Lebesgue measure (see \cite[Lemma 2.3]{Jones}).

\begin{theorem}\label{theoExtensionOperator0}
Let $\Omega\subset\R^d$ be an interior $(\varepsilon,\delta)$-corkscrew domain, let $1\leq p<\infty$, $1\leq q\leq \infty$ and $0<s<1$. Then, $\Lambda_0: F^{s,C\ell_0}_{p,q,1}(\Omega)\to F^{s,1}_{p,q,1}(\R^d)$, with $C$ depending only on $d$ and $\varepsilon$ while $\ell_0$ depends also on $\delta$.
\end{theorem}

\begin{proof}
In light of \rf{eqCompareNormsRadius}, it is enough to check $\Lambda_0: F^{s,C\ell_0}_{p,q,1}(\Omega)\to F^{s,\ell_0}_{p,q,1}(\R^d)$ for $\ell_0$ small enough, that is
\begin{align*}
\norm{\Lambda_0 f}_{{F}^{s,\ell_0}_{p,q,1}(\R^d)}
	&=\norm{\Lambda_0 f}_{L^p}+ \norm{\norm{\frac{ \norm{\Lambda_0 f(x)-\Lambda_0 f(y)}_{L^1_y (B_{x,t})}}{t^{s+{d}+\frac1q}}}_{L^q_t(0,\ell_0)}}_{L^p_x(\R^d)} \lesssim \norm{f}_{F^{s,C\ell_0}_{p,q,1}(\Omega)},
\end{align*}
where $B_{x,t}:=B(x,t)$.

First, note that $\norm{\Lambda_0f}_{L^p}\leq \norm{f}_{L^p(\Omega)}+\norm{\Lambda_0f}_{L^p(\Omega^c)}$. By Jensen's inequality, we have that 
$$\norm{\Lambda_0f}_{L^p(\Omega^c)}^p\lesssim_p \sum_{Q\in\mathcal{W}_3} |f_{Q^*}|^p\norm{\psi_Q}_{L^p}^p \leq \sum_{Q\in\mathcal{W}_3} \frac{1}{\ell(Q)^d}\norm{f}_{L^p(Q^*)}^p  \left(\frac{11}{10}\ell(Q)\right)^d . $$
By the finite overlapping of the symmetrized cubes, 
\begin{equation}\label{eqExtensionInLp}
\norm{\Lambda_0f}_{L^p(\Omega^c)}^p\lesssim \norm{f}_{L^p(\Omega)}^p. 
\end{equation}

 It remains to check that
\begin{equation*}
\norm{\Lambda_0 f}_{{\dot F}^{s,{\ell_0}}_{p,q,1}(\R^d)}= \norm{\norm{\frac{ \norm{\Lambda_0 f(x)-\Lambda_0 f(y)}_{L^1_y (B_{x,t})}}{t^{s+{d}+\frac1q}}}_{L^q_t(0,\ell_0)}}_{L^p_x(\R^d)}\lesssim \norm{f}_{F^{s,C{\ell_0}}_{p,q,1}(\Omega)}.
\end{equation*}
{By a decomposition of the integration domains  $x,y\in \Omega \cup \Omega^c$, we can reduce the preceding estimate to proving that }
\begin{align*}
\circled{a}+\circled{b}+\circled{c} \lesssim \norm{f}_{F^{s,C\ell_0}_{p,q,1}(\Omega)},
\end{align*}
where
\begin{align*}
\circled{a} 
	& := \norm{\norm{\frac{ \norm{f(x)-\Lambda_0 f(y)}_{L^1_y (\Omega^c_{x,t})}}{t^{s+{d}+\frac1q}}}_{L^q_t(0,\ell_0)}}_{L^p_x(\Omega)}	,
\end{align*}
\begin{align*}
\circled{b} 
	& := \norm{\norm{\frac{ \norm{\Lambda_0f(x)- f(y)}_{L^1_y (\Omega_{x,t})}}{t^{s+{d}+\frac1q}}}_{L^q_t(0,\ell_0)}}_{L^p_x(\Omega^c)}
	\end{align*}
	and
\begin{align*}
\circled{c}
	& := \norm{\norm{\frac{ \norm{\Lambda_0f(x)- \Lambda_0f(y)}_{L^1_y (\Omega^c_{x,t})}}{t^{s+{d}+\frac1q}}}_{L^q_t(0,\ell_0)}}_{L^p_x(\Omega^c)}.
\end{align*}
{Here, and for the rest of the paper, $\Omega_{x,t}:=\Omega\cap B_{x,t}$ and $\Omega^c_{x,t}:=\overline{\Omega}^c\cap B_{x,t}$. }

Let us begin with 
\begin{align*}
\circled{a}
	& = \norm{\norm{\frac{ \norm{f(x)-\sum_{S\in\mathcal{W}_3} \psi_{S}(y) f_{S^*}}_{L^1_y (\Omega^c_{x,t})}}{t^{s+{d}+\frac1q}}}_{L^q_t(0,\ell_0)}}_{L^p_x(\Omega)}.
\end{align*}
Call $\mathcal{W}_4:=\{S\in\mathcal{W}_3: \mbox{ all the neighbors of $S$ are in } \mathcal{W}_3\}$. We write $ \mathcal{W}_j(Q,t):=\{S\in\mathcal{W}_j: S\cap \bigcup_{x\in Q} B_{x,t}\neq \emptyset \} $. 
Note that {the integral in $t$ is null if $x\in Q\in \mathcal{W}_0\setminus\mathcal{W}_1$. Moreover,} 
 if $S\in\mathcal{W}_3(Q,t)$ with $Q\in \mathcal{W}_1$ and $t<\ell_0$, then $S\in \mathcal{W}_4$. Given $y\in \frac{11}{10}S$, where $S\in \mathcal{W}_4$, we have that $\sum_{P\in\mathcal{W}_3} \psi_P(y)\equiv 1$. Therefore
\begin{align*}
\circled{a}
	&  \leq \norm{\norm{\norm{\frac{ \sum_{S\in\mathcal{W}_4(Q,t)} |   f(x)- f_{S^*}| \int_{\frac{11}{10}S} \psi_{S}(y) \,dy}{t^{s+{d}+\frac1q}}}_{L^q_t(0,\ell_0)}}_{L^p_x(Q)}}_{\ell^p_Q(\mathcal{W}_{1})}
\\
\end{align*}
{Above,  $\ell^p_Q(\mathcal{W}_1)$ stands for the $\ell^p$ of a sequence indexed by the  Whitney cubes in $\mathcal{W}_1$. }
By the choice of the symmetrized cube we have that $\int_{\frac{11}{10}S}  \psi_{S}(y) \,dy \approx \ell(S^*)^d$. Thus,
\begin{align*}
\circled{a}
	& \lesssim_d \norm{\norm{\norm{\frac{\sum_{S\in\mathcal{W}_4(Q,t)} \int_{S^*}|f(x)-f(\xi)|\, d\xi}{t^{s+{d}+\frac1q}}}_{L^q_t(0,\ell_0)}}_{L^p_x(Q)}}_{\ell^p_Q(\mathcal{W}_{1})}
\end{align*}
By \rf{eqLongDistanceInvariant}, if $Q\in\mathcal{W}_1$ and $S\in \mathcal{W}_4(Q,t)$ then $S^*\in \mathcal{W}_1(Q, C_\Omega t)$ and using also the finite overlapping of the symmetrized cubes, we get that 
\begin{align*}
\circled{a}
	& \lesssim_{d,\varepsilon} \norm{\norm{\norm{\frac{\sum_{S\in\mathcal{W}_1(Q,C_\Omega t)} \int_{S}|f(x)-f(\xi)|\, d\xi}{t^{s+{d}+\frac1q}}}_{L^q_t(0,\ell_0)}}_{L^p_x(Q)}}_{\ell^p_Q(\mathcal{W}_{1})}\lesssim_{s,d}  \norm{f}_{\dot F^{s,C_\Omega}_{p,q,1}(\Omega)}.
\end{align*}
	

Next, note that, using the same reasoning as above and the finite superposition of the rescaled Whitney cubes, we have that
\begin{align*}
\circled{b}
	& = \norm{\norm{\frac{ \norm{\sum_{Q\in\mathcal{W}_3} \psi_{Q}(x) f_{Q^*}- f(y)}_{L^1_y (\Omega_{x,t})}}{t^{s+{d}+\frac1q}}}_{L^q_t(0,\ell_0)}}_{L^p_x(\Omega^c)}\\
	& \lesssim_{d,p} \norm{\norm{ \psi_Q(x) \norm{\frac{ \norm{ f_{Q^*}- f(y)}_{L^1_y (\Omega_{x,t})}}{t^{s+{d}+\frac1q}}}_{L^q_t(0,\ell_0)}}_{L^p_x(\frac{11}{10}Q)}}_{\ell^p_Q(\mathcal{W}_4)}.
\end{align*}
Taking absolute values inside and enlarging the integration domain in $y$ and computing the integral in $Q$, we have that
\begin{align*}
\circled{b}
	& \lesssim_{d,p} \norm{ \ell(Q)^\frac dp \norm{\frac{ \norm{\frac{1}{\ell(Q)^{d}} \norm{f(\xi)-f(y)}_{L^1_\xi(Q^*) }}_{L^1_y (\Omega_{Q,t})}}{t^{s+{d}+\frac1q}}}_{L^q_t(0,\ell_0)}}_{\ell^p_Q(\mathcal{W}_4)}\\
\end{align*}
where $\Omega_{Q,t}=\bigcup_{x\in Q}\Omega_{x,t}$. Thus, applying Fubini's theorem and  Minkowsky's integral inequality (see \cite[Appendix A1]{SteinPetit}), we get
\begin{align*}
\circled{b}
	& \lesssim_{d,p} \norm{ \frac{\ell(Q)^\frac dp}{\ell(Q)^{d}} \norm{\norm{\frac{ \norm{f(\xi)-f(y)}_{L^1_y (\Omega_{Q,t})}}{t^{s+{d}+\frac1q}}}_{L^q_t(0,\ell_0)}}_{L^1_\xi(Q^*) }}_{\ell^p_Q(\mathcal{W}_4)}.
\end{align*}
If $Q\in \mathcal{W}_4$ and $y\in \Omega_{Q,t}$, then $y\in\Omega_{\xi,C_\Omega t}$ for every $\xi\in Q^*$. By H\"older's inequality and the finite overlapping of symmetrized cubes, we get that
\begin{align*}
\circled{b}
	& \lesssim_{d,p,\varepsilon} \norm{ \norm{\norm{\frac{ \norm{f(\xi)-f(y)}_{L^1_y (\Omega_{\xi,C_\Omega t})}}{t^{s+{d}+\frac1q}}}_{L^q_t(0,\ell_0)}}_{L^p_\xi(Q) }}_{\ell^p_Q(\mathcal{W}_1)}	\lesssim  \norm{f}_{\dot F^{s,C_\mathcal{W}}_{p,q,1}(\Omega)}.
\end{align*}
	

Let us focus on $\circled{c}$.  By the triangle inequality, we have that
\begin{align*}
\circled{c}
	& \leq \norm{ \norm{\norm{\frac{ \norm{\sum_{P\in\mathcal{W}_3} \psi_{P}(x) f_{P^*}-\sum_{S\in\mathcal{W}_3} \psi_{S}(y) f_{S^*}}_{L^1_y (\Omega^c_{x,t})}}{t^{s+{d}+\frac1q}}}_{L^q_t(\frac{\ell(Q)}{10},\ell_0)}}_{L^p_x(Q) }}_{\ell^p_Q(\mathcal{W}_4)}\\
	& \quad + \norm{ \norm{\norm{\frac{ \norm{\sum_{P\in\mathcal{W}_3} \psi_{P}(x) f_{P^*}-\sum_{S\in\mathcal{W}_3} \psi_{S}(y) f_{S^*}}_{L^1_y (\Omega^c_{x,t})}}{t^{s+{d}+\frac1q}}}_{L^q_t(0,\frac{\ell(Q)}{10})}}_{L^p_x(Q) }}_{\ell^p_Q(\mathcal{W}_4)}\\
	& \quad + \norm{ \norm{\norm{\frac{ \norm{\sum_{P\in\mathcal{W}_3} \psi_{P}(x) f_{P^*}-\sum_{S\in\mathcal{W}_3} \psi_{S}(y) f_{S^*}}_{L^1_y (\Omega^c_{x,t})}}{t^{s+{d}+\frac1q}}}_{L^q_t(0,\ell_0)}}_{L^p_x(Q) }}_{\ell^p_Q(\mathcal{W}_2\setminus\mathcal{W}_4)}\\
	& =: \circled{c1}+\circled{c2}+\circled{c3}
\end{align*}

%

The first term is bounded using the same techniques as in $\circled{a}$ and $\circled{b}$. 
Indeed, given $x\in \frac{11}{10}Q$ where $Q\in \mathcal{W}_4$ and $y\in \Omega^c_{x, {\ell_0}/{10}}$, then neither $x$ nor $y$ are in the support of any bump function of a cube in $\mathcal{W}_2\setminus\mathcal{W}_3$, so $\sum_{S\in\mathcal{W}_3} \psi_S(y)\equiv 1$ and $\sum_{P\in\mathcal{W}_3} \psi_P(x)\equiv 1$. Therefore
$$\sum_{P\in\mathcal{W}_3} \psi_{P}(x) f_{P^*}-\sum_{S\in\mathcal{W}_3} \psi_{S}(y) f_{S^*}=  \sum_{P\cap 2Q\neq \emptyset}  \sum_{S\in\mathcal{W}_3} \psi_{P}(x)\psi_{S}(y)\left(f_{P^*}-f_{S^*}\right).$$
Using first the triangle inequality and the bounded number of neighboring cubes, and then computing the integral in $x$, taking absolute values inside the inner integral and changing the order of summation on $Q$ and $P$ we get
\begin{align*}
\circled{c1}
	&=\norm{ \norm{\norm{\frac{ \norm{\sum_{P\cap 2Q\neq \emptyset}  \sum_{S\in\mathcal{W}_3}\left|\psi_{P}(x) \psi_{S}(y)\right| \left|f_{P^*}- f_{S^*}\right|}_{L^1_y (\Omega^c_{x,t})}}{t^{s+{d}+\frac1q}}}_{L^q_t(\frac{\ell(Q)}{10},\ell_0)}}_{L^p_x(Q) }}_{\ell^p_Q(\mathcal{W}_4)}\\
	&  \lesssim_{d,p} \norm{\left(\sum_{P\cap 2Q\neq \emptyset} \norm{ \psi_{P}(x)\norm{\frac{ \sum_{S\in\mathcal{W}_3(P,C_dt)}\left|f_{P^*}- f_{S^*}\right| \ell(S)^d}{t^{s+{d}+\frac1q}}}_{L^q_t(\frac{\ell(Q)}{10},\ell_0)}}_{L^p_x(Q) }^p\right)^\frac1p}_{\ell^p_Q(\mathcal{W}_4)}\\
	&  \lesssim_d \norm{ \ell(P)^\frac dp \norm{\frac{  \sum_{S\in\mathcal{W}_3(P,C_dt)}  \frac{1}{\ell(P)^d}\norm{\norm{ f(\xi)- f(\zeta) }_{L^1_\zeta(S^*)}}_{L^1_\xi(P^*)}}{t^{s+{d}+\frac1q}}}_{L^q_t(\frac{\ell(P)}{20},\ell_0)}}_{\ell^p_P(\mathcal{W}_3)},
\end{align*}
and applying Minkowski's and Jensen's inequalities we obtain
\begin{align*}
\circled{c1}
	&\lesssim_{d,p}  \norm{ \norm{ \norm{\frac{  \sum_{S\in\mathcal{W}_3(P,C_dt)} \norm{ f(\xi)- f(\zeta) }_{L^1_\zeta(S^*)}}{t^{s+{d}+\frac1q}}}_{L^q_t(\frac{\ell(P)}{20},\ell_0)}}_{L^p_\xi(P^*)}}_{\ell^p_P(\mathcal{W}_3)}.\end{align*}
By Lemma \ref{lemSymmetrized}, we get that 
\begin{align*}
\circled{c1}
	& \lesssim_{d,p,\varepsilon}   \norm{ \norm{\frac{  \norm{ f(\xi)- f(\zeta) }_{L^1_\zeta(\Omega_{\xi,C_\Omega t})}}{t^{s+{d}+\frac1q}}}_{L^q_t(0,\ell_0)}}_{L^p_\xi(\Omega)}\lesssim \norm{f}_{\dot F^{s,C_\Omega}_{p,q,1}(\Omega)}.\end{align*}

If $x\in \frac{11}{10}Q$ where $Q\in \mathcal{W}_4$ and $y\in \Omega^c_{x,\ell(Q)/10}$, since the points are `close' to each other, we will use the Lipschitz regularity of the bump functions, so we write 
\begin{equation}\label{eqDecompositionC23}
\sum_{P\in\mathcal{W}_3} \psi_{P}(x) f_{P^*}-\sum_{S\in\mathcal{W}_3} \psi_{S}(y) f_{S^*}= \sum_{P\in\mathcal{W}_3}\left( \psi_{P}(x)-\psi_{P}(y)\right)f_{P^*}.
\end{equation}
Now we use that $\{\psi_Q\}$ is a partition of the unity with $\psi_Q$ supported in $\frac{11}{10}Q$, that is, $\sum_{S\in\mathcal{W}_3}\psi_S(x)=\sum_{S\cap 2Q\neq \emptyset}\psi_S(x)=1$ if $x \in \frac{11}{10}Q$ with $Q \in \mathcal{W}_4$. Thus,
\begin{align*}
\circled{c2}
	& = \norm{ \norm{\norm{\frac{ \norm{\sum_{S\cap 2Q\neq \emptyset} \left(\psi_{S}(x)- \psi_{S}(y) \right)f_{S^*}}_{L^1_y (\Omega^c_{x,t})}}{t^{s+{d}+\frac1q}}}_{L^q_t(0,\frac{\ell(Q)}{10})}}_{L^p_x(Q) }}_{\ell^p_Q(\mathcal{W}_4)}\\
	& = \norm{ \norm{\norm{\frac{ \norm{\sum_{S\cap 2Q\neq \emptyset} \left(\psi_{S}(x)- \psi_{S}(y) \right)\left(f_{S^*}-f_{Q^*}\right)}_{L^1_y (\Omega^c_{x,t})}}{t^{s+{d}+\frac1q}}}_{L^q_t(0,\frac{\ell(Q)}{10})}}_{L^p_x(Q) }}_{\ell^p_Q(\mathcal{W}_4)}.
\end{align*}
Using the fact that $\norm{\nabla\psi_Q }_\infty\lesssim\frac{1}{\ell(Q)}$ and computing the integrals in $x$ and $t$, we have that
\begin{align*}
\circled{c2}
	& \leq\norm{ \norm{\norm{\frac{ |\Omega^c_{x,t}|  \sum_{S\cap 2Q\neq \emptyset} \norm{\nabla\psi_S}_\infty t \left|f_{S^*}-f_{Q^*}\right|}{t^{s+{d}+\frac1q}}}_{L^q_t(0,\frac{\ell(Q)}{10})}}_{L^p_x(Q) }}_{\ell^p_Q(\mathcal{W}_4)}\\
	& \lesssim_d\norm{  \ell(Q)^\frac dp\norm{t^{1-s-\frac1q}}_{L^q_t(0,\frac{\ell(Q)}{10})} \sum_{S\cap 2Q\neq \emptyset}   \frac{\left|f_{S^*}-f_{Q^*}\right|}{\ell(Q)}}_{\ell^p_Q(\mathcal{W}_4)}\\
	& \lesssim_{d,s} \norm{ \ell(Q)^{\frac dp+1-s-1-2d} \norm{\norm{f(\zeta)-f(\xi)}_{L^1_\zeta(C_\Omega Q^*)}}_{L^1_\xi(Q^*)} }_{\ell^p_Q(\mathcal{W}_4)}
\end{align*}
and using Jensen's inequality and Lemma \ref{lemSymmetrized} we get
\begin{align*}
\circled{c2}
	&\lesssim_{d,s,\varepsilon} \norm{ \ell(Q)^{-s-d} \norm{\norm{\left|f(\zeta)-f(\xi)\right| \fint_{|\xi-\zeta|}^{c_1\ell(Q)}dt}_{L^1_\zeta(C_\Omega Q)}}_{L^p_\xi(Q)} }_{\ell^p_Q(\mathcal{W}_1)}.
\end{align*}
If $c_1$ is chosen big enough, depending only on $d$ and the corkscrew constants of $\Omega$, so that $c_1\ell(Q)-C_\Omega \diam(Q)\approx \ell(Q)$, using Fubini's theorem and H\"older's inequality we obtain
 \begin{align*}
\circled{c2}
		&\lesssim_{d,s,\varepsilon} \norm{  \norm{\norm{\frac{\norm{f(\zeta)-f(\xi) }_{L^1_\zeta(\Omega_{\xi,t})} }{\ell(Q)^{s+d+1}}}_{L^1_t (0,c_1\ell(Q))}}_{L^p_\xi(Q)} }_{\ell^p_Q(\mathcal{W}_1)}\lesssim \norm{f}_{\dot F^{s,C_\Omega}_{p,q,1}(\Omega)}.
\end{align*}

Decomposition \rf{eqDecompositionC23} is still valid if $Q\in \mathcal{W}_2\setminus \mathcal{W}_4$ and $y\in \Omega^c_{x,\ell(Q)/10}$. In particular if  $y\in \Omega^c_{x,\ell_0}$ we can use the decomposition, but we treat this case apart since we lose the cancellation of the sums of bump functions and we gain instead a uniform lower bound on the side-lengths of the cubes involved:
\begin{align*}
\circled{c3}
	& = \norm{ \norm{\norm{\frac{ \norm{\sum_{S\in \mathcal{W}_3: S\cap 2Q\neq \emptyset} \left(\psi_{S}(x)- \psi_{S}(y) \right)f_{S^*}}_{L^1_y (\Omega^c_{x,t})}}{t^{s+{d}+\frac1q}}}_{L^q_t(0,\ell_0)}}_{L^p_x(Q) }}_{\ell^p_Q(\mathcal{W}_2\setminus\mathcal{W}_4)}\\
	& \lesssim_d \norm{ \norm{\norm{\frac{ |\Omega^c_{x,t}| \sum_{S\in \mathcal{W}_3: S\cap 2Q\neq \emptyset}| \frac{t}{\ell_0} f_{S^*}|}{t^{s+{d}+\frac1q}}}_{L^q_t(0,\ell_0)}}_{L^p_x(Q) }}_{\ell^p_Q(\mathcal{W}_2\setminus\mathcal{W}_4)}.
\end{align*}
Computing the integrals in $x$ and $t$ and using the triangle inequality we get
\begin{align*}
\circled{c3}
	& \lesssim_d \frac1{\ell_0} \norm{  \ell(Q)^\frac dp\sum_{S\in \mathcal{W}_3: S\cap 2Q\neq \emptyset}|   f_{S^*}| \norm{t^{1-s-\frac1q}}_{L^q_t(0,\ell_0)}}_{\ell^p_Q(\mathcal{W}_2\setminus\mathcal{W}_4)}\\
	& \lesssim_{d,p,s} \frac{1}{\ell_0^s} \left(\sum_{Q\in\mathcal{W}_2\setminus\mathcal{W}_4}   \sum_{S\in \mathcal{W}_3: S\cap 2Q\neq \emptyset}| f_{S^*}|^p \ell(S)^d\right)^{\frac{1}{p}} \lesssim_\varepsilon \norm{f}_{L^p(\Omega)}.
\end{align*}

\end{proof}

\begin{remark}\label{remProofExtension}
It is usual in the literature to define a uniform domain as a domain satisfying the interior corkscrew condition and the so-called Harnack chain condition (this definition can be seen to be equivalent to the one given here). The interior corkscrew condition can be understood as a quantitative openness condition, while the Harnack chain can be understood as a quantitative connectedness condition. It is not quite surprising that we can drop the connectedness condition for smoothness below one, since the norm is completely non-local. That is, the connection between points following paths inside the domain is not needed because the first-order difference is always included in the norm itself. 

The reader may note that the interior corkscrew condition is a bit stronger than the conditions that we have used. Indeed, the proof works for a domain $\Omega$ such that $\overline{\Omega}^c$ is an exterior corkscrew domain and such that $\partial\Omega\setminus \partial(\overline{\Omega}^c)$ has null Lebesgue measure. For instance one can remove segments on planar domains without changing the extendability properties for $F^s_{p,q}$ with $0<s<1$. \cite[Theorem 1.4]{PratsSaksman} can also be proven in such a general setting, with the restriction in the indices $s>\frac dp-\frac dq$.
\end{remark}

\subsection{Uniform domains and smoothness above one}\label{secUniformSmoothness}

Norman G. Meyers introduced a collection of projections $L:W^{k,p}(Q) \to \mathcal{P}^k$ in \cite{Meyers} which allows us to iterate the Poincar\'e inequality. Peter Jones uses the following particular simple case:
\begin{definition}
Let  $Q \subset \R^d$. Given $f\in L^1(Q)$ with weak derivatives up to order $k$, we define $\mathbf{P}^{k}_Q f\in \mathcal{P}^{k}$ as the unique polynomial of degree smaller or equal than $k$ such that 
\begin{equation}\label{eqdefpnQ}
\int_{Q} D^\beta \mathbf{P}_Q^k f \,dm=\int_{Q} D^\beta f\, dm
\end{equation}
for every multiindex $\beta \in \N^d$ with $|\beta| \leq k$.
\end{definition}

\begin{lemma}[see {\cite[Lemma 4.2]{PratsTolsa}}]\label{lempoly}
Given a cube $Q$ and $f\in W^{k,1}(Q)$, the polynomial $\mathbf{P}^{k}_{Q} f\in \mathcal{P}^{k}$ exists and is unique.
Furthermore, this polynomial has the following properties:
\begin{enumerate}
\item The norm of the polynomial is controlled by 
\begin{equation}\label{eqP1Bis}
\norm{P_{Q}^{k} f}_{L^p(Q)}\leq c_k  \sum_{j=0}^{k}  \ell(Q)^j \norm{\nabla^{j} f}_{L^p(Q)} \quad\quad \mbox{for } 1\leq p \leq \infty. 
\end{equation}

\item Furthermore, if $f\in W^{k,p}(Q)$, for ${1\leq p \leq \infty}$ we have
 \begin{equation}\label{eqP2}
\|f-\mathbf{P}_{Q}^{k} f\|_{L^p(Q)}\leq C \ell(Q)^k \norm{\nabla^k f - (\nabla^k f)_{Q}}_{L^p(Q)} .
 \end{equation}
{Here and through all the text $(f)_Q$ will denote the mean of $f$ in a cube $Q$, with $f$ possibly vector-valued.}

\item Given a uniform domain $\Omega$ with Whitney covering $\mathcal{W}$, given $\beta \in \N_0^d$ with $|\beta|\leq k$ and given two Whitney cubes $Q, S\in \mathcal{W}$ and $f\in W^{k,1}(\Omega)$, 
\begin{equation}\label{eqP3}
\norm{D^\beta (\mathbf{P}^{k}_{S} f-\mathbf{P}^{k}_{Q} f)}_{L^1(S)}  \leq \sum_{P\in [S,Q]}\frac{\ell(S)^d \Dist(P,S)^{k-|\beta|}}{\ell(P)^d}\norm{\nabla^k f - (\nabla^k f)_P}_{L^1(5P)}.
\end{equation}
%
%
\end{enumerate}
\end{lemma}
Estimate \rf{eqP3} can be shown just using the approach in \cite[Lemma 3.1]{Jones}. 

 We  define the operator $\Lambda_k: W^{k,1}_{\rm loc}(\Omega) \to W^{k,1}_{\rm loc}(\Omega\cup \overline{\Omega}^c)$ as
$$\Lambda_k f(x)= f(x)\chi_\Omega(x) + \sum_{Q\in\mathcal{W}_3} \psi_Q(x) P^k_{Q^*} f(x).$$
{This} function is defined almost everywhere because the boundary of the domain $\Omega$ has zero Lebesgue measure (see \cite[Lemma 2.3]{Jones}). Note that the operator can be defined in any interior corkscrew domain, but it {may} fail to be an extension operator if the domain is not uniform (see \cite{Jones,Shvartsman,KoskelaRajalaZhang} for optimal conditions to grant the existence of an extension operator for $W^{1,p}$).

\begin{theorem}\label{theoExtension}
Let $\Omega$ be a uniform domain and $k\in \N$. For every $0<\sigma< 1$, {$1\leq p<\infty$ and $1\leq q \leq \infty$}  with $\sigma>\frac{d}{p}-\frac{d}{q}$  and $\ell_0$ small enough, then 
$$\Lambda_k:F^{s,C\ell_0}_{p,q,1}(\Omega)\to F^{s,\ell_0}_{p,q,1}(\R^n)$$
(with $s=\sigma+k$) is a bounded {extension} operator.
\end{theorem}
\begin{proof}
Let $f\in F^{s,C\ell_0}_{p,q,u}(\Omega)$. From \cite[p.700]{PratsTL} we know that
$$\Lambda_k:W^{k,p}(\Omega)\to W^{k,p} (\R^d).$$
Thus, we just need to prove that 
$$\norm{\Lambda_k f}_{\dot F^{s,\ell_0}_{p,q,1}(\R^n)}\leq C\norm{f}_{F^{s,C\ell_0}_{p,q,1}(\Omega)}.$$

The case $k=0$ is proven in Theorem \ref{theoExtensionOperator0} above. 
Let us assume that $k\geq 1$, and consider $\alpha \in \N_0^d$ with $|\alpha| \leq k$. 
We will check that
$$\norm{D^\alpha \Lambda_k f}_{\dot{F}^{\sigma, \ell_0}_{p,q,1}(\R^n)}\leq C\norm{f}_{F^{s, C\ell_0}_{p,q,1}(\Omega)}.$$

 Note that  
\begin{align}\label{eqLambda0}
 D^\alpha \Lambda_k f 
\nonumber	& =  D^\alpha f  \chi_\Omega + \sum_{Q\in\mathcal{W}_3} D^\alpha (\psi_Q P^k_{Q^*} f)
		=   D^\alpha f  \chi_\Omega + \sum_{Q\in\mathcal{W}_3} \sum_{\beta \leq \alpha}{\alpha \choose \beta} D^{\alpha-\beta}\psi_Q D^\beta  P^k_{Q^*} f\\
	& 	=   \Lambda_0 (D^\alpha f) +\sum_{\beta < \alpha}{\alpha \choose \beta} \sum_{Q\in\mathcal{W}_3}  D^{\alpha-\beta}\psi_Q   D^\beta P^k_{Q^*} f.
\end{align}
Now, from Theorem \ref{theoExtensionOperator0} again, we already have 
$$\norm{\Lambda_0 (D^\alpha f)}_{\dot{F}^{\sigma, \ell_0}_{p,q,1}(\R^n)}\leq C \norm{D^\alpha f}_{F^{\sigma, C\ell_0}_{p,q,1}(\Omega)}\leq C \norm{ f}_{F^{s, C\ell_0}_{p,q,1}(\Omega)}.$$

Thus, for every $\beta<\alpha$ we  need to control
\begin{align}\label{eqBreakBeta}
\squaredGreek{0}
\nonumber	& :=\norm{\sum_{P\in\mathcal{W}_3} D^{\alpha-\beta}\psi_P D^\beta  P^k_{P^*} f}_{\dot F^{\sigma,\ell_0}_{p,q,1}(\R^n)}^p\\
\nonumber	&= \norm{ \norm{ \frac{ \norm{\sum_{P\in\mathcal{W}_3} D^{\alpha-\beta}\psi_P(y) D^\beta  P^k_{P^*} f (y) }_{L^1_y (\Omega^c_{x,t})}}{t^{\sigma+{d}+\frac1q}} }_{L^q_t(0,\ell_0)}}_{L^p_x(\Omega)}\\
\nonumber	& \quad + \norm{ \norm{ \frac{ \norm{\sum_{P\in\mathcal{W}_3} D^{\alpha-\beta}\psi_P(x) D^\beta  P^k_{P^*} f (x) }_{L^1_y (\Omega_{x,t})}}{t^{\sigma+{d}+\frac1q}} }_{L^q_t(0,\ell_0)}}_{L^p_x(\Omega^c)}\\
\nonumber	& \quad +\norm{ \norm{ \frac{ \norm{\sum_{P\in\mathcal{W}_3} \left((D^{\alpha-\beta}\psi_P D^\beta  P^k_{P^*} f)(x) - (D^{\alpha-\beta}\psi_P D^\beta  P^k_{P^*} f)(y)\right) }_{L^1_y (\Omega^c_{x,t})}}{t^{\sigma+{d}+\frac1q}} }_{L^q_t(0,\ell_0)}}_{L^p_x(\Omega^c)}\\
	& =\squared{a}+\squared{b}+\squared{c}.	
\end{align}

First we study the term  \squared{a}. Note that if $Q\in \mathcal{W}_1$ and $S\in \mathcal{W}_3(Q,t)$ for $t<\ell_0$, then $S\in \mathcal{W}_4(Q,t)$ necessarily, where $\mathcal{W}_4:=\{S\in \mathcal{W}_3: \mbox{ all the neighbors of $S$ are in $\mathcal{W}_3$}\}$. Thus, if $y\in S$, we have that $\sum_{P\in \mathcal{W}_3}D^{\alpha-\beta}\psi_P (y) =0$. Therefore,
\begin{align*}
\squared{a}
	& \leq \norm{\norm{ \norm{ \frac{ \norm{\norm{\sum_{P:P\cap2S\neq \emptyset} D^{\alpha-\beta}\psi_P D^\beta  P^k_{P^*} f  }_{L^1 (S)}}_{\ell^1_S(\mathcal{W}_4(Q,t))}}{t^{\sigma+d+\frac1q}} }_{L^q_t(0,\ell_0)}}_{L^p_x(Q)}}_{\ell^p_Q(\mathcal{W}_1)}\\
	& \leq \norm{\norm{ \norm{ \frac{ \norm{\sum_{P:P\cap2S\neq \emptyset}\norm{   D^{\alpha-\beta}\psi_P(D^\beta  P^k_{P^*} f - D^\beta  P^k_{S^*} f ) }_{L^1 (S)}}_{\ell^1_S(\mathcal{W}_4(Q,t))}}{t^{\sigma+d+\frac1q}} }_{L^q_t(0,\ell_0)}}_{L^p_x(Q)}}_{\ell^p_Q(\mathcal{W}_1)}.
\end{align*}

We take absolute values and we use that $\norm{D^{\alpha-\beta}\psi_P}_\infty\lesssim \ell(S)^{-|\alpha-\beta|}$. 
Moreover we develop the telescopic summation \rf{eqP3} along an admissible chain connecting $P^*$ and $S^*$:
\begin{equation}\label{eqChainsOnCubesPS}
\norm{D^\beta  (P^k_{P^*} f -   P^k_{S^*} f)}_{L^1(S)}\lesssim_{d,k} \sum_{L\in [P^*,S^*]}\frac{\ell(S^*)^d \Dist(L,S^*)^{k-|\beta|}}{ \ell(L)^d}\norm{\nabla^k f - (\nabla^k f)_L}_{L^1(5L)}
\end{equation}
Note that in our summation $2S\cap P\neq \emptyset$, so both cubes have comparable size and $\Dist(S^*,P^*)\approx \ell(S)$ by \rf{eqLongDistanceInvariant}. Thus, combining \rf{eqLengthDistance} and \rf{eqAdmissible1}, it is clear that all the elements $L\in [P^*,S^*]$ have comparable size and $\Dist(L,S^*)\approx\ell(S)$. Moreover, by \rf{eqLongDistanceInvariant}, it follows that $\Dist(Q,S)\approx \Dist(Q,S^*)\approx \Dist(Q,L)$
\begin{align}
\squared{a}
\nonumber	& \lesssim_{d,k} \norm{\norm{ \norm{ \frac{ \norm{\sum_{P:P\cap2S\neq \emptyset}\sum_{L\in [P^*,S^*]} \norm{\nabla^k f - (\nabla^k f)_L}_{L^1(5L)}}_{\ell^1_S(\mathcal{W}_4(Q,t))}}{t^{\sigma+d+\frac1q}} }_{L^q_t(0,\ell_0)}}_{L^p_x(Q)}}_{\ell^p_Q(\mathcal{W}_1)}.
\end{align}
Fixing $c_0$ big enough in the definition of $\mathcal{W}_1$ (see Section \ref{secCorkscrew}), we can ensure that $L\in \mathcal{W}_1$ for every $L$ appearing in the right-hand side term above. 
To complete the reduction, note that for every $L\in\mathcal{W}_1$ the number of candidates $S\in  \mathcal{W}_4$ and $P\cap2S\neq \emptyset$ such that $L\in [S^*,P^*]$ is  uniformly bounded  by a  constant depending on $d$ and $\varepsilon$. Moreover, for $Q\in\mathcal{W}_1$ and $t<\ell(Q)$ the family  $\mathcal{W}_4(Q, t)$ is empty. Therefore, we can use Lemma \ref{lemControlTotal} {below} to get
\begin{align}\label{eqBetaA1Bounded}
\squared{a}
	& \lesssim_{d,k,\varepsilon}\norm{\norm{ \norm{ \frac{ \norm{\norm{\nabla^k f - (\nabla^k f)_L}_{L^1(5L)}}_{\ell^1_L(\mathcal{W}_1(Q,Ct))}}{t^{\sigma+d+\frac1q}} }_{L^q_t(\ell(Q),\ell_0)}}_{L^p_x(Q)}}_{\ell^p_Q(\mathcal{W}_1)}  \lesssim  C \norm{\nabla^k f}_{\dot F^{\sigma,C\ell_0}_{p,q,1}(\Omega)}.
\end{align}

Next we apply a similar reasoning to deal with \squared{b}. This case is simpler, because we can use the triangle inequality to get
\begin{align*}
\squared{b}
	& \leq \norm{\norm{\sum_{P\in\mathcal{W}_3} D^{\alpha-\beta}\psi_P D^\beta  P^k_{P^*} f \norm{ \frac{ |\Omega_{x,t}|}{t^{\sigma+d+\frac1q}} }_{L^q_t(\ell(Q),\ell_0)}}_{L^p_x(Q)}}_{\ell^p(\mathcal{W}_4)}\\
	& \lesssim_{d,\sigma} \left(\sum_{Q\in \mathcal{W}_4}  \ell(Q)^{-\sigma p} \int_{Q}\left|\sum_{P\in \mathcal{W}_3: P\cap 2Q\neq \emptyset} D^{\alpha-\beta}\psi_P(x) D^\beta  P^k_{P^*} f(x) \right|^p\,dx\right)^\frac 1p.
\end{align*}

As before, we can use the cancellation to obtain
\begin{align}\label{eqBreakBetaB}
\squared{b}
	& \lesssim_{d,\sigma} \left(\sum_{Q\in \mathcal{W}_4}  \ell(Q)^{-\sigma p-|\alpha-\beta|p}\sum_{P\in \mathcal{W}_3: P\cap 2Q\neq \emptyset}  \int_{Q}  \left|D^\beta  P^k_{P^*} f(x)-D^\beta  P^k_{Q^*} f(x)\right|^p\,dx\right)^\frac1p.
\end{align}

We use again \rf{eqP3} and the fact that $\ell(P)\approx\ell(Q)\approx\ell(L) \approx \Dist(Q,L)$ for every $2Q\cap P\neq\emptyset$ and $L\in [Q^*,P^*]$:
\begin{align*}
\squared{b}
	& \lesssim_{\varepsilon,k} \left(\sum_{L\in \mathcal{W}_1}  \ell(L)^{-\sigma p}  \norm{\nabla^k f - (\nabla^k f)_L}_{L^p(5L)}^p\right)^\frac1p.
\end{align*}

Note that by Jensen's inequality we have that
\begin{align}\label{eqHisYokeIsEasy}
\squared{b}
	& \lesssim \left(\sum_{L\in \mathcal{W}_1}  \int_{5L}\left(\int_L \frac{|\nabla^k f(x) - \nabla^k f(\xi)|\fint_{|\xi-x|}^{c_1\ell(L)}dt}{\ell(L)^{d+\sigma}  } \, d\xi \right)^p dx\right)^\frac1p	
\end{align}
If $c_1$ is chosen big enough, depending only on $d$, so that $c_1\ell(Q)-\diam(Q)\approx \ell(Q)$, using Fubini and Jensen we obtain
 \begin{align}\label{eqBetaBBounded}
\squared{b}
	& \lesssim \left(\sum_{L\in\mathcal{W}_1}   \int_{L}\left(\fint_{0}^{c_1\ell(L)} \frac{\fint_{\Omega_{\xi,t}} |\nabla^kf(x)-\nabla^k f(\xi)|\,  d\xi}{\ell(L)^{\sigma}} dt \right)^p dx   \right)^{\frac{1}{p}}\lesssim \norm{\nabla^k f}_{\dot F^{\sigma,C_d}_{p,q,1}(\Omega)}.
\end{align}

Finally  we need to deal with the term 
\begin{align}\label{eqBetaC}
\squared{c}	&= \norm{ \norm{ \frac{ \norm{\sum_{P\in\mathcal{W}_3} \left((D^{\alpha-\beta}\psi_P D^\beta  P^k_{P^*} f)(x) - (D^{\alpha-\beta}\psi_P D^\beta  P^k_{P^*} f)\right) }_{L^1 (\Omega^c_{x,t})}}{t^{\sigma+d+\frac1q}} }_{L^q_t(0,\ell_0)}}_{L^p_x(\Omega^c)}
\end{align}
Here we will use the previous techniques but some additional tools have to be used to tackle the case $\dist(x,y) << \dist(x,\partial\Omega)$, so we separate the integration regions with this idea in mind. { We get}
\begin{align}\label{eqBreakBetaC}
\squared{c}	
\nonumber & \leq \norm{\norm{ \norm{ \frac{ \norm{\sum_{P\in\mathcal{W}_3} (D^{\alpha-\beta}\psi_P D^\beta  P^k_{P^*} f)(x) - (D^{\alpha-\beta}\psi_P D^\beta  P^k_{P^*} f)(y) }_{L^1_y (\Omega^c_{x,t})}}{t^{\sigma+d+\frac1q}} }_{L^q_t(\frac{\ell(Q)}{10},\ell_0)}}_{L^p_x(Q)}}_{\ell^p_Q(\mathcal{W}_4)}\\
\nonumber & + \norm{\norm{ \norm{ \frac{ \norm{\sum_{P\in\mathcal{W}_3} (D^{\alpha-\beta}\psi_P D^\beta  P^k_{P^*} f)(x) - (D^{\alpha-\beta}\psi_P D^\beta  P^k_{P^*} f)(y) }_{L^1_y (\Omega^c_{x,t})}}{t^{\sigma+d+\frac1q}} }_{L^q_t(0,\frac{\ell(Q)}{10})}}_{L^p_x(Q)}}_{\ell^p_Q(\mathcal{W}_4)}\\
\nonumber & + \norm{\norm{ \norm{ \frac{ \norm{\sum_{P\in\mathcal{W}_3} (D^{\alpha-\beta}\psi_P D^\beta  P^k_{P^*} f)(x) - (D^{\alpha-\beta}\psi_P D^\beta  P^k_{P^*} f)(y) }_{L^1_y (\Omega^c_{x,t})}}{t^{\sigma+d+\frac1q}} }_{L^q_t(0,\ell_0)}}_{L^p_x(Q)}}_{\ell^p_Q(\mathcal{W}_2\setminus \mathcal{W}_4)}\\
	& =: \squared{c.1}+\squared{c.2}+\squared{c.3}.
\end{align}

Let us consider the case $x\in Q \in \mathcal{W}_4$, $t>\ell(Q)/10$ and $y\in \Omega^c_{x,t}$. In this case we will bound the numerator in \rf{eqBreakBetaC} above by
\begin{equation}\label{eqFirstBound}
\left|\sum_{P\in\mathcal{W}_3} D^{\alpha-\beta}\psi_P(x) D^\beta  P^k_{P^*} f(x)\right| +  \left|\sum_{P\in\mathcal{W}_3}D^{\alpha-\beta}\psi_P(y) D^\beta  P^k_{P^*} f(y)\right|.
\end{equation}
We obtain
\begin{align*}
\squared{c.1}	
	& \lesssim \norm{\norm{ \norm{ \frac{ \norm{\sum_{P\in\mathcal{W}_3} D^{\alpha-\beta}\psi_P D^\beta  P^k_{P^*} f }_{L^1 (\Omega^c_{x,t})}}{t^{\sigma+d+\frac1q}} }_{L^q_t(\frac{\ell(Q)}{10},\ell_0)}}_{L^p_x(Q)}}_{\ell^p_Q(\mathcal{W}_4)}\\
	& \quad + \norm{\norm{  \sum_{P\in\mathcal{W}_3} D^{\alpha-\beta}\psi_P D^\beta  P^k_{P^*} f \norm{ \frac{ |\Omega^c_{x,t}|}{t^{\sigma+d+\frac1q}} }_{L^q_t(\frac{\ell(Q)}{10},\ell_0)}}_{L^p_x(Q)}}_{\ell^p_Q(\mathcal{W}_4)}\\
	& = \squared{c.1.1}+\squared{c.1.2}.
\end{align*}
Now, \squared{c.1.1} is bounded as \squared{a}: for every $x\in Q\in \mathcal{W}_4$ and $t\in \left(\frac{\ell(Q)}{10},\ell_0\right)$ we write
\begin{align*}
\norm{\sum_{P\in\mathcal{W}_3} D^{\alpha-\beta}\psi_P D^\beta  P^k_{P^*} f }_{L^1 (\Omega^c_{x,t})}	& \lesssim \norm{\sum_{P:P\cap2S\neq \emptyset} \ell(S)^{-|\alpha-\beta|}\norm{ D^\beta  P^k_{P^*} f - D^\beta  P^k_{S^*} f  }_{L^1 (S)}}_{\ell^1_S(\mathcal{W}_3(Q,t))}.
\end{align*}
Then we use cube chains as in \rf{eqChainsOnCubesPS} and Lemma \ref{lemControlTotal} {below} to get
\begin{align} 
\squared{c.1.1}
\nonumber	& \lesssim_\varepsilon \norm{\norm{ \norm{ \frac{ \norm{\norm{\nabla^k f - (\nabla^k f)_L}_{L^1(5L)}}_{\ell^1_L(\mathcal{W}_1(Q,Ct))}}{t^{\sigma+d+\frac1q}} }_{L^q_t(\frac{\ell(Q)}{10},\ell_0)}}_{L^p_x(Q)}}_{\ell^p_Q(\mathcal{W}_4)} \lesssim  C \norm{\nabla^k f}_{\dot F^{\sigma,Ct}_{p,q,1}(\Omega)}.
\end{align}

On the other hand, \squared{c.1.2} is bounded as \squared{b} without much change:
\begin{align*}
\squared{c.1.2}
	& \lesssim \left(\sum_{Q\in \mathcal{W}_4}  \ell(Q)^{-\sigma p} \int_{Q}\left|\sum_{P\in \mathcal{W}_3: P\cap 2Q\neq \emptyset} D^{\alpha-\beta}\psi_P(x) D^\beta  P^k_{P^*} f(x) \right|^p\,dx\right)^\frac 1p \lesssim_{d,s,\varepsilon}  \norm{\nabla^k f}_{\dot F^{\sigma,C_\Omega}_{p,q,1}(\Omega)}.
\end{align*}
Combining these estimates we obtain
\begin{equation}\label{eqBetaC3Bounded}
\squared{c.1} \lesssim_{d,s,\varepsilon}  \norm{\nabla^k f}_{\dot F^{\sigma,C_\Omega}_{p,q,1}(\Omega)}.
\end{equation}

If $x\in Q \in \mathcal{W}_4$ and  $y\in  B(x,{\ell(Q)}/{10})$, then we can use the fact that  
$$\sum_{P\in\mathcal{W}_3} D^{\alpha-\beta}\psi_P(y) = \sum_{P\in\mathcal{W}_3} D^{\alpha-\beta}\psi_P(x) = 0.$$
We will bound the numerator of the second term in \rf{eqBreakBetaC} by
\begin{align}\label{eqBreakALot}
\nonumber& \left|\sum_{P\in\mathcal{W}_3} (D^{\alpha-\beta}\psi_P(x)-D^{\alpha-\beta}\psi_P(y)) (D^\beta  P^k_{P^*} f(x) )+ D^{\alpha-\beta}\psi_P(y) (D^\beta  P^k_{P^*} f(x) - D^\beta  P^k_{P^*} f(y)) \right| \\
		& \leq \left|\sum_{P\in\mathcal{W}_3} (D^{\alpha-\beta}\psi_P(x)-D^{\alpha-\beta}\psi_P(y)) (D^\beta  P^k_{P^*} f(x)-D^\beta  P^k_{Q^*} f(x))\right|\\
\nonumber & + \left|  \sum_{P\in\mathcal{W}_3} D^{\alpha-\beta}\psi_P(y) \left((D^\beta  P^k_{P^*} f-D^\beta  P^k_{Q^*} f)(x) - (D^\beta  P^k_{P^*} f-D^\beta  P^k_{Q^*} f)(y)\right) \right| .
\end{align}

We obtain
\begin{align}\label{eqBreakBetaC2}
\squared{c.2}	
\nonumber	& \lesssim  \norm{\norm{ \norm{ \frac{ \norm{\sum_{\substack{P\in\mathcal{W}_3\\ P\cap 2Q \neq \emptyset}} \norm{\nabla D^{\alpha-\beta}\psi_P}_\infty \left|D^\beta  P^k_{P^*} f(x)-D^\beta  P^k_{Q^*} f(x)\right| }_{L^1_y (\Omega^c_{x,t})}}{t^{\sigma-1+d+\frac1q}} }_{L^q_t(0,\frac{\ell(Q)}{10})}}_{L^p_x(Q)}}_{\ell^p_Q(\mathcal{W}_4)}\\
\nonumber	&  +  \norm{\norm{ \norm{ \frac{ \norm{\sum_{\substack{P\in\mathcal{W}_3\\ P\cap 2Q \neq \emptyset}}   \norm{ D^{\alpha-\beta}\psi_P}_\infty   \norm{ \nabla(D^\beta  P^k_{P^*} f-D^\beta  P^k_{Q^*} f)}_{L^\infty(P)}}_{L^1_y (\Omega^c_{x,t})}}{t^{\sigma-1+d+\frac1q}} }_{L^q_t(0,\frac{\ell(Q)}{10})}}_{L^p_x(Q)}}_{\ell^p_Q(\mathcal{W}_4)}\\
	& = \squared{c.2.1}+\squared{c.2.2} .
\end{align}

In the first term above, we  integrate on $y$ and $t$, we use the control on the derivatives of the bump functions and we plug \rf{eqP3} in to get
\begin{align*}
\squared{c.2.1}	
	& \lesssim_{d,\sigma,p}  \left(\sum_{Q\in \mathcal{W}_4}  \sum_{P\in\mathcal{W}_3: P\cap 2Q \neq \emptyset} \ell(P)^{-(|\alpha-\beta|+1)p} \norm{D^\beta  P^k_{P^*} f-D^\beta  P^k_{Q^*} f}_{L^p(Q)}^p \ell(Q)^{(1-\sigma)p}  \right)^\frac1p\\
	& \lesssim_\varepsilon \left(\sum_{L\in \mathcal{W}_1} \ell(L)^{-(|\alpha-\beta|+1)p} \ell(L)^{(1-\sigma)p}  \frac{\ell(L)^d \ell(L)^{(|\alpha-\beta|)p}}{\ell(L)^d}\norm{\nabla^k f - (\nabla^k f)_L}_{L^p(5L)}  \right)^\frac1p
\end{align*}
so, as in \rf{eqHisYokeIsEasy} we get
\begin{align}\label{eqBetaC11Bounded}
\squared{c.2.1}	
	& \lesssim \left( \sum_{L\in \mathcal{W}_1}  \ell(L)^{-\sigma p} \norm{\nabla^k f - (\nabla^k f)_L}_{L^p(5L)} \right)^\frac1p \lesssim  \norm{\nabla^k f}_{\dot F^{\sigma,C_\Omega}_{p,q,1}(\Omega)}.
\end{align}

Note that the equivalence of norms of polynomials {implies}
$$\norm{\nabla(D^\beta  P^k_{P^*} f-D^\beta  P^k_{Q^*} f)}_{L^\infty(P)}^p\ell(P)^d \approx_{d,k} \norm{\nabla(D^\beta  P^k_{P^*} f-D^\beta  P^k_{Q^*} f)}_{L^p(P)}^p .$$ 
Thus, in the second term, using the same reasoning as above we get
\begin{align*}
\squared{c.2.2}	
	& \lesssim_{\sigma,p} \left( \sum_{Q\in \mathcal{W}_4}  \sum_{P\in\mathcal{W}_3: P\cap 2Q \neq \emptyset} \ell(P)^{-|\alpha-\beta|p} \norm{\nabla(D^\beta  P^k_{P^*} f-D^\beta  P^k_{Q^*} f)}_{L^\infty(P)}^p \ell(Q)^{(1-\sigma)p+d} \right)^\frac1p \\
	& \lesssim_{d,k,\varepsilon} \left( \sum_{L\in \mathcal{W}_1} \ell(L)^{-|\alpha-\beta| p} \frac{\ell(L)^d \ell(L)^{(|\alpha-\beta|-1)p}}{\ell(L)^d}\norm{\nabla^k f - (\nabla^k f)_L}_{L^p(5L)}  \ell(L)^{(1-\sigma)p} \right)^\frac1p
\end{align*}
and we get the same case as before. By \rf{eqBreakBetaC2} and \rf{eqBetaC11Bounded} we get
\begin{equation}\label{eqBetaC1Bounded}
\squared{c.2} \lesssim_{d,s,p,\varepsilon}  \norm{\nabla^k f}_{\dot F^{\sigma,C_\Omega}_{p,q,1}(\Omega)}.
\end{equation}

{Finally we} deal with the term \squared{c.3}. Whenever $x\in Q\in \mathcal{W}_2\setminus \mathcal{W}_4$ and $y\in B(x,\ell(Q)/10)\subset \frac{11}{10}Q$, we bound the numerator in \rf{eqBreakBetaC} by the left-hand side of \rf{eqBreakALot} above:
\begin{align*}
\squared{c.3}
\nonumber	& \leq \norm{\norm{ \norm{ \frac{ \norm{\sum_{P\in\mathcal{W}_3: P\cap 2Q \neq \emptyset} \norm{\nabla D^{\alpha-\beta}\psi_P}_\infty D^\beta  P^k_{P^*} f(x)}_{L^1_y (\Omega^c_{x,t})}}{t^{\sigma-1+d+\frac1q}} }_{L^q_t(0,\ell_0)}}_{L^p_x(Q)}}_{\ell^p_Q(\mathcal{W}_2\setminus \mathcal{W}_4)}	\\
\nonumber	& \quad+  \norm{\norm{ \norm{ \frac{ \norm{\sum_{P\in\mathcal{W}_3: P\cap 2Q \neq \emptyset} \norm{ D^{\alpha-\beta}\psi_P}_\infty \norm{\nabla D^\beta  P^k_{P^*} f}_{L^\infty\left(\frac{11}{10}Q\right)} }_{L^1_y (\Omega^c_{x,t})}}{t^{\sigma-1+d+\frac1q}} }_{L^q_t(0,\ell_0)}}_{L^p_x(Q)}}_{\ell^p_Q(\mathcal{W}_2\setminus \mathcal{W}_4)}
\end{align*}
{We write $Q\in \mathcal{W}_3'$ if $Q\in\mathcal{W}_2$ has neighbors $P\in \mathcal{W}_3$.} Both terms are controlled by integrating on $y$ and $t$ again and using the control on the derivatives of the bump functions together with \rf{eqP1Bis} and the finite overlapping of symmetrized cubes to get
\begin{align}\label{eqBetaC2Bounded}
\squared{c.3}	
\nonumber	& \lesssim_{\sigma,p} \left( \sum_{Q\in \mathcal{W}_3'\setminus\mathcal{W}_4}  \sum_{P\in\mathcal{W}_3: P\cap 2Q \neq \emptyset}    \ell_0^{-|\alpha-\beta|p}(\ell_0^{-p}+1) \norm{|D^\beta  P^k_{P^*} f|+|\nabla D^\beta  P^k_{P^*} f|}_{L^p(Q)}^p  \ell_0^{(1-\sigma)p}  \right)^\frac1p\\
	& \lesssim_k \left(\sum_{P\in  \mathcal{W}_3}  \norm{f}_{W^{k,p}(P^*)}^p  \right)^\frac1p \lesssim_{\varepsilon,d} \norm{f}_{W^{k,p}(\Omega)}.
\end{align}

Combining \rf{eqBreakBetaC}, \rf{eqBetaC3Bounded},  \rf{eqBetaC1Bounded} and \rf{eqBetaC2Bounded} we have
\begin{equation*}
\squared{c} \lesssim \norm{f}_{F^{s,C_\Omega}_{p,q,1}(\Omega)}, 
\end{equation*}
which combined with \rf{eqBreakBeta}, \rf{eqBetaA1Bounded} and \rf{eqBetaBBounded}, leads to
\begin{equation*}
\squared{0} \lesssim_{d,s,p,\varepsilon} \norm{f}_{F^{s,C_\Omega}_{p,q,1}(\Omega)}
\end{equation*}
and the theorem follows.

\end{proof}

It remains to proof a couple of technical lemmata used during the proof of the  boundedness of the extension operator.

\begin{lemma}\label{lemControlTotal}
Let  $d\geq 1$ be a natural number, let $0<\sigma<1$, let $1\leq p<\infty$,  let  $1\leq q \leq \infty$  and $\ell_0$ small enough. There exists a constant $C$ such that for every  $f\in L^p(\Omega)$, 
$$\norm{ \norm{ \frac{\norm{\norm{ f -  f_L}_{L^1(5L)}}_{\ell^1_L(\mathcal{W}_1(Q,Ct))}}{t^{\sigma+d+\frac1q}} }_{L^q_t(\frac{\ell(Q)}{10},\ell_0)}\ell(Q)^\frac dp}_{\ell^p_Q(\mathcal{W}_1 \cup \mathcal{W}_4)}
\leq C \norm{f}_{\dot F^{\sigma,C_\Omega}_{p,q,1}(\Omega)}.$$
\end{lemma}
\begin{proof}

 Writing $h(\xi):=\sum_{L\in\mathcal{W}_1}\chi_{5L}(\xi) \ell(L)^{-\sigma} | f(\xi)-f_L|$ and applying Lemma \ref{lemNormA} below, we get
\begin{align*}
\squared{A}
	& = \left(\sum_{Q\in \mathcal{W}_1 \cup  \mathcal{W}_4}\left( \int_{\frac{\ell(Q)}{10}}^{\ell_0} \left( \sum_{L\in \mathcal{W}_1(Q,\rho t)} \norm{h}_{L^1(L)}  \frac{ \ell(L)^{\sigma}}{t^{\sigma +  d}}  \right)^q \frac{dt}{t} \right)^\frac pq \ell(Q)^d \right)^\frac1p \lesssim \norm{h}_{L^p(\Omega)}.
\end{align*}

 By Jensen's inequality, for $q<\infty$, $\xi\in L\in\mathcal{W}_0$ we get
\begin{align*}
h(\xi)= \sum_{P\ni \xi}\ell(P)^{-\sigma} | f(\xi)-f_{5P}|  \lesssim_d \sum_{P\ni \xi}\fint_{5P} | f(\xi)-f(\zeta)| \int_{|\xi-\zeta|}^{C\ell(L)} \frac{dt}{t^{\sigma+d+\frac1q}}\, d\zeta\ell(L)^{d+\frac1q-1},
\end{align*}
where $C$ is an appropriate dimensional constant. 
Reordering and applying Jensen's inequality we get 
\begin{align*}
h(\xi)\lesssim \fint_{0}^{C\ell(L)}   \frac{\fint_{\Omega_{\xi,t}} | f(\xi)-f(\zeta)| \, d\zeta }{t^{\sigma}}\frac{dt}{t^{\frac1q}} \ell(L)^{\frac1q}\lesssim \left(\int_{0}^{C\ell_0}   \left(\frac{\int_{\Omega_{\xi,t}} | f(\xi)-f(\zeta)| \, d\zeta }{t^{\sigma+{d}}}\right)^q \frac{dt}{t}\right)^\frac1q .
\end{align*}
In case $q=\infty$ then also using Fubini's theorem
\begin{align*}
h(\xi)=\ell(L)^{-\sigma} | f(\xi)-f_{5L}|  \lesssim \sup_{t\in({0},{C\ell_0})}   \frac{\int_{\Omega_{\xi,t}} | f(\xi)-f(\zeta)| \, d\zeta }{t^{\sigma+{d}}} .\end{align*}
Therefore, 
\begin{align*}
\norm{h}_{L^p(\Omega)}	&\lesssim \norm{f}_{\dot F^{\sigma,C\ell_0}_{p,q,1}(\Omega)},
\end{align*}
and the lemma follows.\end{proof}

\begin{lemma}\label{lemNormA}
Let  $d\geq 1$ be a natural number, let $0<\sigma<1$, let $1\leq p<\infty$, $1\leq q \leq \infty$ and $\ell_0$ small enough. For every constant $\rho>0$ there exists a constant $c_2$ such that for every  $h\in L^p(\Omega)$, 
$$\left(\sum_{Q\in \mathcal{W}_1 \cup  \mathcal{W}_4}\left( \int_{\frac{\ell(Q)}{10}}^{\ell_0} \left( \sum_{L\in \mathcal{W}_1(Q,\rho t)} \norm{h}_{L^1(L)}  \frac{ \ell(L)^{\sigma }}{t^{\sigma  +  d}}  \right)^q \frac{dt}{t} \right)^\frac pq \ell(Q)^d \right)^\frac1p \leq C \norm{h}_{L^p},$$
with the usual modifications when $q=\infty$.
\end{lemma}
\begin{proof}
We write
$$\squared{B}:=\left(\sum_{Q\in \mathcal{W}_1}\left( \int_{\frac{\ell(Q)}{10}}^{\ell_0} \left( \sum_{L\in \mathcal{W}_1(Q,\rho t)} \norm{h}_{L^1(L)}  \frac{ \ell(L)^{\sigma}}{t^{\sigma +  d}}  \right)^q \frac{dt}{t} \right)^\frac pq \ell(Q)^d \right)^\frac1p,$$
with the usual modification for $q=\infty$. 

First of all let us assume that $1 =p\leq q<\infty$.  Let us note the following: for $Q\in\mathcal{W}_1$, $t\in \left(\frac{\ell(Q)}{10},\ell_0\right)$ and $L\in \mathcal{W}_1(Q,\rho t)$ it follows that $\rho t\gtrsim \Dist(Q,L)$. Thus, using Minkowski's inequality we get
$$\squared{B}\leq \sum_{Q\in \mathcal{W}_1}\sum_{L\in \mathcal{W}_1} \norm{h}_{L^1(L)} \ell(L)^{\sigma } \left( \int_{\frac{\Dist(Q,L)}{C\rho}}^{\ell_0}   \frac{ dt}{t^{\sigma q +  dq+1}}    \right)^\frac 1q \ell(Q)^d .$$
Computing, using H\"older's inequality and \rf{eqMaximalAllOver} we get
$$\squared{B}\lesssim_\rho \sum_{L\in \mathcal{W}_1} \norm{h}_{L^1(L)}  \ell(L)^{\sigma} \sum_{Q\in \mathcal{W}_1}   \frac{ \ell(Q)^d}{\Dist(Q,L)^{\sigma  +  d}}       \lesssim \sum_{L\in \mathcal{W}_1} \norm{h}_{L^1(L)}  \ell(L)^{\sigma -\sigma }  =  \norm{h}_{L^1}.$$

If $p=1$ and $q=\infty$, then
\begin{align*}
\squared{B}_\infty 	
	& := \sum_{Q\in \mathcal{W}_1}  \sup_{t\in\left( \frac{\ell(Q)}{10} , \ell_0\right)} \sum_{L\in \mathcal{W}_1(Q,\rho t)} \norm{h}_{L^1(L)}  \frac{ \ell(L)^{\sigma}}{t^{\sigma+ d}}   \ell(Q)^d.  
\end{align*}
Since the supremum of a series is bounded by the series of supremums, we get
$$\squared{B}_\infty \leq \sum_{Q\in \mathcal{W}_1} \sum_{L\in \mathcal{W}_1} \norm{h}_{L^1(L)} \ell(L)^{\sigma } \ell(Q)^d \sup_{t\in\left( \frac{\Dist(Q,L)}{C\rho} , \ell_0\right)} \frac{ 1}{t^{\sigma +  d}}      \leq \sum_{L\in \mathcal{W}_1} \norm{h}_{L^1(L)} \ell(L)^{\sigma } \sum_{Q\in \mathcal{W}_1}  \frac{ \ell(Q)^d}{\Dist(Q,L)^{\sigma +  d}},$$
and the lemma follows as in the preceding case.

Next we focus on the case $1< p$ and $q<\infty$. Consider 
$$f(x,t,y) = \sum_{Q\in \mathcal{W}_1} \sum_{L\in \mathcal{W}_1(Q,\rho t)} \chi_Q(x) \chi_{\left(\frac{\ell(Q)}{10},\ell_0\right)}(t) \chi_L (y) \norm{h}_{L^1(L)} \frac{ \ell(L)^{\sigma -d}}{t^{\sigma + \frac dq+\frac1q}}.$$
 Then by duality we get
\begin{align*}
 \norm{f}_{L^p_x(L^q_t(L^1_y))} = \sup_{\norm{g}_{L^{p'}_x\left(L^{q'}_t\left(L^{\infty}_y\right)\right)} \leq 1}  \int_{\Omega}\int_0^\infty \int_\Omega f(x,t,y) g(x,t,y) \, dy\, dt\, dx.
\end{align*}
Thus, it is enough to bound 
\begin{align*}
\squared{B}_g
	&:= \sum_{Q\in \mathcal{W}_1}\int_Q \int_{\frac{\ell(Q)}{10}}^{\ell_0} \sum_{L\in \mathcal{W}_1(Q,\rho t)} \norm{h}_{L^1(L)} \frac{ \ell(L)^{\sigma - d}}{t^{\sigma + d+\frac1q}}   \int_L g(x,t,y) \, dy\, dt\, dx
\end{align*}
for every given function $g$ such that $\norm{g}_{L^{p'}_x\left(L^{q'}_t\left(L^{\infty}_y\right)\right)} \leq 1$.

We will use duality and the boundedness of the maximal operator in the Lebesgue spaces. In particular we will use an extra index $r$ to be fixed later on in order to gain the necessary room for the boundedness of the maximal operator. 
In the sum above,  for every $Q$ and ${L}$ appearing in the sum we consider a chain $[Q,{L}]$ with central cube $R=Q_{{L}}$. If $\ell_0$ is small enough, we can grant that $\dist(R,\partial\Omega)<<\delta$ and in particular \rf{eqMaximalAllOver} and \rf{eqMaximalGuay} hold. We define $\mathcal{W}_1'$ to be the cubes in $\mathcal{W}_0$ that satisfy these properties. Note that $\ell(R)\approx\Dist(Q,R)\leq C\Dist(Q,L)\leq C\rho t$.
Then, using \rf{eqAdmissible2} and reordering we get
\begin{align}\label{eqFirstStepTowardsVictory}
\squared{B}_g
	&\leq \sum_{Q\in \mathcal{W}_1} \int_Q \int_{\frac{\ell(Q)}{10}}^{\ell_0} \sum_{\substack{R\in \mathcal{W}_1'(Q,C\rho t) \\ Q\in \SH(R)}} \sum_{L\in \SH(R)} \norm{h}_{L^r(L)} \frac{ \ell(L)^{\sigma - \frac dr}}{t^{\sigma + d+\frac1q}}   \int_L g(x,t,y) \, dy\, dt\, dx\\
\nonumber	&\leq  \sum_{R\in \mathcal{W}_1'}  \sum_{Q\in \SH(R)} \int_Q  \int_{C\rho^{-1} \ell(R)}^{\ell_0}\sum_{L\in \SH(R)} \norm{h}_{L^r(L)} \frac{ \ell(L)^{\sigma -\frac dr}}{t^{\sigma + d+\frac1q}}   \int_L g(x,t,y) \, dy dt\, dx
\end{align}

Let $1<r<p$. We  apply the H\"older inequality  to get
\begin{align*}
\squared{B}_g
	& \leq \sum_{R\in \mathcal{W}_1'}  \sum_{Q\in \SH(R)} \int_Q  \int_{C\rho^{-1} \ell(R)}^{\ell_0}   \frac{\sup_{\SH(R)} g(x,t,\cdot) \, dt}{t^{\sigma + d+\frac1q}} dx \\
	& \quad \left(\sum_{L\in \SH(R)}  \norm{ h}_{L^r(L)}^{r}\right)^\frac1{r}
	\left(\sum_{L\in \SH(R)} \left(\ell(L)^{\sigma+d-\frac{d}{r}}\right)^{r'}\right)^\frac{1}{r'}.
\end{align*}

Let us denote 
$$ \norm{g(x,t,\cdot)}_{L^{\infty}(\Omega)} =: G_x(t).$$
By \rf{eqMaximalGuay} we get
$$\sum_{L\in \SH(R)}  \norm{ h}_{L^r(L)}^{r}\leq \ell(R)^d \inf_{\zeta \in R} M\left(|h|^{r}\right)(\zeta).$$
Finally, by \rf{eqMaximalAllOver} we get
$$\left(\sum_{L\in \SH(R)} \left(\ell(L)^{\sigma+{d}-\frac{d}{r}}\right)^{r'}\right)^\frac{1}{r'}
	= \left(\sum_{L\in \SH(R)} \left(\ell(L)^{\sigma r' +d}\right)\right)^\frac{1}{r'} 	
	\lesssim \ell(R)^{\sigma+\frac{d}{r'}} .$$
All together, we have gotten
\begin{align*}
\squared{B}_g
	& \lesssim \sum_{R\in \mathcal{W}_1'}  \sum_{Q\in \SH(R)} \int_Q  \int_{C\rho^{-1} \ell(R)}^{\ell_0} G_x(t) \frac{dt}{t^{\sigma + d+\frac1q}} dx \left(\ell(R)^d \inf_{\zeta \in R} M\left(|h|^{r}\right)(\zeta)\right)^\frac1{r} \ell(R)^{\sigma+\frac{d}{r'}} .
\end{align*}

By H\"older's inequality
$$\int_{C\rho^{-1} \ell(R)}^{\ell_0} G_x(t) \frac{dt}{t^{\sigma + d+\frac1q}}\leq \norm{G_x}_{L^{q'}} \left(\int_{C\rho^{-1} \ell(R)}^{\ell_0}  \frac{dt}{t^{\sigma q + {dq}+1}}\right)^\frac1q \lesssim \norm{G_x}_{L^{q'}} (\rho \ell(R))^{-\sigma-d},$$
so writing $G(x):=\norm{G_x}_{L^{q'}}$ we get
\begin{align*}
\squared{B}_g
	& \lesssim_\rho \sum_{R\in \mathcal{W}_1'}  \sum_{Q\in \SH(R)} \int_Q G(x)dx \ell(R)^{-\sigma-d}  \left(\inf_{\zeta \in R} M\left(|h|^{r}\right)(\zeta)\right)^\frac1{r} \ell(R)^{\sigma+{d}} .
\end{align*}
Using \rf{eqMaximalGuay} again we get that $ \sum_{Q\in \SH(R)} \int_Q G(x) \, dx \lesssim \int_R MG(\zeta)\, d\zeta$ and, computing we get
\begin{align*}
\squared{B}_g
	& \lesssim_\rho \sum_{R\in \mathcal{W}_1'}   \int_R MG(\zeta) \left( M\left(|h|^{r}\right)(\zeta)\right)^\frac1{r}\, d\zeta
		\leq \norm{ MG}_{L^{p'}(\Omega)}  \norm{ M\left(|h|^{r}\right)}_{L^\frac{p}{r}(\Omega)}^\frac1{r}\\
	& \lesssim \norm{G}_{L^{p'}(\Omega)}  \norm{|h|^{r}}_{L^\frac{p}{r}(\Omega)}^\frac1{r}
		=  \norm{g}_{L^{p'}_x\left(L^{q'}_t\left(L^{\infty}_y\right)\right)}  \norm{h}_{L^p(\Omega)} \leq   \norm{h}_{L^p(\Omega)}.
\end{align*}

When $1<p<\infty$ and $q=\infty$ we can perform a similar reasoning avoiding the duality expression:
\begin{align*}
\squared{B}_\infty 	
	& := \left(\sum_{Q\in \mathcal{W}_1} \left( \sup_{t\in\left( \frac{\ell(Q)}{10} , \ell_0\right)} \sum_{L\in \mathcal{W}_1(Q,\rho t)} \norm{h}_{L^1(L)}  \frac{ \ell(L)^{\sigma}}{t^{\sigma+ d}}  \right)^p \ell(Q)^d \right)^\frac1p\\
	& \leq \left(\sum_{Q\in \mathcal{W}_1}  \sup_{t\in\left( \frac{\ell(Q)}{10} , \ell_0\right)} \left( \sum_{\substack{R\in \mathcal{W}_1'(Q,C\rho t) \\ Q\in \SH(R)}} \sum_{L\in \SH(R)} \norm{h}_{L^{r}(L)}  \frac{ \ell(L)^{\sigma  +  \frac{d}{r'}}}{t^{\sigma+  d}}  \right)^p   \ell(Q)^d \right)^\frac1p\\
	& \leq \left(\sum_{Q\in \mathcal{W}_1}  \sup_{t\in\left( \frac{\ell(Q)}{10} , \ell_0\right)} \left( \sum_{\substack{R\in \mathcal{W}_1'(Q,C\rho t) \\ Q\in \SH(R)}} \norm{h}_{L^{r}(\SH(R))} \frac{ \ell(R)^{\sigma  + \frac{d}{r'}}}{t^{\sigma +  d}}  \right)^p   \ell(Q)^d \right)^\frac1p.
\end{align*}
Next we use that $Q\in \SH(R)$ implies that $\norm{h}_{L^{r}(\SH(R))} \lesssim \left(\ell(R)^d \inf_Q M(|h|^{r})\right)^\frac{1}{r}$, so 
\begin{align*}
\squared{B}_\infty 	
	& \leq \left(\sum_{Q\in \mathcal{W}_1} \int_Q M(|h|^{r})(x)^\frac{p}{r}dx  \sup_{t\in\left( \frac{\ell(Q)}{10} , \ell_0\right)} \left( \sum_{\substack{R\in \mathcal{W}_1'(Q,C\rho t) \\ Q\in \SH(R)}}  \frac{ \ell(R)^{\sigma  +\frac{d}{r'}+\frac{d}{r}}}{t^{\sigma +  d}}  \right)^p   \right)^\frac1p.
\end{align*}
Using \rf{eqAscendingToGlory} we get
\begin{align*}
\squared{B}_\infty 	
	& \lesssim_\rho \left(\sum_{Q\in \mathcal{W}_1} \int_Q M(|h|^{r})(x)^\frac{p}{r}dx \right)^\frac1p\lesssim \norm{h}_{L^p}.
\end{align*}

On the other hand, for $Q\in \mathcal{W}_4$, we can just use the finite overlapping of symmetrized cubes to reduce it to the previous situation.
\end{proof}

\renewcommand{\abstractname}{Acknowledgements}
\begin{abstract}  {The author wants to thank Kari Astala and Eero Saksman for many conversations regarding the problems discussed in the paper, Winfred Sickel for providing references and encouragement, and also to the referees for an excellent job spotting misprints and mistakes. The author was supported by the Spanish State Research Agency, through the Severo Ochoa and Mar\'ia de Maeztu Program for Centers and Units of Excellence in R\&D (CEX2020-001084-M), by the Spanish government under the grant IJC2018-035373-I, by the ERC grant 834728-QUAMAP, by the Spanish Ministerio de Ciencia e Innovaci\'on, projects PID2020-114167GB-I00, PID2021-125021NA-I00, PID2021-123405NB-I00, and also by the catalan government projects 2017-SGR-395, 2021 SGR 00087 and 2021-SGR-00071. }
\end{abstract}

\bibliography{LlibresAPS}

\begin{thebibliography}{{Bou}00}

\bibitem[APS17]{AstalaPratsSaksman}
Kari Astala, Mart{\'\i} Prats, and Eero Saksman.
\newblock Global smoothness of quasiconformal mappings in the
  {T}riebel-{L}izorkin scale.
\newblock https://doi.org/10.48550/arXiv.1901.07844, February 2017.

\bibitem[BMS10]{BourdaudMoussaiSickelTL}
G\'{e}rard Bourdaud, Madani Moussai, and Winfried Sickel.
\newblock Composition operators on {L}izorkin-{T}riebel spaces.
\newblock {\em J. Funct. Anal.}, 259(5):1098--1128, 2010.

\bibitem[BMS14]{BourdaudMoussaiSickel}
G\'{e}rard {Bourdaud}, Madani {Moussai}, and Winfried {Sickel}.
\newblock Composition operators acting on {B}esov spaces on the real line.
\newblock {\em Annali di Matematica Pura ed Applicata}, 193(5):1519--1554,
  2014.

\bibitem[BMS20]{BourdaudMoussaiSickel2}
G\'{e}rard Bourdaud, Madani Moussai, and Winfried Sickel.
\newblock A necessary condition for composition in {B}esov spaces.
\newblock {\em Complex Var. Elliptic Equ.}, 65(1):22--39, 2020.

\bibitem[{Bou}00]{Bourdaud}
G\'{e}rard {Bourdaud}.
\newblock Changes of variable in {B}esov spaces ii.
\newblock {\em Forum Math.}, 12:545--563, 2000.

\bibitem[BS99]{BourdaudSickel}
G\'{e}rard {Bourdaud} and Winfried {Sickel}.
\newblock Changes of variable in {B}esov spaces.
\newblock {\em Mathematische Nachrichten}, 198(1):19--39, 1999.

\bibitem[CFR10]{ClopFaracoRuiz}
Albert Clop, Daniel Faraco, and Alberto Ruiz.
\newblock Stability of {C}alder{\'o}n's inverse conductivity problem in the
  plane for discontinuous conductivities.
\newblock {\em Inverse Probl. Imaging}, 4(1):49--91, 2010.

\bibitem[CHK15]{CampbellHenclKonopecky}
Daniel Campbell, Stanislav Hencl, and Franti{\v{s}}ek Konopeck\`y.
\newblock The weak inverse mapping theorem.
\newblock {\em Zeitschrift f\"ur Analysis und ihre Anwendungen},
  34(3):321--342, 2015.

\bibitem[Dah79]{Dahlberg}
Bj\"{o}rn E.~J. Dahlberg.
\newblock A note on {S}obolev spaces.
\newblock In {\em Harmonic analysis in {E}uclidean spaces ({P}roc. {S}ympos.
  {P}ure {M}ath., {W}illiams {C}oll., {W}illiamstown, {M}ass., 1978), {P}art
  1}, Proc. Sympos. Pure Math., XXXV, Part, pages 183--185. Amer. Math. Soc.,
  Providence, R.I., 1979.

\bibitem[GT01]{GilbargTrudinger}
David Gilbarg and Neil~S Trudinger.
\newblock {\em Elliptic Partial Differential Equations of Second Order}, volume
  224.
\newblock Springer Science \&amp; Business Media, 2001.

\bibitem[HK08]{HenclKoskela}
Stanislav Hencl and Pekka Koskela.
\newblock Mappings of finite distortion: composition operator.
\newblock In {\em Annales-Academiae Scientiarum Fennicae Mathematica},
  volume~33, page~65. Academia Scientiarum Fennica, 2008.

\bibitem[HK13]{HenclKoskelaComposition}
Stanislav Hencl and Pekka Koskela.
\newblock Composition of quasiconformal mappings and functions in
  {T}riebel-{L}izorkin spaces.
\newblock {\em Math. Nachr.}, 286(7):669--678, 2013.

\bibitem[Jon81]{Jones}
Peter~W. Jones.
\newblock Quasiconformal mappings and extendability of functions in {S}obolev
  spaces.
\newblock {\em Acta Math.}, 147(1):71--88, 1981.

\bibitem[KP92]{KrantzParks}
Steven~G Krantz and Harold~R Parks.
\newblock {\em A primer of real analytic functions}.
\newblock Birkh{\"a}user, 1992.

\bibitem[KRZ15]{KoskelaRajalaZhang}
Pekka Koskela, Tapio Rajala, and Yi~Zhang.
\newblock A geometric characterization of planar {S}obolev extension domains.
\newblock {\em arXiv: 1502.04139 [math.CA]}, 2015.

\bibitem[Mey78]{Meyers}
Norman~G Meyers.
\newblock Integral inequalities of {P}oincar{\'e} and {W}irtinger type.
\newblock {\em Arch. Ration. Mech. Anal.}, 68(2):113--120, 1978.

\bibitem[OP17]{OlivaPrats}
Marcos Oliva and Mart{\'\i} Prats.
\newblock Sharp bounds for composition with quasiconformal mappings in
  {S}obolev spaces.
\newblock {\em J. Math. Anal. Appl.}, 451(2):1026--1044, 2017.

\bibitem[Pra19]{PratsTL}
Mart{\'\i} Prats.
\newblock Measuring {T}riebel-{L}izorkin fractional smoothness on domains in
  terms of first-order differences.
\newblock {\em Journal of the London Mathematical Society}, 100(2):692--716,
  2019.

\bibitem[PS17]{PratsSaksman}
Mart{\'\i} Prats and Eero Saksman.
\newblock A {T}(1) theorem for fractional {S}obolev spaces on domains.
\newblock {\em J. Geom. Anal.}, 27(3):2490--2538, 2017.

\bibitem[PT15]{PratsTolsa}
Mart{\'\i} Prats and Xavier Tolsa.
\newblock A {T(P)} theorem for {S}obolev spaces on domains.
\newblock {\em J. Funct. Anal.}, 268(10):2946--2989, May 2015.

\bibitem[Rog06]{RogersExtension}
Luke~G Rogers.
\newblock Degree-independent sobolev extension on locally uniform domains.
\newblock {\em J. Funct. Anal.}, 235(2):619--665, 2006.

\bibitem[RS96]{RunstSickel}
Thomas Runst and Winfried Sickel.
\newblock {\em Sobolev spaces of fractional order, {N}emytskij operators, and
  nonlinear partial differential equations}, volume~3 of {\em De Gruyter series
  in nonlinear analysis and applications}.
\newblock Walter de Gruyter; Berlin; New York, 1996.

\bibitem[Ryc99]{Rychkov}
Vyacheslav~S Rychkov.
\newblock On restrictions and extensions of the {B}esov and
  {T}riebel--{L}izorkin spaces with respect to {L}ipschitz domains.
\newblock {\em J. London Math. Soc.}, 60(1):237--257, 1999.

\bibitem[Shv10]{Shvartsman}
Pavel Shvartsman.
\newblock On {S}obolev extension domains in $\mathbb{R}^n$.
\newblock {\em J. Funct. Anal.}, 258(7):2205--2245, 2010.

\bibitem[Ste70]{SteinPetit}
Elias~M. Stein.
\newblock {\em Singular integrals and differentiability properties of
  functions}, volume~30 of {\em Princeton Mathematical Series}.
\newblock Princeton University Press, 1970.

\bibitem[Tri06]{TriebelTheoryIII}
Hans Triebel.
\newblock {\em Theory of function spaces III}, volume 100 of {\em Monographs in
  Mathematics}.
\newblock Birkh{\"a}user, 2006.

\bibitem[Vod89]{Vodopyanov}
Sergei~Konstantinovich Vodopyanov.
\newblock Mappings of homogeneous groups and imbeddings of functional spaces.
\newblock {\em Siberian Math. Zh.}, 30:25--41, 1989.

\bibitem[Zie89]{Ziemer}
William~P Ziemer.
\newblock {\em Weakly differentiable functions: {S}obolev spaces and functions
  of bounded variation}, volume 120 of {\em Graduate Texts in Mathematics}.
\newblock Springer Berlin Heidelberg, 1989.

\end{thebibliography}
\end{document}